\documentclass{amsart}
\usepackage[usenames, dvipsnames]{color}
\usepackage{cite}
\usepackage{tikz-cd}
\usepackage{tikz}
\usepackage{mathrsfs}
\usepackage{graphicx}
\usepackage{amsmath}
\usepackage{amssymb}
\usepackage{calc}
\usepackage{accents}

\makeatletter
\newcommand*{\doublerightarrow}[2]{\mathrel{
  \settowidth{\@tempdima}{$\scriptstyle#1$}
  \settowidth{\@tempdimb}{$\scriptstyle#2$}
  \ifdim\@tempdimb>\@tempdima \@tempdima=\@tempdimb\fi
  \mathop{\vcenter{
    \offinterlineskip\ialign{\hbox to\dimexpr\@tempdima+1em{##}\cr
    \rightarrowfill\cr\noalign{\kern.5ex}
    \rightarrowfill\cr}}}\limits^{\!#1}_{\!#2}}}
\makeatother

\usepackage[linktoc=all]{hyperref}
\hypersetup{
  colorlinks,
  citecolor=black,
  filecolor=black,
  linkcolor=black,
  urlcolor=black
}
\usepackage{amsfonts}
\usepackage{amsthm}
\usepackage[shortlabels]{enumitem}

\numberwithin{equation}{section}
\newtheorem{theorem}{Theorem}[subsection]
\newtheorem{lemma}[theorem]{Lemma}
\newtheorem{corollary}[theorem]{Corollary}
\newtheorem{proposition}[theorem]{Proposition}

\theoremstyle{definition}
\newtheorem{definition}[theorem]{Definition}

\newtheorem{example}[theorem]{Example}

\newtheoremstyle{remark}
{2pt}
{5pt}
{}
{}
{\itshape}
{:}
{.5em}
{}
\newtheorem{remark}[theorem]{Remark}

\newcommand{\git}{/\!\!/}

\newcommand{\cond}[1]{&\text{if }#1}
\newcommand*{\SHOWPROOFS}{} 
\renewcommand{\tilde}{\widetilde}
\renewcommand{\hat}{\widehat}
\def\vdef{\ar@{}[d]|{\overset{\cdot\cdot}{||}}}
\title{Quasimap wall-crossing for GIT quotients}
\thanks{The author is supported by the Simons Collaboration Grant for Mathematicians and the Center
  of Mathematical Sciences and Applications, Harvard University.}
\address{Center of Mathematical Sciences and Applications, Harvard University,
  20 Garden Street,
  Cambridge, MA 02143, USA}
\email{yangzhou@cmsa.fas.harvard.edu}
\author{Yang Zhou}
\date{}

\begin{document}
\maketitle
\setcounter{tocdepth}{1}
\begin{abstract}
  In this paper, we prove a wall-crossing formula for $\epsilon$-stable
  quasimaps to GIT quotients conjectured by
  Ciocan-Fontanine and Kim, for all targets in all genera, including the orbifold
  case. We prove that stability conditions in adjacent chambers give equivalent invariants, provided
  that both chambers are stable. In the case of genus-zero quasimaps with one
  marked point, we compute the invariants in the left-most stable
  chamber in terms of the small $I$-function. Using this we prove that the
  quasimap $J$-functions are on the Lagrangian cone of the Gromov--Witten
  theory.
  The proofs are based on virtual
  localization on a master space, obtained via some universal construction on
  the moduli of weighted curves.
  The fixed-point loci are in one-to-one
  correspondence with the terms in the wall-crossing formula.
\end{abstract}
\tableofcontents
\section{Introduction}
In this section we briefly review the theory of quasimaps to GIT quotients and
state the main theorems. We refer the reader to
\cite{ciocan2014stable, cheong2015orbifold} for more details.

\subsection{Overview}
For a smooth projective variety $X$, the Gromov--Witten theory studies the
moduli space $\overline{\mathcal M}_{g,n}(X,d)$ of degree-$d$ stable maps into
$X$ from an $n$-marked nodal curve of genus $g$. This moduli space is a proper
Deligne--Mumford stack and carries a virtual fundamental class. The
Gromov--Witten invariants are defined by integrating certain cohomology classes
against this virtual fundamental class.
Let ${\mathcal M}_{g,n}(X,d)\subset \overline{\mathcal M}_{g,n}(X,d)$ be
the substack where the domain curve is smooth. In some cases, there are other
natural compactifications of ${\mathcal M}_{g,n}(X,d)$ that produce other
invariants, which are closely related to the Gromov--Witten invariants.

For a large class of GIT quotients of affine varieties,
the theory of $\epsilon$-stable quasimaps to them were developed in
\cite{ciocan2014stable, cheong2015orbifold}, unifying and generalizing many
previous constructions
\cite{mustacta2007intermediate,ciocan2010moduli,marian2011moduli,
  toda2011moduli}.
Examples of those targets include smooth
complete intersections in
toric Deligne--Mumford stacks, type-$A$ flag varieties and Nakajima quiver
varieties.
The theory depends on a stability parameter $\epsilon\in
\mathbb Q_{> 0}$.
The space of stability conditions has a wall-and-chamber
structure,  with only finitely many walls once the degree of the quasimap is
fixed. Roughly speaking, as $\epsilon$ gets larger, the domain curve is allowed
to ``bubble out'' more
rational components and the quasimap is closer to being a map. We write
$\epsilon = 0^+$ (resp. $\epsilon = \infty$) for the chamber where $\epsilon$
is sufficiently small (resp. large).
For $\epsilon
= \infty$, one obtains the Gromov--Witten
theory (i.e.\ stable maps); for $\epsilon = 0^+$, one obtains the theory of
stable quotients in the case of Grassmannians \cite{marian2011moduli}.
The goal of this paper is to compare the invariants
from different chambers.
\subsubsection{Motivations and applications}
A main motivation for the quasimap theory is the Mirror Conjecture. At
least for the quintic threefold, the generating function of $0^+$-stable
quasimap invariants is precisely the holomorphic limit of the $B$-model partition function in
\cite{BERSHADSKY1993279} (c.f.\ \cite[\S1.5]{ciocan2020quasimap}).
Using quasimap wall-crossing, the holomorphic anomaly equations for the
formal quintic and some other examples have been
proven in \cite{lho2018stable}.

A second motivation is that the $0^+$-stable quasimap invariants can be easier
to compute in some cases, due to simpler domain curves. It has been used to prove the genus-$1$ and genus-$2$ mirror
theorems for the quintic threefold \cite{kim2018mirror, guo2017mirror} (c.f.\
\cite{cooper2014mirror}). The situation is especially nice when the genus is
$0$. In that case we consider the invariants with one marked point.
For $\epsilon$ in the left-most stable chamber, the
domain curves are irreducible, and thus the
moduli spaces are much simpler. Those invariants can be easily computed in terms
of the $I$-function. 
We hope this will be helpful for
computing the genus-$0$ invariants when the standard localization technique
fails. In the ongoing project \cite{orbifold-GW}, we compute all the genus-$0$ Gromov--Witten invariants for some 
$3$-dimensional Calabi--Yau complete intersections in weighted projective spaces, to which
the Quantum Lefschetz Theorem does not apply.

A third motivation is the Landau--Ginzburg/Calabi--Yau
(LG/CY) correspondence. For the
quintic threefold, the LG side (Fan--Jarvis--Ruan--Witten (FJRW) theory) and the CY
side (Gromov--Witten theory)
can be realized as two phases of the same Gauged Linear Sigma Model
(GLSM) \cite{witten1993phases, fan2017mathematical}. The theory of quasimaps naturally
interpolates between the two phases, where the FJRW theory corresponds to
$\epsilon \to -\infty$. It is expected that the two phases are more closely
related near $\epsilon = 0$. This has been worked out in genus $0$  and genus
$1$ \cite{guo2016genus, ross2017wall}.

Finally, quasimap invariants naturally appear in the GLSM for
more general targets \cite{fan2017mathematical}. Using quasimaps
instead of stable maps gives more flexibility in the compactification of the
moduli spaces.

\subsubsection{Relation to others' work}
An explicit wall-crossing formula in all genera
was conjectured in \cite{ciocan2017higher,
ciocan2020quasimap} for non-orbifold targets. It has been proven in the
following cases.
\begin{enumerate}
\item
  It has been proven for targets with a good torus action
  \cite{ciocan2014wall,
    cheong2015orbifold, ciocan2017higher}. This also
  includes the twisted theories for those targets. In particular, it includes the genus-$0$
  theory of the vanishing locus of a section of a convex equivariant vector bundle.
\item
  It has been proven in all genera for complete intersections in projective spaces
  \cite{ciocan2020quasimap, clader2017higher}.
\item
  During the preparation of this paper, the author learned that Jun Wang independently proved
  \cite{wang2019mirror} the genus-$0$
  wall-crossing formula for hypersurfaces in toric
  stacks for which convexity can fail,
  and used that to prove a mirror theorem.\footnote{The wall-crossing formula of
    Jun Wang was only proved for
    those targets. But his method potentially generalizes to other targets. His
    method is different from the method in this paper.}
\end{enumerate}
Motivated by the LG/CY correspondence, there are also
similar wall-crossing results on the LG side \cite{clader2017higherLG,
  zhou2020higher}. The method in this paper also works in the  setting of K-theoretic
Gromov--Witten invariants. We will prove the K-theoretic wall-crossing formula
and its corollaries in \cite{k-theory-wall-crossing}. 

\subsubsection{Summary of the results}
In this paper we prove the wall-crossing formula for all targets and in all
genera (Theorem \ref{thm:Chow-version}). In genus $0$, we will also compute in
terms of the $I$-function the
invariants with one marked point, for $\epsilon$ in the left-most stable chamber
(Theorem~\ref{thm:genus-0-special-case}). Combining those, we prove that the
quasimap $J$-function lies on the Lagrangian cone of the Gromov--Witten
theory (Theorem~\ref{thm:big-J-general}), recovering and generalizing the
classical mirror theorem \cite{giventalequiv, lian1997mirror}.

\subsection{The target space}
Consider a ``stacky'' GIT quotient
\[
  X = [W^{\mathrm{ss}}(\theta)/G]
\]
where
\begin{itemize}
\item
  $W$ is an affine variety with at worst local complete intersection singularities,
\item
  $G$ is a reductive group acting on $W$ (from the right),
\item
  $\theta$ is character of $G$ such that the $\theta$-stable locus
  $W^{\mathrm{s}}(\theta)$ is smooth, nonempty, and coincides with the
  $\theta$-semistable locus $W^{\mathrm{ss}}(\theta)$.
\end{itemize}
Replacing $\theta$ by a multiple of itself does not change the quotient.
\footnote{It changes the $\epsilon$-stability condition to be introduced below in a simple
  way: $\epsilon$-stable for $(W,G,k \theta)$ is equivariant to
  $(k\epsilon)$-stable for $(W,G,\theta)$, $k\in \mathbb Z_{>0}$.
  See \cite[\textsection2.1]{ciocan-fontanine2016}.}
Hence, without loss of generality we further
assume that the ring
\[
  \bigoplus_{m=0}^\infty H^0(W,\mathcal O_W(m\theta))^G
\]
of invariants
is generated by $H^0(W,\mathcal O_W(\theta))^G$ as an $H^0(W,\mathcal
O_W)^G$-algebra. Using a basis for $H^0(W,\mathcal O_W(\theta))^G$ of size $N+1$
we form the map
\begin{equation}
  \label{eq:embedding-of-quotient}
  {[W/G]} \longrightarrow  [\mathbb C^{N+1}/\mathbb C^*].
\end{equation}
The preimage of the origin in $\mathbb C^{N+1}$ is precisely
$[W^{{us}}(\theta)/G]$, where 
$W^{{us}}(\theta)$ is the 
the $\theta$-unstable locus.
Taking the stable locus, \eqref{eq:embedding-of-quotient} induces a closed embedding
\[
  \underline{X} \longrightarrow \mathbb P^N \times W \git_{0} G,
\]
where $\underline{X}$ is the coarse moduli of $X$, and $W \git_{0} G =
\operatorname{Spec} H^0(W,\mathcal O_W)^G$ is the affine quotient. Thus $X$ is a
smooth proper Deligne--Mumford stack over $W\git_0 G$.

The $G$-equivariant line bundle $\mathcal
O_W(\theta)$ descents to
$L_{\theta}\in \mathrm{Pic}([W/G])$. The restriction of $L_{\theta}$ to $X$ is
equal to the pullback of $\mathcal O_{\mathbb P^N}(1)$. In
particular, for any closed point $x\in X$, the stabilizer $G_x$ of $x$
acts trivially on $L_{\theta}|_{x}$. We refer to $L_{\theta}$ as the
polarization on $[W/G]$ or on $X$.
\subsection{Twisted curves with trivialized gerbe markings}

Throughout this paper, a curve without further specification will be a
twisted curve with balanced nodes and trivialized gerbe markings (see e.g.\
\cite[\textsection4]{abramovich2008gromov}).
Thus a marking on a family of curves
$\pi:\mathcal C\to S$ is a closed substack $\Sigma\subset \mathcal C$ contained
in the relative smooth locus, together with a
section $S\to \Sigma$ of $\pi|_\Sigma$, such that $\Sigma \to S$ is a gerbe
banded by some $\mu_{r}$.
Note that this is slightly different from
the convention in \cite[\textsection2.5.2]{cheong2015orbifold}, where the section is
absent. 
See Section~\ref{sec:comparison-gerbe-markings} and
\cite[\textsection6.1.3]{abramovich2008gromov} for the comparison.

\subsection{The moduli of $\epsilon$-stable quasimaps}
A quasimap
\[
  (C,x_1 ,\ldots, x_n, u)
\]
to $X$ consists of a curve $(C,x_1 ,\ldots, x_n)$ with $n$ marked
points (or markings for short), together with a representable morphism
\[
  u: C \longrightarrow [W/G]
\]
such that $u^{-1}[W^{\mathrm{us}}/G]\subset C$ is finite and disjoint from all the nodes
and marked points. Rigorously speaking, the notion of quasimaps depends not only
on $X$ but also on the stack $[W/G]$, etc. But we will just say quasimaps to
$X$, following \cite{ciocan2014stable}. A point $y\in C$ is called a base point
if $y\in u^{-1}[W^{\mathrm{us}}/G]$. The length of a base point $y$, denoted by
$\ell(y)$, is defined to be the length at $y$ of the subscheme
$u^{-1}[W^{\mathrm{us}}/G]$.

The curve class of a quasimap $u$ is the group homomorphism
\[
  \beta: \mathrm{Pic}([W/G]) \longrightarrow \mathbb Q
\]
defined by $\beta(L) = \deg(u^*L)$, for any $L\in \mathrm{Pic}([W/G])$. A group
homomorphism $\beta: \mathrm{Pic}([W/G]) \to \mathbb Q$ is called an effective
curve class if it is the curve class of some $u$. The set of effective curve
classes is denoted by $\mathrm{Eff}(W,G,\theta)$. The degree of $u$, or the
degree of $\beta$, is defined to be
\[
  \deg(\beta) := \beta(L_{\theta}) \in \mathbb Z.
\]
We have suppressed $\theta$ from the notation $\deg(\beta)$, since $\theta$ will be fixed
throughout this paper. Note that $\deg(\beta)$ is indeed an integer since we
have assumed that $L_\theta$ has no monodromy at orbifold markings, by raising
$\theta$ to a multiple of itself.

We adopt the following short-hand notation, which will be convenient in complicated
formulas:
\begin{itemize}
\item $\beta\geq 0$ means $\beta\in \mathrm{Eff}(W,G,\theta)$,
\item $\beta>0$ means $\beta\geq 0$ and $\beta\neq 0$.
\end{itemize}
For $\epsilon\in \mathbb Q_{>0} \cup \{0^+,\infty\}$, a quasimap $(C,x_1
,\ldots, x_n, u)$ is called $\epsilon$-stable if
\begin{enumerate}
\item
  for each base point $y\in C$, we have
  $\ell(y)<1/\epsilon$,
\item
  the $\mathbb Q$-line bundle $(u^*L_\theta)^{\otimes \epsilon}\otimes \omega_{C,\log}$
  is positive, in the sense that it has a positive degree on each irreducible
  component of $C$, where $\omega_{C,\log} := \omega_{C}(\sum_{i} x_i)$ is the
  log dualizing sheaf.
\end{enumerate}

It is straightforward to define families of quasimaps over any scheme $S$. We
denote by $Q^{\epsilon}_{g,n}(X,\beta)$ the moduli of genus-$g$
$\epsilon$-stable quasimaps to $X$ of curve class $\beta$ with $n$ marked
points. It is a Deligne--Mumford stack, proper over $W\git_0 G$, with a perfect
obstruction theory\cite[Theorem~2.7]{cheong2015orbifold}, which induces a
virtual fundamental class $[Q^{\epsilon}_{g,n}(X,\beta)]^{\mathrm{vir}}$.
Throughout this paper, whenever the moduli space is empty, the virtual
fundamental class is understood as $0$.

The space $\mathbb Q_{\geq 0}\cup \{0^+,\infty\}$ of stability conditions is
divided into chambers by the walls $\{1/d\mid d\in \mathbb Z_{>0}\}$. When
$\epsilon$ varies within a chamber, the moduli space and its virtual fundamental
class do not change. When a wall is crossed, the change of the invariants will
be given by a wall-crossing formula (Theorem~\ref{thm:Chow-version}).

\subsection{The state space and quasimap invariants}
\label{sec:state-space-and-invariants}
Let
\[
  I_{\mu} X = \coprod_r I_{\mu_r} X
\]
be the cyclotomic inertia stack of $X$ (c.f.\
\cite[\textsection3.1]{abramovich2008gromov}). Recall that for any scheme $T$,
the groupoid of $T$-points of $I_{\mu_r} X$ consists of triples $(\pi: \Sigma\to
T, f, \sigma)$, where $\pi: \Sigma \to T$ is a gerbe banded by $\mu_r$, $f:
\Sigma \to X$ is a representable morphism and $\sigma: T\to \Sigma$ is a section
of $\pi$ (c.f.\ \cite[\textsection3.2]{abramovich2008gromov}). Thus the
restriction of the universal map to the markings gives rise to evaluation maps
\[
  \mathrm{ev}_i: Q^{\epsilon}_{g,n}(X,\beta) \longrightarrow I_\mu X.
\]
Although we will prove the wall-crossing formula in the Chow group, it is often
convenient to consider the cohomology when forming generating series. 
We define the state space to be the Chen--Ruan cohomology group
\[
  H_{\mathrm{CR}}^*(X,\mathbb Q) := H^*(I_{\mu}X,\mathbb Q).
\]
Let $\mathbf r$ be the locally constant function on $I_{\mu}X$ that takes value
$r$ on $I_{\mu_r}X$.
Let $\iota$ be the canonical involution on $I_\mu(X)$ defined by inverting the
banding. In this paper we use the pairing on $H_{\mathrm{CR}}^*(X,\mathbb Q)$ defined by
\begin{equation}
  \label{eq:pairing}
  (\alpha_1, \alpha_2) := \int_{I_\mu X} \frac{\alpha_1 \cup
    \iota^*\alpha_2}{\mathbf r^2}.
\end{equation}
For any Chen--Ruan cohomology classes
  $\gamma_1 ,\ldots, \gamma_n\in
  H_{\mathrm{CR}}^*(X,\mathbb Q)$
and non-negative integers $k_1 ,\ldots, k_n$, we define the quasimap invariants
with descendants
\begin{equation}
  \label{eq:correlators}
  \langle {\gamma_1\psi^{k_1} ,\ldots, \gamma_n\psi^{k_n}}\rangle_{g,n,\beta}^{X, \epsilon}
  = p_*\Big(
  \prod_{i=1}^n \mathrm{ev}_i^*(\gamma_i)\psi_i^{k_i} \cap
  [Q^{\epsilon}_{g,n}(X,\beta)]^{\mathrm{vir}}
  \Big) \in H^{\mathrm{BM}}_*(W\git_{0}G, \mathbb Q),
\end{equation}
where $p$ is the forgetful morphism $p:Q^{\epsilon}_{g,n}(X,\beta) \to
W\git_0G$. When $W\git_{0} G$ is a point, we get a rational number.

For simplicity of the notation, from now on we assume that $W\git_{0} G$ is a
point. The theorems hold true in general and the proofs are verbatim.

\subsection{Comparison with the state space in \cite{cheong2015orbifold}}
\label{sec:comparison-gerbe-markings}
In \cite{cheong2015orbifold}, the state space was defined as $H^*(\bar I_{\mu}
X, \mathbb Q)$, where $\bar I_{\mu} X$ is the rigidified cyclotomic inertia
stack of $X$. We need to specify an isomorphism between $H^*(\bar I_{\mu}
X, \mathbb Q)$ and $H^*(I_{\mu} X, \mathbb Q)$ in order to compare our results
to those in \cite{cheong2015orbifold}.
Let
\[
  \tau: I_{\mu} X \longrightarrow  \bar I_{\mu}X
\]
be the natural projection. Then the pushforward of our $I$-function is the one
defined in \cite{cheong2015orbifold}. In other words, when comparing the
$I$-functions we identify $\alpha \in H^*(\bar I_{\mu_r} X, \mathbb Q)$ with $r
\cdot \tau^*\alpha\in H^*(I_{\mu_r} X, \mathbb Q)$. The advantage of this
identification is that the correlators \eqref{eq:correlators} will be the same
as those in \cite{cheong2015orbifold}. The pairing in
\cite[\textsection3.1]{cheong2015orbifold} also corresponds to the pairing
\eqref{eq:pairing}. Thus, Theorem~\ref{thm:Chow-version} and
Theorem~\ref{thm:big-J-general} do not change when interpreted using the
definitions in \cite{cheong2015orbifold}. While for
Theorem~\ref{thm:genus-0-special-case}, one needs to replace the $\mathbf r^2$
by $\mathbf r$.

\subsection{The $\psi$-classes at orbifold markings}
\label{sec:orbifold-psi-classes}
We will often use the orbifold $\psi$-classes, which we will denote by
$\tilde\psi$. The orbifold $\psi$-classes are the first Chern classes of the
orbifold relative cotangent bundle along the orbifold markings. This makes sense
because the markings have canonical sections, using which we can pullback those
relative cotangent bundles to the moduli space. We denote the $\psi$-classes of
the associated coarse curves by $\psi$. The two $\psi$-classes are related by
$\psi = r\tilde\psi$, if the marking is a gerbe banded by $\mu_r$.
\subsection{Virtual localization}
A key tool is the $\mathbb C^*$-localization of virtual cycles
\cite{graber1999localization, chang2017torus}. We will use the version in
\cite{chang2017torus}, which only requires that the virtual normal bundle has a
global resolution. This can be easily verified.

We will use $z$ to denote the equivariant parameter, i.e.\  the generator of
$A^*_{\mathbb C^*}(\mathrm{pt})$. We denote by $e_{\mathbb C^*}$ the equivariant
Euler class. We make the convention that $e_{\mathbb C^*}(\mathbb
C_{\mathrm{std}}) = -z$, where $\mathbb C_{\mathrm{std}}$ is the standard
representation of $\mathbb C^*$ on $\mathbb C$.

When we have a $\mathbb C^*$-action on some space,
a fixed-point component means a union of several connected components of the
fixed-point locus.
\subsection{More conventions and notation}
We make more conventions before stating the main theorems.

We work over $\mathbb C$ through out the paper. Hence all Deligne--Mumford
stacks are tame and cyclotomic inertia stacks are the same as inertia stacks.
All schemes/stacks are locally noetherian.
Chow groups and (co)homology groups are always with $\mathbb Q$-coefficients.
A curve is always a twisted curve with balanced nodes unless otherwise specified. We
will omit the usual ``$\mathrm{tw}$'' from our notation, which means ``twisted''.

We will use various truncation operations on formal series in $z$.
Let $f(z)=\sum_{n}a_nz^n$ be a formal Laurent series. Then $[f(z)]_{k}:=a_k$,
$[f(z)]_{z^{\geq k}}:= \sum_{n\geq k}a_n z^{k}$, and
$[f(z)]_{+}:=[f(z)]_{z^{\geq 0}}$, etc.

\subsection{The $I$-function}
\label{sec:I-function}
To state the wall-crossing theorem, we need to introduce the $I$-function,
which is defined via localization on the graph space.
We only need a special case of the definition in \cite{cheong2015orbifold}.
Given an effective curve class $\beta$, choose $\epsilon\in \mathbb Q_{>0}$ and
$A\in \mathbb Z_{>0}$ such that
$1/A < \epsilon < 1/ \deg(\beta)$. We 
view $\mathbb P^1$ as the GIT quotient $\mathbb C^2\git\mathbb C^*$ 
polarized by the character of $\mathbb C^*$ given by $t\mapsto t^A, t \in \mathbb C^*$. 
Define the graph space to be
\[
  QG_{0,1}(X,\beta) := Q^{\epsilon}_{0,1}(X\times \mathbb P^1, \beta\times
  [\mathbb P^1]).
\]
Thus it parametrizes quasimaps to $X\times \mathbb P^1$ such that the quasimap
to $\mathbb P^1$ is an actual map, and the domain curve has a unique rational
tail, whose coarse moduli is mapped isomorphically onto $\mathbb P^1$.
Thus it is
independent of the choice of $A$ and $\epsilon$.
We denote the unique marking by $x_{\star}$.
Consider the $\mathbb C^*$-action on $\mathbb P^1$ given by
\begin{equation}
  \label{eq:C*-action-on-P1}
  t[\zeta_0,\zeta_1] = [t\zeta_0,\zeta_1], \quad t\in \mathbb C^*.
\end{equation}
Thus the tangent space at $\infty \in \mathbb P^1$ is isomorphic to the standard
$\mathbb C^*$-representation. This induces an action on $QG_{0,1}(X,\beta)$.
There is a distinguished fixed-point component $F^{0,\beta}_{\star,0}$
consisting of $\mathbb C^*$-fixed quasimaps such that the marking $x_{\star}$ is
mapped to $\infty$, while the entire class $\beta$ lies over $0\in \mathbb P^1$,
which means that $0$ is a base point of length $\deg(\beta)$. By
\cite{graber1999localization}, $F^{0,\beta}_{\star,0}$ has a virtual fundamental
class $[F^{0,\beta}_{\star,0}]^{\mathrm{vir}}$ and a virtual normal bundle
$N^{\mathrm{vir}}_{F^{0,\beta}_{\star,0}/ QG_{0,1}(X,\beta)}$.

We now introduce some generating series. Let $\Lambda := \{\sum_{\beta \geq
0}a_{\beta}q^\beta \mid a_{\beta}\in \mathbb Q\} $ be the Novikov ring.
Let
\begin{equation}
  \label{map:star-check}
  \check{\mathrm{ev}}_{\star}:QG_{0,1}(X,\beta) \longrightarrow
  I_{\mu}X
\end{equation}
be the evaluation map at $x_{\star}$ composed with the involution $\iota$ on
$I_{\mu}X$ (c.f.\cite{abramovich2008lectures}).
\begin{definition} We define
  \begin{equation}
    \label{eq:small-I}
    I(q,z) = 1 + \sum_{\beta > 0}
    (-z \mathbf r^2 ) q^{\beta} \check{\mathrm{ev}}_*
    \Big(
    \frac{[F^{0,\beta}_{\star,0}]^{\mathrm{vir}}}{e_{\mathbb
        C^*}(N^{\mathrm{vir}}_{F^{0,\beta}_{\star,0}/ QG_{0,1}(X,\beta)})}
    \Big) \in A^*(  I_{\mu}X) [z^{\pm 1}] \widehat{\otimes}_{\mathbb Q} \Lambda.
\end{equation}
\end{definition}
\begin{remark}
  The ``$\widehat{\phantom{\otimes}}$'' means the completion in the $q$-adic
  topology.
  Thus the coefficient of each $q^\beta$ is an element in $A^{*}(I_{\mu}X)[z^{\pm 1}]$,
  carrying the information of quasimaps from a parameterized (orbifold) $\mathbb
  P^1$ to $[W/G]$ of curve class $\beta$. The factor $(-z)$ is
  introduced to cancel the equivariant Euler class of the deformation space of
  the unique marking. Recall the $\mathbf r$ from
  Section~\ref{sec:state-space-and-invariants}.
  One copy of $\mathbf r$ is introduced in the formula above because
  our markings are trivialized gerbes
  (c.f.\ Section~\ref{sec:state-space-and-invariants}), the other is
  introduced because there is already such a factor in \cite{cheong2015orbifold} .
\end{remark}

We define
\[
  \mu_{\beta}(z) \in A^*( I_{\mu}X)[z]
\]
to be the coefficient of $q^\beta$ in $[zI(q,z)-z]_+$, where
$[\cdot]_+$ means the truncation by taking only nonnegative powers of $z$.

\subsection{The higher-genus wall-crossing formula}
Let $\epsilon_0 =1/d_0$ be a wall. Let
$\epsilon_-<\epsilon_+$ be stability conditions in the two
adjacent chambers separated by $\epsilon_0$. Fix the curve class
$\beta$ of degree $d\geq d_0$, the genus $g$ and the number of markings $n$,
such that $2g -2 + n + d\epsilon_0 >0$.
We would like to
compare the two virtual fundamental classes
$[Q^{\epsilon_+}_{g,n}(X,\beta)]^{\operatorname{vir}}$ and
$[Q^{\epsilon_-}_{g,n}(X,\beta)]^{\operatorname{vir}}$. In general we do not
have a natural contraction map $Q^{\epsilon_+}_{g,n}(X,\beta) \to
Q^{\epsilon_-}_{g,n}(X,\beta)$. Instead we have
\begin{equation*}
\begin{tikzcd}
    Q^{\epsilon_+}_{g,n}(X,\beta) \arrow[d,"i"]&  Q^{\epsilon_-}_{g,n}(X,\beta) \arrow[d,"i"]\\
    Q^{\epsilon_+}_{g,n}(\mathbb P^N, d) \arrow[r,"c"]&  Q^{\epsilon_-}_{g,n}(\mathbb P^N, d)
\end{tikzcd},
\end{equation*}
where the $i$'s are defined by composing the quasimap with
\eqref{eq:embedding-of-quotient} and forgetting the orbifold structure of the
domain curves, and $c$ is defined by contracting all the
degree-$d_0$ rational tails to length-$d_0$ base points.
See \cite[\textsection3.2.2]{ciocan2014wall} for
the precise definition. We will pushforward the two virtual fundamental classes
to $Q^{\epsilon_-}_{g,n}(\mathbb P^N, d)$ and compare them there.

Let
\begin{equation}
  \label{map:marking-to-base-points}
  b_k: Q^{\epsilon_-}_{g,n+k}(\mathbb P^N,d-kd_0) \longrightarrow
  Q^{\epsilon_-}_{g,n}(\mathbb P^N, d)
\end{equation}
be the map that replaces the last $k$ markings by base points of length $d_0$
\cite[\textsection 3.2.3]{ciocan2014wall}.
\begin{theorem}[Theorem~\ref{thm:main-case}]
  \label{thm:Chow-version}
  Assuming that $2g-2 +n+\epsilon_0\deg(\beta)>0$, we have
  \begin{equation}
    \label{eq:Chow-version-literature}
    \begin{aligned}
      & i_*[Q^{\epsilon_-}_{g,n}(X,\beta)]^{\mathrm{vir}} -
      c_*i_*[Q^{\epsilon_+}_{g,n}(X,\beta)]^{\mathrm{vir}} \\
      & = \sum_{k\geq 1} \sum_{\vec\beta} \frac{1}{k!} b_{k*} c_{*} i_*\Big(
      \prod_{a=1}^k \mathrm{ev}_{n+a}^*\mu_{\beta_a}(z)|_{z=-\psi_{n+a}} \cap
      [Q^{\epsilon_+}_{g,n+k}(X,\beta^\prime)]^{\mathrm{vir}}
      \Big)
    \end{aligned}
  \end{equation}
  in $A_*(Q_{g,n}^{\epsilon_-}(\mathbb P^N,d))$, where
  $\vec\beta$ runs through all the $(k+1)$-tuples of effective curve classes
  \[
    \vec\beta = (\beta^\prime,\beta_1,\ldots, \beta_k)
  \]
  such that $\beta = \beta^\prime + \beta_1 + \cdots + \beta_k$ and
  $\deg(\beta_i) = d_0$ for all $i=1 ,\ldots, k$.
\end{theorem}
The non-orbifold case is \cite[Conjecture 1.1]{ciocan2020quasimap}.
\begin{remark}
First, the equation remains true after capping both sides with any
$\prod_{i=1}^n \psi_i^{a_i}\mathrm{ev}_i^*(\gamma_i)$,
for $\gamma_i\in A^*(I_\mu X)$. This is because all the maps involved in the
theorem commute with the evaluation maps at the first $n$ markings, and induce
isomorphisms of the relative cotangent spaces at the first $n$ markings. Hence
the claim follows from the projection formula. Second, one can take the cycle
map from the Chow groups to the homology groups and get a (co)homological
version of the wall-crossing formula.
\end{remark}

We now derived a numerical version of the wall-crossing formula by pushing
forward the relation \eqref{eq:Chow-version-literature} to a point (assuming
$W\git_{0}G$ is a point).

Let $\mathbf t(z)\in H^*_{\mathrm{CR}}(X, \mathbb Q)[\![z]\!]$ be generic
and we define
\[
  F^{\epsilon}_{g}(\mathbf t(z)) =
  \sum_{n=0}^\infty\sum_{\beta\geq 0} \frac{q^\beta}{n!} \langle \mathbf t(\psi)
  ,\ldots, \mathbf t(\psi) \rangle^{X, \epsilon}_{g,n,\beta},
\]
where the unstable terms are interpreted as zero.
These are viewed as formal functions near the origin of $H^*_{\mathrm{CR}}(X,
\mathbb Q)$.
Define
\[
  \mu^{\geq \epsilon}(q,z) =
  \sum_{\epsilon\leq 1/\deg(\beta)<\infty} \mu_\beta(z)q^{\beta}.
\]
By repeatedly applying Theorem~\ref{thm:Chow-version} to cross the all the walls
contained in
$[\epsilon, \infty)$, we obtain
\begin{corollary}
  \label{cor:higher-genus-numerical}
  For $g\geq 1$ and any $\epsilon$, we have
  \begin{equation*}
    F^{\epsilon}_{g}(\mathbf t(z)) =
    F^{\infty}_{g}(\mathbf t(z)+\mu^{\geq \epsilon}(q,-z)).
\end{equation*}
    For $g=0$, the same equation holds true modulo the constant and linear terms
    in $\mathbf t$.
\end{corollary}

\subsection{The genus-$0$ wall-crossing formula}
Let $\epsilon_-<\epsilon_0 = \frac{1}{d_0} < \epsilon_+$ be as before.
We consider the case $g=0,n=1$ and $\deg(\beta)=d_0$.
In this case $Q^{\epsilon_-}_{0,1}(X,\beta)$ is
empty and Theorem~\ref{thm:Chow-version} does not apply. Instead, the
$\epsilon_+$-stable quasimap invariants can be
computed directly and we will prove
\begin{theorem}[Lemma~\ref{lem:genus-zero-special-case}]
  For $\epsilon\in (\frac{1}{\deg(\beta)}, \frac{1}{\deg(\beta)-1})$, we have
  \label{thm:genus-0-special-case}
  \begin{equation*}
    \mathbf r^2 \check{\mathrm{ev}}_*(
    \frac{[Q^{\epsilon}_{0,1}(X,\beta)]^{\mathrm{vir}}}{z(z-\psi_1)}
    )
  = [I(q,z)]_{z^{\leq -2},q^\beta}.
\end{equation*}
\end{theorem}

We now combine Theorem~\ref{cor:higher-genus-numerical} and
Theorem~\ref{thm:genus-0-special-case}.
First define $I_{\beta}(z)$ by
\[
  I(q,z) = 1 + \sum_{\beta>0} I_{\beta}(z) q^\beta.
\]
Thus $\mu_\beta(z) = [zI_{\beta}(z)]_+$ for $\beta>0$.
Let $T^p$ be a basis for $H^*_{\mathrm{CR}}(X,\mathbb Q)$ and $\{T_p\}$ be its
dual basis under the pairing \eqref{eq:pairing}.
For generic $t\in H^*_{\mathrm{CR}}(X, \mathbb Q)$, we define the big $J$-function
\begin{equation*}
  J^\epsilon(t,q,z) = 1 + \frac{t}{z} +
  \!\!\! \sum_{0<\deg(\beta)\leq 1/\epsilon} \!\!\!\!
  I_{\beta}(z)q^\beta + \sum_{\beta\geq 0,k\geq 0} \frac{q^\beta}{k!}\sum_p T_p
  \langle {
    \frac{T^p}{z(z-\psi)}, t ,\ldots, t
  }\rangle_{0,1+k,\beta}^{X,\epsilon},
\end{equation*}
where the unstable terms are interpreted as zero.
More generally, for generic $\mathbf t(z) \in H^*_{\mathrm{CR}}(X, \mathbb Q) [\![z]\!]$, we define
\begin{equation*}
  \begin{aligned}
    J^\epsilon(\mathbf t(z),q,z) = & 1 + \frac{\mathbf t(-z)}{z} +
                                     \sum_{0<\deg(\beta)\leq 1/\epsilon}
                                     I_{\beta}(z)q^\beta \\
                                   & + \sum_{\beta\geq 0,k\geq 0} \frac{q^\beta}{k!}\sum_p T_p
                                     \langle {
                                     \frac{T^p}{z(z-\psi)}, \mathbf t(\psi) ,\ldots, \mathbf t(\psi)}
                                     \rangle_{0,1+k,\beta}^{X,\epsilon}.
  \end{aligned}
\end{equation*}
We will prove the following in Section~\ref{sec:proof-genus-0}.
\begin{theorem}
  \label{thm:big-J-general}
  For any $\epsilon$,
  \begin{equation}
    \label{eq:big-J-general}
    J^{+\infty}(\mathbf t(z)+\mu^{\geq\epsilon}(q,-z),q,z) = J^{\epsilon}(\mathbf t(z),q, z).
  \end{equation}
\end{theorem}
Setting $\mathbf t(z) = t\in H^*_{\mathrm{CR}}(X, \mathbb Q)$, we see that $J^{\epsilon}(t,q,z)$
is on the Lagrangian cone of the Gromov--Witten theory of $X$ (c.f. \cite{coates2009computing}).

\subsection{Idea of the proof}

The main idea is to construct a
``master space'' with a $\mathbb C^*$-action such that the
fixed-point components are closely related to the moduli spaces that appear in the
wall-crossing formula \eqref{eq:Chow-version-literature}.

Geometrically, when a wall $\epsilon_0 = 1/d_0$ is crossed from the right to the
left, degree-$d_0$ rational tails (Definition~\ref{def:tail-and-bridge}) are
replaced by length-$d_0$ base points. We consider quasimaps that may have both
degree-$d_0$ rational tails and length-$d_0$ base points. A degree-$d_0$
rational tail containing a length-$d_0$ base point may have infinitely many
automorphisms. Hence those quasimaps form an Artin stack. Motivated by the work
of Thaddeus \cite{thaddeus1996geometric}, we will form a $\mathbb P^1$-bundle
over that Artin stack such that the automorphism groups of the rational tails
act nontrivially on the fibers. Then we hope that some open substack of that
$\mathbb P^1$-bundle will be a proper Deligne--Mumford stack, which will be our
master space. Since when the curve has more than one rational tails, the
automorphism group may be larger than $\mathbb C^*$, some virtual blowups have
to be introduced at the beginning of this construction.

Let us ignore the virtual blowups for the moment and describe the $\mathbb
P^1$-bundle and the $\mathbb C^*$-action. Roughly speaking, the $\mathbb
P^1$-bundle is constructed by adding the choice of a $v\in \Theta^{\vee}_1
\otimes \cdots \otimes \Theta^{\vee}_k\cup \{\infty\}$ to the moduli problem,
where $\Theta_i\cong \mathbb C$ are the infinitesimal smoothings of the nodes
that lie on the degree-$d_0$ rational tails. We will call this $v$ the
calibration (Section~\ref{sec:cali-bundle}). The $\mathbb C^*$-action on the
master space is defined by scaling $v$. We then take an open substack by
requiring that the quasimap is $\epsilon_+$-stable when $v=0$, and
$\epsilon_-$-stable when $v=\infty$.

We describe some typical orbits and their closures, in order to illustrate how
the $\mathbb C^*$-action relates rational tails to base points. First consider
an $\epsilon_+$-stable quasimap with only one degree-$d_0$ rational tail and
$v\neq 0,\infty$. This is not a fixed point and we study the closure of its
orbit. As $t \to 0$, we have the $\epsilon_+$-stable quasimap with $v=0$; as
$t\to \infty$, we are forced to ``reparametrize'' the degree-$d_0$ rational tail
to keep $v$ away from $\infty$, since we have imposed $\epsilon_-$-stability
when $v=\infty$. This will create a length-$d_0$ base point on that rational
tail. Thus the limit lands in a fixed-point component that will contribute to
the correction terms in the wall-crossing formula. Note that $\mathbb C^*$ acts
nontrivially on that rational tail, fixing the unique node and the unique base
point. That is how the small $I$-function, defined via localization on graph
spaces, comes into the wall-crossing formula. Then consider an
$\epsilon_-$-stable quasimap with only one length-$d_0$ base point and $v\neq
0,\infty$. As $t\to \infty$, we get an $\epsilon_-$-stable quasimap with
$v=\infty$; as $t\to 0$, we are forced to ``bubble'' a rational tail from the
length-$d_0$ base point in order to keep $v$ away from $0$. Again this will
produce a length-$d_0$ rational tail which contains a length-$d_0$ base point.

The picture is more complicated when we have more than one degree-$d_0$ rational
tails. In this case we need to decide which one to ``reparametrize'', and to
create base points on. This issue is resolved by blowing up the Artin stack of
weighted curves before introducing the $\mathbb P^1$-bundle. For the properness
of the master space, we need to carefully choose the stability condition to
guarantee the desired existence and uniqueness of the orbit limits. Say we have
two degree-$d_0$ rational tails $E_1$ and $E_2$. Then the blowup amounts to
introducing the datum $(w_1,w_2) \in \mathbb P(\Theta_1\oplus \Theta_2)$. The
stability condition is chosen so that when $w_1 = 0$, we are forced to
``reparametrize'' $E_1$ but not $E_2$ (thus in the limit as $t\to \infty$, we
have a length-$d_0$ base point on $E_1$ but not on $E_2$); when $w_2 = 0$, we
``reparametrize'' $E_2$ but not $E_1$; otherwise, we ``reparametrize'' $E_1$ and
$E_2$ simultaneously at equal speed. Note that in the last case, $(w_1,w_2)$ is
an isomorphism between $\Theta_1$ and $\Theta_2$. Thus ``equal speed'' makes
sense. Therefore we call this additional datum coming from blowups the
``entanglement''.

The construction becomes more involved when there are more degree-$d_0$ rational
tails. Section~\ref{sec:entangled-and-entanglement} is devoted to the study of
this construction and its geometric properties. After carefully choosing the
stability condition in Section~\ref{sec:master-space}, we prove the properness
of the master space in Section~\ref{sec:properness}. Thus the standard virtual
localization technique gives us relations among the fixed-point components
(Section~\ref{sec:localization}). This is almost the desired formula, except
that the blowups have changed the boundary divisors in the moduli of weighted
curves. Finally in Section~\ref{sec:the-formula} we deal with the boundary
divisors and prove the desired wall-crossing formula. Note that the ``raw''
formula from the master space technique involves taking the residue of a product
of several terms (e.g.\ \eqref{eq:residue-gl-F-beta}), instead of the product of
several residues (e.g.\ \eqref{eq:Chow-version-literature}). The contribution
from the boundary divisors allows us to transform the formula into the desired
shape.

\subsection{Acknowledgment}
This project was initiated when I was a graduate student at Stanford University
under the supervision of Prof.\ Jun Li and is finished when I am a postdoc at
CMSA under the supervision of Prof.\ Shing-Tung Yau. I am sincerely grateful for
their support and guidance.

I would also like to thank Prof.\ Ravi Vakil, Prof.\ Huai-Liang Chang, Dr.\ Ming
Zhang, Dr.\ Zijun Zhou, Dr.\ Dingxin Zhang, Dr.\ Tsung-Ju Lee for helpful
discussions.

\section{Curves with entangled tails and calibrated tails}
\label{sec:entangled-and-entanglement}
\subsection{Weighted twisted curves}
\label{sec:weighted-twisted-curves}
Let $\epsilon_0=1/d_0$ be a wall. We fix non-negative integers $g,n,d$ such that
$2g-2+n+\epsilon_0 d\geq 0$. When $g=n=0$, in addition we require that
$\epsilon_0 d>2$. We consider $n$-marked weighted twisted
curves of genus $g$ and degree $d$. Here ``weighted'' means that to each
irreducible component is assigned a nonnegative integer referred to as the
``degree''.
For a family of
curves, the assignment is required to be continuous in the sense that locally it
is given by the degree of a line bundle.
See \cite{costello2006higher, hu2010genus} for the precise definition.
Later, the degrees will become the
degrees of quasimaps.

We denote an $n$-marked weighted twisted curve by $(C,x_1 ,\ldots, x_n)$, or by
$(C, \mathbf x)$ for short.
The degrees are suppressed from the notation. From now on, all curves are
weighted twisted curves unless otherwise specified.
\begin{definition}
  \label{def:tail-and-bridge}
  A rational tail (resp.\ bridge) of $(C,\mathbf x)$ is a smooth rational irreducible
  component of $C$ whose normalization has one (resp.\ two) special points, i.e.\
  the preiamge of nodes or markings.
\end{definition}
Thus a rational subcurve with one self-node and no markings is a rational
bridge, while a rational tail is always smooth.

When studying the wall-crossing phenomenon across the wall $\epsilon_0$, we do
not want to consider rational
bridges of degree zero and rational tails of degree strictly smaller than $d_0$.
The degree-$d_0$ rational tails will be crucial.

Let $\mathfrak M^{\mathrm{wt}}_{g,n,d}$ be the moduli stack of $n$-marked weighted twisted
curves of genus $g$ and degree $d$. Here ``$\mathrm{wt}$'' means ``weighted'', and we
suppress the adjective ``twisted'' from the notation, since we always consider twisted curves.
\begin{definition}
We define the stack of ($\epsilon_0$)-semistable weighted curves to be the open substack
$\mathfrak M^{\mathrm{wt},\mathrm{ss}}_{g,n,d}\subset \mathfrak
M^{\mathrm{wt}}_{g,n,d}$
defined by the following conditions:
\begin{itemize}
\item the curve has no degree-$0$ rational bridge,
\item the curve has no rational tail of degree strictly less that $d_0$.
\end{itemize}
\end{definition}
This is a smooth Artin stack of dimension $3g-3 + n$.

\subsection{The entangled tails}
\label{sec:the-entangled-tails}
From $\mathfrak
M_{g,n,d}^{\mathrm{wt,ss}}$ we will construct
another Artin stack $\tilde{\mathfrak M}_{g,n,d}$ via a sequence of blowups.
We will see that for a generic point $\xi\in\tilde{\mathfrak M}_{g,n,d}$ in the fiber over
$(C, \mathbf x)\in \mathfrak M_{g,n,d}^{\mathrm{wt,ss}}$, the identity component
of $\mathrm{Aut}(\xi)$
consists of
those automorphisms of $(C, \mathbf x)$ that act on the degree-$d_0$
rational tails in a certain compatible way (c.f.\ Lemma~\ref{lem:auto-eta}). This
explains the terminology ``entangled tails''.

We now describe the sequence of blowups.
Set
\[
  m=\lfloor d/d_0\rfloor, \quad \mathfrak U_m= \mathfrak M^{\mathrm{wt},\mathrm{ss}}_{g,n,d}.
\]
Thus $m$ is the maximum of the number of degree-$d_0$ rational tails.
Let $\mathfrak Z_i\subset \mathfrak U_m$ be the reduced closed substack
parametrizing curves with at least $i$ rational tails of degree $d_0$. Thus
$\mathfrak Z_1$ is a simple normal crossing divisor and $\mathfrak Z_{i}$ is its
stratum of codimension $i$ (in $\mathfrak U_{m}$).
Note that $\mathfrak Z_{m}$ is smooth, being the deepest stratum.
Let
\[
  \mathfrak U_{m-1} \longrightarrow \mathfrak U_{m}
\]
be the blowup along $\mathfrak Z_{m}$ and let $\mathfrak E_{m-1}\subset \mathfrak
U_{m-1}$ be the exceptional divisor. Inductively for $i=m-1,\ldots,1$, let
\[ \mathfrak Z_{(i)} \subset \mathfrak U_{i}
\]
be the proper
transform of $\mathfrak Z_i$
and let
\[
  \mathfrak U_{i-1}\longrightarrow\mathfrak U_{i}
\]
be the blowup along ${\mathfrak Z_{(i)}}$ with exceptional divisor $\mathfrak
E_{i-1} \subset \mathfrak U_{i-1}$.
Note that $\mathfrak Z_{(i)}$ is the deepest
stratum of the proper transform of
$\mathfrak Z_1$. Hence it is smooth. 
Also note that
$\mathfrak U_0 = \mathfrak U_1$ since $\mathfrak Z_{(1)}$ is a divisor.
We set $\tilde{\mathfrak M}_{g,n,d} = \mathfrak U_{0}$.
\begin{definition}
  We call $\tilde{\mathfrak M}_{g,n,d}$ the moduli stack of genus-$g$
  $n$-marked $\epsilon_0$-semistable curves of degree $d$
  with entangled tails.
\end{definition}

Let $S$ be any scheme.
By definition, an $S$-family of semistable curves with entangled tails is simply
a morphism $e: S \to \tilde{\mathfrak M}_{g,n,d}$. For convenience, we denote it by
\[
  (\pi:\mathcal C\to S, \mathbf x ,e)
\]
where $(\pi:\mathcal C\to S, \mathbf x)$ is the family of (weighted twisted)
marked curves associated with the composition $S\overset{e}{\to}
\tilde{\mathfrak M}_{g,n,d} \to \mathfrak M^{\mathrm{wt},\mathrm{ss}}_{g,n,d}$.
In other words, given $(\pi:\mathcal C\to S, \mathbf x)$, $e$ is a choice of a
lifting $e$ the classifying morphism $S\to \mathfrak
M^{\mathrm{wt},\mathrm{ss}}_{g,n,d}$. Heuristically, when $(\pi:\mathcal C\to S,
\mathbf x)$ is given, we will refer to this choice of $e$ as the
``entanglement''. See Section~\ref{sec:set-theoretic-entanglement} for a
description of the set of such choices.

The following definition of entangled tails is the key to the stability
condition for master moduli space.
Let
\[
  \xi = (\pi: C\to \operatorname{Spec}  K, \mathbf x ,e)
\]
be a geometric point of $\tilde{\mathfrak M}_{g,n,d}$. Let $E_1 ,\ldots, E_\ell$
be the degree $d_0$-rational tails in $C$.
Locally near the image of $\xi$, the analytic branches of $\mathfrak Z_1\subset
\mathfrak M^{\mathrm{wt},\mathrm{ss}}_{g,n,d}$ are
\[
  \mathfrak H_1 ,\ldots, \mathfrak H_l,
\]
where $\mathfrak H_i$ is the locus where $E_i$ remains a rational tail, $i=1
,\ldots, \ell$ \cite{olsson2007}. 
Set
\[
  k = \min\{i \mid \text{the image of $\xi$ in $\mathfrak U_i$ lies in $\mathfrak Z_{(i)}$}\}.
\]
Let $\mathfrak H_{i,k}$ be the proper transform of $\mathfrak H_i$ in $\mathfrak
U_k$, $i=1 ,\ldots, \ell$. 
There exists a unique subset $\{i_1 ,\ldots, i_{k}\}\subset \{1 ,\ldots,
\ell\}$ of size $k$ such that the image of $\xi$ in $\mathfrak U_k$ lies in the
intersection of $\mathfrak H_{i_1,k} ,\ldots, \mathfrak H_{i_k,k}$.
\begin{definition}
  \label{def:entangled-tails}
  In the situation above, we call $E_{i_1} ,\ldots, E_{i_k}$ the
  \textit{entangled} tails of $\xi$.
\end{definition}
\begin{remark}
  The size of the collection entangled tails, as a
function on the closed points of $\tilde{\mathfrak M}_{g,n,d}$, is neither
upper-semicontinuous nor lower-semicontinuous. The level sets are locally closed substacks.
  See Lemma~\ref{lem:entangled-tails-from-gluing}.
\end{remark}

\subsection{A set-theoretic description of entanglements}
\label{sec:set-theoretic-entanglement}
It is hard to find a moduli interpretation for $\tilde{\mathfrak M}_{g,n,d}$.
However, it
is much easier to describe its closed points, which is the goal of this
subsection.
Although this set-theoretic description is helpful for picturing the results,
this paper is logically independent of this subsection. Hence we will work out
an explicit an example and state the general results
without proofs. They can be justified by the remainder of this section.

Fix $(C,\mathbf x)\in {\mathfrak M}^{\mathrm{wt,ss}}_{g,n,d}(\mathbb C)$, the
fiber
\footnote{
  The fiber means the $2$-categorical fiber product with $(C,\mathbf x):
  \operatorname{Spec} \mathbb C\to {\mathfrak M}^{\mathrm{wt,ss}}_{g,n,d}$.
 Thus we are not taking the quotient by $\operatorname{Aut}(C)$.
}
$F_{C}$ of
\[
  \tilde{\mathfrak M}_{g,n,d} \longrightarrow {\mathfrak
    M}^{\mathrm{wt},ss}_{g,n,d}
\]
over $(C, \mathbf x)$ is a projective scheme, since the map is the composition
of sequence of blowups. We will describe $F_C$ explicitly.

We write $C = C^\prime \cup E_1 \cup  \cdots \cup E_\ell$, where $E_1 ,\ldots,
E_\ell$ are the degree-$d_0$ rational tails and $C^\prime$ is the closure of $C\setminus
\cup_{i=1}^\ell E_i$. Suppose $E_i\cap C^\prime = \{p_i\}$ and set
$\Theta_i = T_{p_i}E_i \otimes T_{p_i}C_0$, for $i=1 ,\ldots, \ell$. Here we are
using the tangent space on the orbifold curve. In other words, $\Theta_i$ is the
fiber at $(C, \mathbf x)$ of the normal bundle $N_{\mathfrak H_i/\mathfrak
M_{g,n,d}^{\mathrm{wt,ss}}}$, where $\mathfrak H_i$ is as in the previous section.

The fiber $F_C$ is decomposed into locally closed strata $V_{\mathcal I}\subset F_C$,
indexed by nested subsets $\mathcal I = (I_0 ,\ldots, I_\ell)$, where
\[
  \{1 ,\ldots, \ell\} = I_{\ell} \supseteq I_{\ell-1} \supseteq \cdots \supseteq
  I_0 = \emptyset,
\]
such that for $i=1 ,\ldots, \ell$,
\begin{itemize}
\item
  $\#I_{i}\leq i$
\item
  $I_{i-1} = I_i$ unless $\#I_{i} = i$.
\end{itemize}
The idea is similar to that of Definition~\ref{def:entangled-tails}. The set
$I_i$ determines which local component of $\mathfrak Z_{(i)}$ a point $\xi\in
F_C$ is mapped into. Let $\mathfrak H_{j,i}$ be the proper transform of
$\mathfrak H_j$ in $\mathfrak U_i$. We define $V_{\mathcal I}\subset F_C$ by
requiring that $\xi\in V_{\mathcal I}$ if and only if for each $i=1 ,\ldots,
\ell-1$, we have
\[
  \{j\mid \xi \text{ is mapped to }\mathfrak H_{j,i}\subset \mathfrak U_{i}\} = I_i.
\]

We have a canonical identification
\[
  V_{\mathcal I} = \prod_{I_i\neq I_{i+1}} \mathbb P^*\Big(
  {\bigoplus_{j\in I_{i+1}\setminus
    I_{i}}\Theta_{j}}\Big),
\]
where $\mathbb P^*(\cdots)$ means the complement of the union of all the coordinate
hyperplanes in the projective space $\mathbb P(\cdots)$. The set of entangled
tails is $\{E_j\mid j\in I_k\}$ where $i$ is the smallest index such that
$I_k\neq \emptyset$.

\begin{example}
Let $\ell=5$, $I_5 = \{1 ,2,3,4, 5\}$, $I_3=I_4 = \{1,2,3\}$, $I_2=I_1 = \{1\}$.
For any point $\xi\in V_{\mathcal I}$, let $\xi_i$ be the image of $\xi$ in $\mathfrak
U_i$. We analyze the sequence of blowups near the points $\xi_i$. Thus all the
statements below only hold true near $\xi_i$.
\begin{enumerate}
\item
  $\xi_4\in \mathfrak E_4|_{\xi_5}= \mathbb P(\Theta_1\oplus \cdots \oplus
  \Theta_5)$. Since $I_4=\{1,2,3\}$, we have
  $\xi_4\in \mathfrak H_{1,4}\cap \mathfrak H_{2,4} \cap \mathfrak H_{3,4}$
  and $\xi_4 \not\in \mathfrak
  H_{4,4} \cup \mathfrak H_{5,4}$. Thus indeed $\xi_4 \in \mathbb P^*(\Theta_4 \oplus
  \Theta_5)$.
\item
  $\mathfrak U_3 \to   \mathfrak U_4$ is an isomorphism. This
  explains why we must have $I_3=I_4$.
\item
  $\mathfrak U_2 \to \mathfrak U_3$ is the blowup along $\mathfrak
  H_{1,3}\cap \mathfrak H_{2,3} \cap \mathfrak H_{3,3}$. Since as divisors
  $\mathfrak H_{i,3} = \mathfrak H_i - \mathfrak E_{4}$, we have $p_2\in \mathfrak
  E_2|_{p_3} = \mathbb P^*(\Theta_1 \oplus \Theta_2 \oplus \Theta_3)$. Since $I_2=
  \{1\}$, we have $p_2\in \mathfrak H_{1,2}$ and $p_2\not\in \mathfrak H_{2,2}\cup
  \mathfrak H_{3,2}$. Thus indeed $p_2\in \mathbb P^*(\Theta_2 \oplus \Theta_3)$.
\item
  $\mathfrak U_1 \to \mathfrak U_2$ is an isomorphism. This explains why
  $I_1=I_2$.
\item
  $\mathfrak U_0 \to  \mathfrak U_1$ is the blowup along $\mathfrak
  H_{1,1}$, which is a divisor. Hence it is an isomorphism. Nevertheless, we
  write $\mathfrak E_{0}|_{p_1} = \mathbb P(\Theta_1) = \mathbb P^*(\Theta_1)$.
\end{enumerate}
Hence we obtain an isomorphism
\[
  V_{\mathcal I} \cong \mathbb P^*(\Theta_1) \times \mathbb P^*(\Theta_2\oplus
  \Theta_3) \times \mathbb P^*(\Theta_4\oplus \Theta_5).
\]
And there is only one entangled tail, namely, $E_1$.
\end{example}
\subsection{The structure of the blowup cernters $\mathfrak Z_{(k)}$}
Consider ${\mathfrak M}^{\mathrm{wt, ss}}_{g,n+k,d-kd_0}$ and $\mathfrak
M^{\mathrm{wt,ss}}_{0,1,d_0}$. Note that the curves parameterized by the latter
are all (irreducible) rational tails of degree $d_0$. We would like to glue
those tails to the curves parameterized by the former at the their last
$k$-markings. Since our markings are trivialized gerbes, we can glue two
markings as along as their automorphism groups have the same size. We write
\begin{equation}
  \label{eq:matching-order}
  {\mathfrak M}^{\mathrm{wt, ss}}_{g,n+k,d-kd_0} {\times}^\prime {\big(
      \mathfrak M^{\mathrm{wt,ss}}_{0,1,d_0}\big)}^k
\end{equation}
for 
${\mathfrak M}^{\mathrm{wt,ss}}_{g,n+k,d-kd_0} \underset{\mathbb N^k}{\times} {( \mathfrak
      M^{\mathrm{wt,ss}}_{0,1,d_0})}^k$,
where the fiber product is formed by matching the sizes of the automorphism
groups at the last $k$-markings.
We use the same notation in other similar situations where we need to glue markings.  
Let
\[
  \mathbf r_i: {\mathfrak M}^{\mathrm{wt, ss}}_{g,n+k,d-kd_0} {\times}^\prime {\big(
    \mathfrak M^{\mathrm{wt,ss}}_{0,1,d_0}\big)}^k \longrightarrow  \mathbb N
\]
be the projection to the $i$-th
copy of $\mathbb N$. Thus the value of $\mathbf r_i$ is the order of the automorphism
group at the $(n+i)$-th marking of ${\mathfrak M}^{\mathrm{wt,ss}}_{g,n+k,d-kd_0}$.
Gluing the $k$ rational tails from ${( \mathfrak M^{\mathrm{wt,ss}}_{0,1,d_0})}^k$ to the
last $k$ markings of the universal curve over ${\mathfrak
  M}^{\mathrm{wt,ss}}_{g,n+k,d-kd_0}$ defines a gluing morphism
\[
  \mathrm{gl}_k: {\mathfrak M}^{\mathrm{wt,ss}}_{g,n+k,d-kd_0} {\times}^\prime {\big( \mathfrak
      M^{\mathrm{wt,ss}}_{0,1,d_0}\big)}^k
  \longrightarrow
  \mathfrak Z_{k}\subset \mathfrak M^{\mathrm{wt,ss}}_{g,n,d}.
\]
Since ${\mathfrak M}^{\mathrm{wt,ss}}_{g,n+k,d-kd_0} \times^\prime {( \mathfrak
      M^{\mathrm{wt,ss}}_{0,1,d_0})}^k$ is smooth, this morphism factors through the
normalization of $\mathfrak Z_{k}$, which we denote by $\mathfrak Z^{\mathrm{nor}}_{k}$.
\begin{lemma}
  The induced morphism
  \[
    \mathrm{gl}^{\mathrm{nor}}_k: {\mathfrak M}^{\mathrm{wt,ss}}_{g,n+k,d-kd_0}
    \times^\prime {\big( \mathfrak
    M^{\mathrm{wt,ss}}_{0,1,d_0}\big)}^k \longrightarrow  \mathfrak Z^{\mathrm{nor}}_{k}
  \]
  is \'etale (but not representable in general) of degree $k!/\prod_{i=1}^{k}\mathbf r_i$.

\end{lemma}
\ifdefined\SHOWPROOFS
\begin{proof}
For any scheme $S$, an object $\xi\in \mathfrak Z_k(S)$ corresponds to an
$S$-family of curves with at least $k$ rational tails of degree $d_0$.
A lifting of $\xi$ to $\xi^{\mathrm{nor}}\in\mathfrak Z^{\mathrm{nor}}_k(S)$
corresponds to a choice of a set of $k$ rational tails of degree $d_0$. A
further lifting of $\xi^{\mathrm{nor}}$ to an $S$-point of ${\mathfrak
M}^{\mathrm{wt,ss}}_{g,n+k,d-kd_0} \times^\prime {( \mathfrak
    M_{0,1,d_0}^{\mathrm{wt,ss}})}^k$  corresponds to a
choice of labeling of those $k$ tails by $\{1 ,\ldots, k\}$, together with a
choice of an $S$-section of the orbifold nodes on those tails.\footnote{More
  rigorously, one should
  consider the groupoid of liftings of $\xi^{\mathrm{nor}}$, since $\mathrm{gl}_k$ is
  not representable in general. We leave the precise formulation to the reader.}
\end{proof}
\fi
\begin{lemma}
  \label{lem:structure-Z_k}
  For each $1\leq k\leq m$, there is a unique $\tilde{\mathrm{gl}}_k$ we have a
  fibered diagram
  \[
    \begin{tikzcd}
      \tilde{\mathfrak M}_{g,n+k,d-kd_0} \times^\prime {\big( \mathfrak
        M_{0,1,d_0}^{\mathrm{wt,ss}}\big)}^k \arrow[r,"\tilde{\mathrm{gl}}_k"] \arrow[d] &
      \mathfrak Z_{(k)}\arrow[d] \\
      \mathfrak M^{\mathrm{wt,ss}}_{g,n+k,d-kd_0} \times^\prime {\big( \mathfrak
        M^{\mathrm{wt,ss}}_{0,1,d_0}\big)}^k \arrow[r,
      "\mathrm{gl}^{\mathrm{nor}}_k" ] & \mathfrak Z_{k}^{\mathrm{nor}}
    \end{tikzcd}.
  \]
  Thus $\tilde{\mathrm{gl}}_k$ is \'etale of degree
  $k!/\prod_{i=1}^{k}\mathbf r_i$.
  Moreover, $\tilde{\mathrm{gl}}_k$ is the unique morphism making the diagram
  commute.\footnote{For a diagram of morphisms between stacks, we should consider $2$-commutativity.
    However, this causes no issue since
    $\mathfrak Z_{(k)} \to \mathfrak Z^{\mathrm{nor}}_{k}$
    is representable.}
\end{lemma}
\ifdefined\SHOWPROOFS
\begin{proof}
  Let $\mathfrak V\subset \mathfrak Z^{\mathrm{nor}}_{k}$ be the dense open
    substack where there are at most $k$ rational tails of degree $d_0$.
    Then $(\mathrm{gl}^{\mathrm{nor}}_k)^{-1}(\mathfrak V)$ is also dense open.
    Over $\mathfrak V$ and $(\mathrm{gl}^{\mathrm{nor}}_k)^{-1}(\mathfrak V)$
    the vertical arrows are identities.
    Since $\mathfrak Z_{(k)} \to \mathfrak Z_k^{\mathrm{nor}}$ is representable and
    separated, the uniqueness is obvious and we can work locally.

    We first introduce some notation.  We use $\mathfrak U_{i}^\prime$, $\mathfrak
    Z_{i}^\prime$ and
    $\mathfrak Z_{(i)}^\prime$ to denote the stacks introduced in the
    construction of $\tilde{\mathfrak M}_{g,n+k,d-kd_0}$, in parallel with
    $\mathfrak U_{i}$, $\mathfrak Z_{i}$ and $\mathfrak
    Z_{(i)}$. For $k\leq i\leq m$, let $\mathfrak Z_{k,i}$ be the proper
    transform of $\mathfrak Z_{k}$ in $\mathfrak U_i$.  Let $\mathfrak
    Z_{k,i}^{\mathrm{nor}}$ be the normalization of $\mathfrak Z_{k,i}$.
    Thus $\mathfrak Z_{k,m} = \mathfrak Z_{k}$ and
    $\mathfrak Z_{k,k}^{\mathrm{nor}} = \mathfrak Z_{k,k} = \mathfrak Z_{(k)}$.

    For $i = m ,\ldots, k+1$, we will successively form the fibered diagrams
    \begin{equation}
    \label{cd:fiber-step}
      \begin{tikzcd}
      \mathfrak U^\prime_{i-k-1} \times^\prime {\big( \mathfrak
          M^{\mathrm{wt,ss}}_{0,1,d_0}\big)}^k \arrow[r,dotted] \arrow[d] & \mathfrak
        Z^{\mathrm{nor}}_{k,i-1}\arrow[d] \\
      \mathfrak U^\prime_{i-k}\times^\prime {\big( \mathfrak
          M_{0,1,d_0}^{\mathrm{wt,ss}}\big)}^k \arrow[r] & \mathfrak Z_{k,i}^{\mathrm{nor}}
    \end{tikzcd},
  \end{equation}
  by constructing the dotted arrow. To do this, we will describe the two
  vertical arrows as blowups and compare their blowup centers.

  Consider the divisor $\mathfrak Z^{\prime\prime}_1\subset \mathfrak
  Z^{\mathrm{nor}}_k$ where the curve has at least $(k+1)$ tails of degree
  $d_0$. We claim that the vertical arrows on the right hand side of
  \eqref{cd:fiber-step} is the blowup along the codimension-$(i-k)$ 
  stratum of the proper transform of $\mathfrak Z^{\prime\prime}_1$. Here
  the codimension refers to the codimension in $\mathfrak
  Z^{\mathrm{nor}}_{k,i}$.

  Indeed, locally let $\mathfrak H_1 ,\ldots,
  \mathfrak H_\ell\subset \mathfrak M_{g,n,d}^{\mathrm{wt,ss}}$ be the
  branches of $\mathfrak Z_1$, and $\mathfrak H_{j,i}$ be the proper
  transform of $\mathfrak H_{j}$ in $\mathfrak U_i$, as in
  Section~\ref{sec:the-entangled-tails}.
  By relabeling those branches and working further locally on $\mathfrak
  Z^{\mathrm{nor}}_{k}$, we may assume that
  $\mathfrak Z^{\mathrm{nor}}_{k}$ is equal to $\mathfrak H_1 \cap \cdots \cap
  \mathfrak H_{k}$. Thus $\mathfrak Z^{\prime\prime}_1$ is equal to
  $(\mathfrak H_{k+1}\cup \cdots \cup \mathfrak H_{\ell})\cap \mathfrak
  Z^{\mathrm{nor}}_{k}$. Consider the blowup $\mathfrak U_{i-1} \to
  \mathfrak U_{i}$ for $i>k$. The blowup center $\mathfrak Z_{(i)}$ is the
  disjoint union of
  \[
    \mathfrak H_{j_1,i} \cap \cdots \cap \mathfrak H_{j_i,i}
  \]
  for all $\{j_1 ,\ldots, j_i\}\subset \{1 ,\ldots, \ell\}$.
  Since any $(i+1)$ of
  $\mathfrak H_{1,i} ,\ldots, \mathfrak H_{\ell,i}$ do not intersect, each
  connected component of $\mathfrak Z_{(i)}$ is either disjoint from
  $\mathfrak H_{1,i}\cap \cdots \cap \mathfrak H_{k,i}$ or equal to
  $\mathfrak H_{1,i}\cap \cdots \cap \mathfrak H_{k,i} \cap \mathfrak
  H_{j_1,i} \cap 
  \cdots \cap \mathfrak H_{j_{i-k},i}$, for some $\{j_1 ,\ldots, j_{i-k}\}
  \subset \{k+1 ,\ldots, \ell\}$.
  Thus, inductively it is easy to see that locally $\mathfrak Z^{\mathrm{nor}}_{k,i}
  = \mathfrak H_{1,i} \cap \cdots \cap \mathfrak H_{k,i}$ and $\mathfrak
  Z^{\mathrm{nor}}_{k,i-1} \to \mathfrak
  Z^{\mathrm{nor}}_{k,i-1}$ is the blowup along the codimension-$(i-k)$
  stratum of the  normal crossing divisor $\mathfrak H_{k+1}\cup \cdots \cup
  \mathfrak H_{\ell}$. This proves the claim.

  We then compare it to the vertical arrow on the left hand side of
  \eqref{cd:fiber-step}.
  Let $\mathrm{pr}_1$ be the projection from $\mathfrak
  M^{\mathrm{wt,ss}}_{g,n+k,d-kd_0} \times^\prime {( \mathfrak
    M_{0,1,d_0}^{\mathrm{wt,ss}})}^k$ onto its first factor.
  Working locally as before, $\mathrm{gl}^{\mathrm{nor}}_{k}$ is an \'etale
  cover of $\mathfrak H_1 \cap \cdots \cap \mathfrak H_{k}$, and the preimage of $(\mathfrak
  H_{k+1}\cup \cdots \cup \mathfrak H_{\ell}) \cap \mathfrak Z^{\mathrm{nor}}_k$
  is equal to $\mathrm{pr}_1^{-1}(\mathfrak
  Z_1^\prime)$. Hence, $\tilde{\mathrm{gl}}{}_k^{-1}(\mathfrak
  Z_1^{\prime\prime}) = \mathrm{pr}_1^{-1}(\mathfrak Z_1^\prime)$.
  Since $\mathrm{pr}_1$ is smooth, the vertical arrow on the
  left hand side of $\eqref{cd:fiber-step}$ is the blowup along the
  codimension-$(i-k)$ stratum of the proper transform of
  $\mathrm{pr}_1^{-1}(\mathfrak Z_1^\prime)$.
  Since taking proper transforms, taking the codimension-$(i-k)$ stratum and
  blowing up all commute with \'etale morphisms,
  we obtain the fibered diagram \eqref{cd:fiber-step}.
  Thus the proof is complete.
\end{proof}
\fi

\subsection{Structure of the exceptional divisors $\mathfrak E_{k-1}$, entangled
  tails revisited}
\label{sec:stru-ent}
Recall the gluing morphism in Lemma~\ref{lem:structure-Z_k}
\begin{equation}
  \label{map:quotient-by-S}
  \tilde{\mathrm{gl}}_k: \tilde{\mathfrak M}_{g,n+k,d-kd_0} \times^\prime {\big( \mathfrak
      M_{0,1,d_0}^{\mathrm{wt,ss}}\big)}^k \longrightarrow \mathfrak Z_{(k)}.
\end{equation}
We will study everything by pulling them back along $\tilde{\mathrm{gl}}_k$.
For $i=1,\ldots,k$, let $\Theta_i$ be the line bundles on
$\tilde{\mathfrak M}_{g,n+k,d-kd_0} \times^\prime {(\mathfrak
M^{\mathrm{wt,ss}}_{0,1,d_0})}^k$ formed by
the tensor product of two tangent lines to the curves, one at the
$(n+i)$-th marking of $\tilde{\mathfrak M}_{g,n+k,d-kd_0}$, and the other
at the unique marking of the $i$-th copy of $ \mathfrak M^{\mathrm{wt,ss}}_{0,1,d_0}$.

By abuse of notation, we will use $\mathcal O_{\mathfrak U_{k}}(\mathfrak
E_{i})$ to denote the pullback of $\mathcal O_{\mathfrak U_{i}}(\mathfrak
E_{i})$, for $i = k ,\ldots, m-1$. For $i\geq m$,
we set $\mathfrak E_i = \emptyset$.
\begin{lemma}
  \label{lem:structure-E_k-1}
  When pulled back along $\tilde{\mathrm{gl}}_k$, the normal bundle of $\mathfrak Z_{(k)}$ in
  $\mathfrak U_{k}$ is canonically isomorphic to (the pullback of)
  \[
    \textstyle
    \left( \Theta_1\oplus\cdots \oplus \Theta_k \right) \otimes \mathcal
    O_{\mathfrak U_k}\big(-\sum_{i=k}^\infty \mathfrak E_i \big).
  \]
  In particular,
  \begin{equation}
    \label{map:E_k-1-as-projective-bundle}
    \tilde{\mathrm{gl}}_k^*\mathfrak E_{k-1} \longrightarrow \tilde{\mathfrak M}_{g,n+k,d-kd_0}
    \times^\prime {\big( \mathfrak M_{0,1,d_0}^{\mathrm{wt,ss}}\big)}^k
  \end{equation}
  is isomorphic to the projective bundle
  \[
    \textstyle
    \mathbb P\big((\Theta_1\oplus\cdots\oplus \Theta_k) \otimes \mathcal
    O_{\mathfrak U_k}
    (-\sum_{i=k}^\infty \mathfrak E_i )\big) \cong \mathbb
    P((\Theta_1\oplus\cdots\oplus \Theta_k).
  \]
\end{lemma}
\ifdefined\SHOWPROOFS
\begin{proof}
  When pulled back along $\tilde{\mathrm{gl}}_k$, the normal bundle of the local
  immersion $\mathfrak Z_{k}^{\mathrm{nor}}\to \mathfrak U_{m}$ is
  $\Theta_1\oplus\cdots \oplus \Theta_k$.

  Let $\mathfrak Z_{k,i}$ be the proper transform of $\mathfrak Z_k$ in
  $\mathfrak U_{i}$, and $\mathfrak Z^{\mathrm{nor}}_{k,i}$ be the normalization
  of $\mathfrak Z_{k,i}$, $i = k+1,\ldots, m$. Working locally as in the proof of
  Lemma~\ref{lem:structure-Z_k}, if a local branch of $\mathfrak
  Z_{k,i}^{\mathrm{nor}}$ intersects a connected component of $\mathfrak Z_{(i)}$
  nontrivially, then it contains that connected component. Since $\mathfrak E_{k}
  ,\ldots, \mathfrak E_{m-1}$ are the exceptional divisors of the sequence of
  blowups $\mathfrak U_{k}\to \mathfrak U_m$ and $\mathfrak
  Z^{\mathrm{nor}}_{k,k}= \mathfrak Z_{k,k} = \mathfrak Z_{(k)}$, we get the
  desired isomorphism of vector bundles locally on $\mathfrak Z_{(k)}$. Since the
  isomorphism is canonical, it glues to a global isomorphism.
\end{proof}
\fi

Let $\mathfrak Z_{1,k-1}$ be the proper transform of $\mathfrak Z_{1}$ in
$\mathfrak U_{k-1}$.
\begin{lemma}
  \label{lem:Z1-in-E}
  The divisor $\mathfrak Z_{1,k-1} \cup \mathfrak E_{k-1}$ is a normal crossing
  divisor, and
  \[
    \tilde{\mathrm{gl}}_k^* \mathfrak E_{k-1} \times_{\mathfrak U_{k-1}} \mathfrak
    Z_{1,k-1} = 
    \textstyle
    \bigcup_{i=1}^k \mathbb P({\Theta_1 \oplus \cdots \oplus \hat\Theta_{i} \oplus
      \cdots \oplus \Theta_{k}}).
  \]
  is the union of coordinate hyperplanes of the projective
  bundle~\eqref{map:E_k-1-as-projective-bundle}.
\end{lemma}
\ifdefined\SHOWPROOFS
\begin{proof}
  We continue the local study in Lemma~\ref{lem:structure-Z_k}. Using the
  notation there, $\mathfrak Z_{1,k-1} = \mathfrak H_{1,k-1}\cup \cdots \cup
  \mathfrak H_{\ell,k-1}$. The $\mathfrak H_{k+1,k-1} ,\ldots, \mathfrak
  H_{\ell,k-1}$ are disjoint from the exceptional divisor over $\mathfrak
  H_{1,k}\cap \cdots \cap \mathfrak H_{k,k}$, and $\mathfrak H_{i,k-1}$
  intersects the exceptional divisor transversely along $\mathbb P({\Theta_1
    \oplus \cdots \oplus \hat\Theta_{i} \oplus \cdots \oplus \Theta_{k}})$, $i=1
  ,\ldots, k$. From the local computation it is easy to see that $\mathfrak
  Z_{1,k-1} \cup \mathfrak E_{k-1}$ is a normal crossing divisor.
\end{proof}
\fi
Let $\mathfrak E_{k-1}^*$ be the complement of $\mathfrak Z_{1,k-1} \cap
\mathfrak E_{k-1}$ in $\mathfrak E_{k-1}$.
\begin{corollary}
  \label{cor:E-star-in-M-tilde}
  The morphism 
  $\tilde{\mathfrak M}_{g,n,d} \to \mathfrak U_{k-1}$
  is an isomorphism near $\mathfrak E^*_{k-1}$. 
\end{corollary}
\ifdefined\SHOWPROOFS
\begin{proof}
  This follows from Lemma~\ref{lem:Z1-in-E} and the construction of
  $\tilde{\mathfrak M}_{g,n,d}$.
\end{proof}
\fi

Thus we can view $\mathfrak
E_{k-1}^*$ as a substack of $\tilde{\mathfrak M}_{g,n,d}$. Let
$\tilde{\mathfrak M}^*_{g,n,d}$ be the dense open substack of $\tilde{\mathfrak M}_{g,n,d}$ where
the curve has no degree-$d_0$ rational tail. Then $\tilde{\mathfrak M}_{g,n,d} \setminus
\tilde{\mathfrak M}^*_{g,n,d}$ is the disjoint union of $\mathfrak E^*_0 ,\ldots, \mathfrak E^*_{m-1}$.
Indeed, this is a partition according to the size of the entangled tails, by the following Lemma.

\begin{lemma}
  \label{lem:entangled-tails-from-gluing}
  Let $\xi$ be a closed point of $\mathfrak E^*_{k-1}$ and let $\xi_k$ be its
  image in $\mathfrak Z_{(k)}\subset \mathfrak U_k$.
  Suppose that 
  \[
    \tilde{\mathrm{gl}}_{k}(\xi^\prime, (E_1,y_1) ,\ldots, (E_k,y_k))
  \]
  is isomorphic to $\xi_k$, for some $\xi^\prime \in \tilde {\mathfrak
    M}_{g,n+k,d-kd_0}$ and $(E_i,y_i)\in \mathfrak M_{0,1,d_0}^{\mathrm{wt,ss}}$,
  $i=1 ,\ldots, k$. Then, viewing $\xi$ as a closed point of
  $\tilde{\mathfrak M}_{g,n,d}$, its entangled tails are precisely the $E_1 ,\ldots,
  E_k$.
\end{lemma}
\ifdefined\SHOWPROOFS
\begin{proof}
  This follows immediately from the proof of Lemma~\ref{lem:structure-Z_k} and
  the discussion above.
\end{proof}
\fi

The following lemma is an elementary fact about blowups and projective bundles.
\begin{lemma} When pulled back to $\tilde{\mathrm{gl}}_k^{*}\mathfrak
  E_{k-1}^*$,
  we have canonical isomorphisms \label{lem:normal-E_k-1}
  \[
    \Theta_1\cong \cdots\cong \Theta_k,
  \]
  and the normal bundle of $\mathfrak E_{k-1}^*$ in $\mathfrak U_{k-1}$ is (the
  pullback of)
  \[
    \textstyle
    \Theta_1 \otimes \mathcal O_{\mathfrak U_{k}}\big(-\sum_{i=k}^\infty \mathfrak
    E_i\big).
  \]
\end{lemma}
\subsection{Digression on inflated projective bundles}
\label{sec:inflated-proj-bundle}
To study the proper transform of $\mathfrak E_{k-1}$ in $\tilde{\mathfrak M}_{g,n,d}$, we
digress on an elementary construction which we call inflated projective
bundles.
Let $X$ be any algebraic stack and $L_1,\ldots,L_k$ be line bundles on
$X$. We first form the projective bundle
\[
  P = \mathbb P(L_1\oplus\ldots\oplus L_k) \longrightarrow X.
\]
Consider the coordinate hyperplanes
\[
  H_i = \mathbb P(L_1\oplus\cdots \oplus \{0\} \oplus \cdots\oplus L_k),
\]
where the $\{0\}$ appears in the $i$-th place only.
The construction of the inflated projective bundle is analogous to the
construction of $\tilde {\mathfrak M}_{g,n,d}$. The normal crossing
divisor $\bigcup_{i=1}^k H_i$ here plays the role of
$\mathfrak Z_1 \subset \mathfrak M^{\mathrm{wt,ss}}_{g,n,d}$.

More specifically, for $i=1,\ldots,k-1$, let $Z_i\subset P$ be the union of the codimension-$i$
coordinate subspaces,
i.e.\
\[
  \textstyle
  Z_i = \bigcup H_{j_1}\cap \cdots\cap H_{j_i},
\]
where $\{j_1,\ldots,j_i\}$ runs through all subsets of $\{1,\ldots,k\}$ of size $i$.
First set $P_{k-1}= P$. Inductively for $i=k-1,\cdots,1$, let
$ Z_{(i)} \subset P_i$ be the proper transform of $Z_i$ and
\[
  P_{i-1} \longrightarrow P_i
\]
be the blowup along $Z_{(i)}$, with exceptional divisor $E_{i-1}\subset P_{i-1}$.

\begin{definition}
  \label{def:inflated-projective-bundle}
  We call $\tilde {\mathbb  P }(L_1,\ldots,L_k) := P_0\to X$ the
  \textit{inflated projective bundle} associated to the line bundles $L_1,\ldots,L_k$.
\end{definition}
Let $D_i\subset \tilde {\mathbb P }(L_1,\ldots,L_k)$ be the proper transform of
$E_i$, for $i=0,\ldots,k-2$.
They are also the total transforms.
Since line bundles are locally trivial, these
constructions commute with arbitrary base change. Hence $\tilde{\mathbb P}
(L_1,\ldots,L_k)$ is smooth over $X$ of relative dimension $k-1$, and $D_i$ are
relative effective Cartier divisors.
\begin{definition}
  \label{def:tautological-divisor-inflated-projective-bundle}
  We call $D_i$ the $i$-th tautological divisor of the inflated projective
  bundle, $i=0 ,\ldots, k-2$.
\end{definition}

\subsection{The boundary divisors $\mathfrak D_{k-1}$ of $\tilde {\mathfrak
    M}_{g,n,d}$ and their intersection theory}
\label{sec:boundary-divisor}
We now come back to the moduli $\tilde {\mathfrak M}_{g,n,d}$ of
curves with entangled tails.
Recall that $\mathfrak E_{k-1}^*\subset \tilde {\mathfrak M}_{g,n,d}$ is
the locus where there are exactly $k$ entangled tails
(Lemma~\ref{lem:entangled-tails-from-gluing}). This is locally closed smooth substack
of codimension $1$.
\begin{definition}
  \label{def:D_k-1}
 For $k=1 ,\ldots, m$, the boundary divisors $\mathfrak D_{k-1}\subset
 \tilde{\mathfrak{M}}_{g,n,d}$ is defined to be the
closure of $\mathfrak E_{k-1}^*$. 
\end{definition}
In other words, $\mathfrak D_{k-1}$ is the
proper transform of $\mathfrak E_{k-1}\subset \mathfrak U_{k-1}$.

In particular, 
$\mathfrak D_{k-1}$ is a smooth divisor in $\tilde{\mathfrak M}_{g,n,d}$, and
we have natural maps
\[
  \mathfrak D_{k-1} \longrightarrow  \mathfrak E_{k-1} \longrightarrow
  \mathfrak Z_{(k)} \subset \mathfrak U_k.
\]
Let $\tilde{\mathrm{gl}}_k^* \mathfrak D_{k-1}$ be the pullback of $\mathfrak
D_{k-1}$ along \eqref{map:quotient-by-S}. Thus we have a natural morphism

\begin{equation}
  \label{eq:D-as-inf-proj-bdl}
  {\tilde{\mathrm{gl}}_k^*}\mathfrak D_{k-1} \longrightarrow \tilde{\mathfrak M}_{g,n+k,d-kd_0}
  \times^\prime {\big(\mathfrak M^{\mathrm{wt,ss}}_{0,1,d_0}\big)}^k.
\end{equation}

\begin{lemma}
  \label{lem:str-D}
  For $k=1 ,\ldots, m$, $\mathfrak D_{k-1}$ is the inverse image of $\mathfrak
  E_{k-1}$ under $\tilde{\mathfrak M}_{g,n,d} \to \mathfrak U_{k-1}$, and the
  morphism \eqref{eq:D-as-inf-proj-bdl} realizes
  ${\tilde{\mathrm{gl}}_k^*}\mathfrak D_{k-1}$ as the inflated projective bundle
  \[
    \tilde{\mathbb P}\left( \Theta_1, \ldots ,\Theta_k \right) \longrightarrow
    \tilde{\mathfrak M}_{g,n+k,d-kd_0} \times^\prime {\big(\mathfrak
      M_{0,1,d_0}^{\mathrm{wt,ss}}\big)}^k.
  \]
\end{lemma}

\ifdefined\SHOWPROOFS
\begin{proof}
  For $i=0 ,\ldots, k-1$, let $\mathfrak E_{k-1,i}$ be the proper transform of
  $\mathfrak E_{k-1}$ in $\mathfrak U_i$. Recall that $\mathfrak U_{i-1} \to
  \mathfrak U_{i}$ is the blowup along $\mathfrak Z_{(i)}$, where $\mathfrak
  Z_{(i)}$ is the deepest stratum of the proper
  transform $\mathfrak Z_{1,i}$ of the normal crossing divisor $\mathfrak
  Z_1\subset \mathfrak U_m = \mathfrak M_{g,n,d}^{\mathrm{wt,ss}}$.
  The proof is by analyzing $\mathfrak E_{k-1,i-1}\to \mathfrak E_{k-1,i}$
  inductively for $i=k-1 ,\ldots, 1$. First set $i=k-1$. Thus we have $\mathfrak
  E_{k-1,i} = \mathfrak E_{k-1}$.
  By Lemma~\ref{lem:Z1-in-E}, $\mathfrak Z_{1,i}\cup \mathfrak E_{i}$ is a
  normal crossing divisor. Thus some local computation shows
  that $\mathfrak E_{k-1,i-1}$ is equal to the preimage of $\mathfrak
  E_{k-1,i}$, and the map $\mathfrak E_{k-1,i-1}\to \mathfrak
  E_{k-1,i}$ is isomorphic to the blowup along the deepest stratum of 
  $\mathfrak Z_{1,i}\cap \mathfrak E_{i}$. Moreover, $\mathfrak Z_{1,i-1} \cup
  \mathfrak E_{k-1,i-1}$ is again a normal crossing divisor. Thus we can
  continue this argument for all $i=k-1 ,\ldots, 1$. In particular, $\mathfrak
  D_{k-1} = \mathfrak E_{k-1,0}$ is equal
  to the preimage of $\mathfrak E_{k-1}$. The second statement of the Lemma
  follows from comparing the process above and the definition of inflated
  projective bundles, using Lemma~\ref{lem:Z1-in-E}. 
\end{proof}
\fi

Fix $k$ in the remainder of this subsection and let
\[
  \iota_{\mathfrak D}:  {\tilde{\mathrm{gl}}_k^*}\mathfrak D_{k-1}
  \longrightarrow \tilde{\mathfrak M}_{g,n,d}
\]
be the composition of ${\tilde{\mathrm{gl}}_k^*}\mathfrak D_{k-1} \to \mathfrak
D_{k-1} \subset \tilde{\mathfrak M}_{g,n,d}$.
\begin{lemma}
  \label{lem:divisors1}
  Consider the morphism $\iota_{\mathfrak D}$ above.
  \begin{enumerate}
  \item
    For $0\leq \ell \leq k-2$, the pullback $\iota_{\mathfrak
      D}^*\mathfrak D_\ell$ of $\mathfrak D_{\ell}$ as a divisor is equal to the
    $\ell$-th tautological divisor $D_\ell$ of the inflated projective bundle
    \eqref{eq:D-as-inf-proj-bdl}.
  \item
    For $\ell\geq k$, the pullback $\iota_{\mathfrak D}^*\mathfrak
    D_{\ell}$ of $\mathfrak D_{\ell}$ as a divisor is equal to $\operatorname{pr}_1^*(\mathfrak D^\prime_{\ell-k})$,
    where $\operatorname{pr}_1$ is the composition of the projections
    \[
      \mathrm{pr}_1: {\tilde{\mathrm{gl}}_k^*}\mathfrak D_{k-1} \longrightarrow
      \tilde{\mathfrak M}_{g,n+k,d-kd_0}
      \times^\prime {\big(\mathfrak M_{0,1,d_0}^{\mathrm{wt,ss}}\big)}^k \longrightarrow  \tilde{\mathfrak M}_{g,n+k,d-kd_0},
    \]
    and $\mathfrak D^{\prime}_{\ell-k}$ is the boundary divisor of
    $\tilde{\mathfrak M}_{g,n+k,d-kd_0}$.
  \end{enumerate}
\end{lemma}
\ifdefined\SHOWPROOFS
\begin{proof}
  Both statements are clear when we locally write $\mathfrak Z_1$ as a strict
  normal crossing divisor. Or more specifically,
  (1) follows immediately from the proof for Lemma~\ref{lem:str-D} and (2)
  follows immediately from the proof for Lemma~\ref{lem:structure-Z_k}.
\end{proof}
\fi
By abuse of notation, the pullback of $\Theta_i$ to ${\tilde{\mathrm{gl}}_k^*}\mathfrak D_{k-1}$
is still denoted by $\Theta_i$.
We cannot pull back $\mathfrak D_{k-1}$ to itself as a divisor. Instead, we have
the following.
\begin{lemma}
  \label{lem:divisors2}
  Along the map $\iota_{\mathfrak D}$,
  the line bundle
  \[
    \mathcal O_{\tilde{\mathfrak M}_{g,n,d}}\left( k(\mathfrak D_0 + \mathfrak
      D_1 + \cdots + \mathfrak D_{m-1}) \right)
  \]
  pulls back to
  \[
    \mathcal O_{{\tilde{\mathrm{gl}}_k^*}\mathfrak D_{k-1}} \left(
      (k-1)D_0 + (k-2)D_1+ \cdots + D_{k-2}
    \right)\otimes \Theta_1 \otimes \cdots \otimes \Theta_k,
  \]
  where the divisors $D_0,\ldots,D_{k-2}$ are the tautological divisors of the
  inflated projective bundle \eqref{eq:D-as-inf-proj-bdl}.
\end{lemma}
\ifdefined\SHOWPROOFS
\begin{proof}
  We compute $\iota_{\mathfrak D}^*\mathcal O_{\tilde{\mathfrak
      M}_{g,n,d}}(k\mathfrak D_{k-1})$. By Lemma~\ref{lem:str-D}, $\mathcal
  O_{\tilde{\mathfrak M}_{g,n,d}}(\mathfrak D_{k-1})$ is isomorphic to the
  pullback of $\mathcal O_{\mathfrak U_{k-1}}(\mathfrak E_{k-1})$. Note that the
  map ${\tilde{\mathrm{gl}}_k^*}\mathfrak D_{k-1}\to \mathfrak U_{k-1}$ factors
  through $\mathfrak E_{k-1}$. Hence $\iota_{\mathfrak D}^* \mathcal O_{\tilde{\mathfrak
      M}_{g,n,d}}(\mathfrak D_{k-1})$ is canonically isomorphic to the pullback of $\mathcal
  O_{\tilde{\mathrm{gl}}_{k}^*\mathfrak E_{k-1}}(-1)$,  the relative
  $\mathcal O(-1)$ of the projective bundle
  \eqref{map:E_k-1-as-projective-bundle}. Let $H_1 ,\ldots, H_k\subset
  \tilde{\mathrm{gl}}_{k}^*\mathfrak E_{k-1}$ be the coordinate hyperplanes of
  the projective bundle. 
  It is an elementary fact about projective bundles that
  \[
    \textstyle
    \mathcal O_{\tilde{\mathrm{gl}}_{k}^*\mathfrak E_{k-1}}(-k) \cong
    \bigwedge^k \Big(
    ( \Theta_1\oplus\cdots \oplus \Theta_k ) \otimes \mathcal
    O_{\mathfrak U_k}\big(-\sum_{i=k}^{m-1} \mathfrak E_i \big)
    \Big) \otimes  \mathcal O_{\tilde{\mathrm{gl}}_{k}^*\mathfrak E_{k-1}}(-\sum_{i=1}^kH_i).
  \]
  From the construction of the inflated projective bundle it is easy to see that
  the pullback of $H_1 + \cdots + H_k$ to $\tilde{\mathrm{gl}}_k^* \mathfrak
  D_{k-1}$ is equal to $D_0 + 2D_1 + \cdots + (k-1)D_{k-2}$. Hence we obtain
  \[
    \textstyle
    \iota_{\mathfrak D}^*\mathcal O_{\tilde{\mathfrak
      M}_{g,n,d}}(k\mathfrak D_{k-1}) \cong \Theta_1 \otimes  \cdots \otimes
  \Theta_{k} \otimes \mathcal O_{\tilde{\mathrm{gl}}_{k}^*\mathfrak E_{k-1}}(-\sum_{i=k}^{m-1} \iota_{\mathfrak
    D}^*\mathfrak D_{i} - \sum_{i=0}^{k-2}(i+1)D_i).
  \]
  Combining this with Lemma~\ref{lem:divisors1}(1), we get the desired formula.
\end{proof}
\fi

\subsection{The calibration bundle}
\label{sec:cali-bundle}
To each semistable weighted curve, we will canonically associate a
$1$-dimensional vector space, on which the automorphism group of the degree-$d_0$
rational tails acts. We will also add infinity to it, forming a $\mathbb
P^1$-bundle over $\mathfrak M^{\mathrm{wt,ss}}_{g,n,d}$.

Recall that $\mathfrak Z_1\subset \mathfrak M_{g,n,d}^{\mathrm{wt,ss}}$ is the
reduced substack where the curve has at least one degree-$d_0$ rational tail.
Also recall that we have assumed that $2g-2+n+\epsilon_0 d\geq 0$, and
$\epsilon_0 d>2$ when $g=n=0$.
\begin{definition}
  \label{def:calibration-bundle}
  When $(g,n,d)\neq (0,1,d_0)$, the universal calibration bundle is defined to
  be the line
  bundle $\mathcal O_{\mathfrak
    M_{g,n,d}^{\mathrm{wt,ss}}}(-\mathfrak Z_1)$; when $(g,n,d) = (0,1,d_0)$, the
  universal calibration bundle is the relative cotangent bundle at the unique
  marking\footnote{This is viewed as a line bundle on $\mathfrak
    M^{\mathrm{wt,ss}}_{0,1,d_0}$ by the canonical section of the gerbe
    marking. See Section~\ref{sec:orbifold-psi-classes}.}.

  For an $S$-family of $\epsilon_0$-semistable, genus-$g$,
  degree-$d$ weighted curves, its calibration bundle is the pullback of universal
  calibration bundle along the classifying morphism $S \to  \mathfrak
  M_{g,n,d}^{\mathrm{wt,ss}}$.
\end{definition}
In particular, the calibration bundle of $\tilde{\mathfrak M}_{g,n,d}$ is the
pullback of the universal calibration bundle of $\mathfrak
M_{g,n,d}^{\mathrm{wt,ss}}$ via the forgetful morphism $\tilde{\mathfrak
  M}_{g,n,d} \to \mathfrak M_{g,n,d}^{\mathrm{wt,ss}}$.

We now focus on the case $(g,n,d)\neq (0,1,d_0)$.
Locally let $\mathfrak H_1 ,\ldots, \mathfrak H_k$ be the branches of
the normal crossing divisor $\mathfrak Z_1$. Over $\mathfrak H_i$, the universal
curve decomposes as a union
\[
  \mathcal C = \mathcal C_g \cup_{\Sigma} \mathcal C_0
\]
of genus-$g$ curves $\mathcal C_g$ and genus-$0$ curves $\mathcal C_0$
intersecting along the node $\Sigma$. Then $\mathcal O(\mathfrak
H_i)|_{\mathfrak H_i}$ is naturally isomorphic to
\[
  N_{\Sigma/\mathcal C_g} \otimes  N_{\Sigma/\mathcal C_0}.
\]
A priori, this is line bundle on $\Sigma$, which is a gerbe over $\mathfrak
H_i$. The balanced-node condition implies that it descends to $\mathfrak H_i$.

In particular, for a single curve $C$ with degree-$d_0$ rational tails
$E_1 ,\ldots, E_k$,
its calibration bundle is naturally isomorphic to $(\Theta_1 \otimes  \cdots
\otimes \Theta_k)^{\vee}$, where $\Theta_i$ is the one dimensional vector space of
infinitesimal smoothings of the unique node on $E_i$.

We denote the calibration bundle by $\mathbb M$ with a suitable subscript,
referring to the ``$\mathbb M$aster space'' to be introduced later.
\begin{definition}
  \label{def:moduli-with-calibrated-tails}
  The moduli of ($\epsilon_0$-)semistable curves with calibrated tails is
  defined to be
  \[
    M \tilde {\mathfrak M}_{g,n,d}  : = \mathbb P_{\tilde {\mathfrak
        M}_{g,n,d}}(\mathbb M_{\tilde {\mathfrak
        M}_{g,n,d}}\oplus \mathcal O_{\tilde {\mathfrak M}_{g,n,d}}),
  \]
  where $\mathbb M_{\tilde {\mathfrak M}_{g,n,d}}$ is the calibration bundle of
  $\tilde{\mathfrak M}_{g,n,d}$.
\end{definition}
Hence an $S$-point of $M \tilde {\mathfrak M}_{g,n,d}$ consists of
\[
  (\pi:\mathcal C\to S, \mathbf x ,e,N,v_1,v_2)
\]
where
\begin{itemize}
\item
  $(\pi:\mathcal C\to S, \mathbf x ,e)\in \tilde {\mathfrak M}_{g,n,d}(S)$;
\item
  $N$ is a line bundle on $S$;
\item
  $v_1\in \Gamma(S,\mathbb  M_{S} \otimes  N)$, $v_2 \in \Gamma(S,
  N)$ have no common zero, where $\mathbb M_S$ is the calibration bundle for the
  family of curves $\pi:\mathcal C\to S$.
\end{itemize}
For two families
\[
  (\pi:\mathcal C\to S, \mathbf x ,e,N,v_1,v_2) \quad \text{and} \quad
  (\pi^\prime:\mathcal C^\prime\to S^\prime, \mathbf x^\prime
  ,e^\prime,N^\prime,v^\prime_1,v^\prime_2),
\]
an arrow between them consists of a triple
\[
  (f, t, \varphi),
\] where
\begin{itemize}
\item
  $f:S\to S^\prime$ is a morphism of schemes;
\item
  $t: (\pi:\mathcal C\to S, \mathbf x ,e) \to f^*(\pi^\prime:\mathcal
  C^\prime\to S^\prime, \mathbf x^\prime ,e^\prime)$ is a $2$-morphism in
  $\tilde {\mathfrak M}_{g,n,d}(S)$;
\item
  $\varphi: N \to f^*N^\prime$ is an isomorphism of line bundles, such that
  the morphisms $1 \otimes  \varphi: \mathbb M_S \otimes N \to
  \mathbb M_{S} \otimes f^*N^\prime = f^* (\mathbb M_{S^\prime} \otimes
  N^\prime)$ and $\varphi$
   send $(v_1,v_2)$ to $(f^*v^\prime_1,f^*v^\prime_2)$.
\end{itemize}
\subsection{Limits of entanglements}
In this subsection let $(R,\mathfrak m)$ be a complete discrete valuation
$\mathbb C$-algebra with residue field $\mathbb C$ and fraction field $K$.

Consider a diagram of $1$-morphisms
\[
  \begin{tikzcd}
    \operatorname{Spec}K \arrow[d]\arrow[r,"f"] \arrow[d,"\iota"']&
    \tilde{\mathfrak M}_{g,n,d}\arrow[d,"\tau"]\\
    \operatorname{Spec}R \arrow[r,"g"] & \mathfrak M^{\mathrm{wt,ss}}_{g,n,d}
  \end{tikzcd}
\]
and a $2$-morphism
\[
  t: \tau\circ f \overset{\cong}{\longrightarrow} g\circ \iota.
\]
Concretely, $g$ corresponds to a family of weighted curves
  $\mathcal C\to \operatorname{Spec}R$, $\tau\circ f$
corresponds to a family of weighted curves
  $\mathcal C^*\to \operatorname{Spec}K$,
and $t$ corresponds to a fibered diagram
\[
  \begin{tikzcd}
    \mathcal C^* \arrow[r]\arrow[d]& \mathcal  C\arrow[d] \\
    \operatorname{Spec}K \arrow[r] &\operatorname{Spec}R
  \end{tikzcd}
\]
which is compatible with the degree assignment on the weighted curves.
Since $\tau$ is projective, by the valuative criterion for properness there is a morphism
\[
  h : \operatorname{Spec} R \longrightarrow \tilde{\mathfrak M}_{g,n,d}
\]
and $2$-morphisms
\[
  t_1: \tau\circ h \longrightarrow g,\quad
  t_2: f  \longrightarrow h\circ \iota,
\]
such that
\[
  (t_1\cdot \iota)\circ (\tau\cdot t_2) = t.
\]
Moreover the triple $(h,t_1,t_2)$ is unique up to obvious
isomorphisms that satisfy obvious commutativity relations.

Thus, given $f,g$ and $t$, the unique $h$ determines an unique ``limiting entanglement''
for the special fiber of $\mathcal C$. We now describe this
limiting entanglement in terms of $f,g$ and $t$.
\subsubsection{Case 1}
Assume that there are no degree-$d_0$ rational tails in $\mathcal C^*$. Let
$E_1,\ldots, E_\ell$ be the degree-$d_0$ rational tails in the special fiber of $\mathcal
C$.
Suppose that $E_i$ intersects the other components of the special fiber at the
node $p_i\in \mathcal C$. Suppose that $\mathcal C$ has
$A_{a_i}$-singularity at $p_i$. \footnote{If $p_i$ is an orbifold point, the
  singularity type is that of a \'etale neighborhood of $p_i$.}
\begin{lemma}
  \label{lem:limit-of-entanglement-case1}
  If
  \[
    a_1=\cdots=a_k, \quad k\leq \ell
  \]
  are the greatest among $a_1,\ldots,a_\ell$, then the tails
  $E_1,\ldots,E_k$ are entangled.
\end{lemma}
\ifdefined\SHOWPROOFS
\begin{proof}
  As in the proof of Lemma~\ref{lem:divisors2}, near the image of $g$
  we can write $\mathfrak Z_1$ as $\mathfrak H_1 \cup\cdots \cup \mathfrak
  H_\ell$. The map $g$ sends the closed point to $\bigcap_{i=1}^\ell\mathfrak H_i$ and
  the contact order with $\mathfrak H_i$ is $a_{i}+1$ (c.f.
  \cite[Remark~1.11]{olsson2007}). After each blow up,
  the contact order of the unique lift of $g$ and the proper transform of
  $\mathfrak H_i$ drops by $1$. Thus the Lemma follows immediately from
  Definition~\ref{def:entangled-tails}.
\end{proof}
\fi

\subsubsection{Case 2}
\label{sec:lim-of-entang2}
Let $\mathcal E_{1},\ldots,\mathcal E_{\ell }\subset \mathcal C$ be the
degree-$d_0$ rational tails whose restriction to $\mathcal C^*$ are all the
entangled tails defined by $f$. Thus $f$ factors through $\mathfrak
E^*_{\ell-1}$ (c.f.\ Lemma~\ref{lem:entangled-tails-from-gluing}) and $g$
factors through $\mathfrak Z_{(\ell)}$. Since $R$ is
complete, possibly after totally ramified finite base change\footnote{A base
  change might be needed due to the orbifold nodes.}, $g$ lifts to
\[
  \tilde g : \operatorname{Spec}R \longrightarrow \tilde{\mathfrak
    M}_{g,n+\ell,d-\ell d_0} \times^\prime {\big( \mathfrak M^{\mathrm{wt,ss}}_{0,1,d_0}\big)}^\ell,
\]
so that the rational tails coming from ${(\mathfrak
M^{\mathrm{wt,ss}}_{0,1,d_0})}^\ell $ are $\mathcal E_1,\ldots,\mathcal E_\ell$.

By Lemma~\ref{lem:structure-E_k-1} and the universal property of projective
bundles, $f$ corresponds to a line bundle $L$ on
$\operatorname{Spec}K$ and sections
\[
  s^*_i \in \Gamma\big(\operatorname{Spec} K, L^\vee \otimes  {(\tilde g\circ
      \iota)}^*\Theta_i \big), \quad i = 1 ,\ldots, \ell.
\]
Pick any trivialization for $L$ over $\operatorname{Spec}K$, we get nontrivial
rational sections
\[
  s_1,\ldots,s_\ell
\]
of $\tilde g^*\Theta_1,\ldots,\tilde g^*\Theta_\ell $ over
$\operatorname{Spec}R$.
For $i=1 ,\ldots, \ell$, let $a_i = \operatorname{ord}(s_i)$ be vanishing order
of $s_i$ at the close point,\footnote{If $s_i$ has a pole of order $b_i$, then
$a_i := -b_i < 0$.} and let $E_i$ be the special fiber of $\mathcal E_i$.
The following lemma is clearly independent of the choice of the trivialization
for $L$.
\begin{lemma}
  \label{lem:limit-of-entanglement-case2}
  Assume that
  \[
    a_1=\cdots=a_k,\quad k\leq \ell
  \]
  are the greatest among $a_1,\ldots,a_\ell$, then $E_1,\ldots, E_k$
  are entangled.
  In particular, the entangled tails in the special fiber form a subset of
  $\{E_1,\ldots,E_\ell\}$.
\end{lemma}
\ifdefined\SHOWPROOFS
\begin{proof}
  The proof is similar to that of Lemma~\ref{lem:limit-of-entanglement-case1}.
  We choose the trivialization of $L$ so that $a_i\geq 0$ for all $i$ and
  $a_i=0$ for at least one $i$.
  Then the integers $a_1,\ldots, a_k$ are the contact orders to the coordinate
  hyperplanes of the projective bundle ${\tilde{\mathrm{gl}}_k^*}\mathfrak E_{k-1} \to
  \tilde{\mathfrak M}_{g,n+k,d-kd_0} \times^\prime {( \mathfrak
      M^{\mathrm{wt,ss}}_{0,1,d_0})}^k$ (c.f.\ Lemma~\ref{lem:structure-E_k-1},\
    Lemma~\ref{lem:str-D}).
\end{proof}
\fi

\subsection{Limits of calibrations}
Suppose we have a diagram of $1$-morphisms
\[
  \begin{tikzcd}
    \operatorname{Spec}  K \arrow[d,"{\iota}"']  \arrow[rd,"f"]
 &   \\
 \operatorname{Spec}  R \arrow[r,shift left,"h_1"] \arrow[r,shift right,"h_2"']
  &  \mathfrak M_{g,n,d}^{\mathrm{wt,ss}}
  \end{tikzcd}
\]
and $2$-morphisms
\[
  h_1\circ \iota  \overset{\cong}{\longrightarrow} f
  \overset{\cong}{\longrightarrow} h_2\circ \iota.
\]
Concretely, this corresponds to two family of weighted curves
$\mathcal C_{h_1}$ and $\mathcal C_{h_2}$ over $\operatorname{Spec}  R$,  and a
$K$-isomorphism
\begin{equation}
  \label{isom:2-generic}
  \mathcal C_{h_1}|_{\operatorname{Spec}  K} \overset{\cong}{\longrightarrow}
  \mathcal C_{h_2}|_{\operatorname{Spec}  K}
\end{equation}
that is compatible with the degree assignment.
We assume that the set of irreducible components of the geometric fiber of $\mathcal
C_{h_1}|_{\operatorname{Spec}  K}$ do not have nontrivial monodromy. This can be
achieved by a finite base change.
Let $\mathbb M_{h_i} = h_i^*(\mathcal O_{\mathfrak M_{g,n,d}^{\mathrm{wt,ss}}}(-\mathfrak
Z_1))$ be the calibration bundle of the family $\mathcal C_{h_i}$.
The $K$-isomorphism of curves induces
\begin{equation}
  \label{isom:calib-generic}
  \mathbb M_{h_1} |_{\operatorname{Spec}  K} \overset{\cong}{\longrightarrow}
  \mathbb M_{h_2}|_{\operatorname{Spec}  K}.
\end{equation}
\begin{lemma}
  \label{lem:limit-calib-1}
  Suppose that in the special fiber, $\mathcal C_{h_1}$ has no degree-$d_0$
  rational tails and $\mathcal C_{h_2}$ has $k$ degree-$d_0$
  rational tails $E_1 ,\ldots, E_k$, $k\geq 0$. Suppose that the node on $E_i$ is a
  $A_{a_i-1}$-singularity of $\mathcal C_{h_2}$, $ i = 1 ,\ldots, k$, 
  then \eqref{isom:calib-generic} extends to an isomorphism
  \[
    \textstyle
    \mathbb  M_{h_1}
    \overset{\cong}{\longrightarrow}
    \mathbb M_{h_2}(\sum_{i=1}^k a_i x),
  \]
  where $x$ is the closed point of $\operatorname{Spec} R$.
\end{lemma}
\ifdefined\SHOWPROOFS
\begin{proof}
  Locally near $h_2(x)$, let $\mathfrak H_i\subset \mathfrak
  M_{g,n,d}^{\mathrm{wt,ss}}$ be the reduced locus where $E_i$ remains a
  rational tail.
  Thus locally $\mathfrak Z_1 = \sum_{i=1}^k \mathfrak H_i$ as Cartier divisors,
  and $a_i$ is the contact order of $h_2$ to $\mathfrak H_i$ at $x$. Let
  $1_{\mathfrak M}$ be the tautological rational section of $\mathcal O_{\mathfrak
    M_{g,n,d}^{\mathrm{wt,ss}}}(-\mathfrak Z_1)$, i.e.\ the section with a pole
  precisely along $\mathfrak Z_1$. Then $h_1^*1_{\mathfrak M}$ is
  regular and nonvanishing and $h_2^*1_{\mathfrak M}$ has a pole to the order
  $\sum_{i=1}^ka_i$ at $x$. The isomorphism \eqref{isom:calib-generic} takes
  $h_1^*1_{\mathfrak M}|_{\operatorname{Spec} K}$ to $h_2^*1_{\mathfrak
    M}|_{\operatorname{Spec} K}$. Thus the Lemma follows.
\end{proof}
\fi
\section{Quasimap invariants with entangled tails}
\subsection{The moduli and its virtual fundamental class}
\label{sec:quasimap-with-entangled-tails-moduli}

We fix $g,n$ and a curve class $\beta$. Let $d = \deg(\beta)$.
Let $\mathfrak{Qmap}_{g,n}(X,\beta)$ be the stack of genus-$g$, $n$-marked
quasimaps to $X$ with curve class $\beta$.
Let $\mathfrak{Qmap}^{\mathrm{ss}}_{g,n}(X,\beta)\subset
\mathfrak{Qmap}_{g,n}(X,\beta)$ be the open substack where the quasimap
has no
\begin{itemize}
\item
  rational tail of degree $<d_0$ or 
  rational bridge of degree $0$,
\item
  base point of length $>d_0$.
\end{itemize}
For a family of quasimaps, the degree assignment induced by the degree of the
quasimaps on each irreducible
component is continuous, in the sense that it is the degree of a line
bundle. 
Taking the underlying curves weighted by the degree of the quasimaps defines a
forgetful morphism
\[
  \mathfrak{Qmap}^{\mathrm{ss}}_{g,n}(X,\beta) \longrightarrow {\mathfrak
    M}_{g,n,d}^{\mathrm{wt,ss}}.
\]
\begin{definition}
  \label{def:semistable-quasimaps-entangled-tails}
  We define the stack of genus-$g$, $n$-pointed, $\epsilon_0$-semistable
  quasimaps \textit{with entangled tails} to $X$ with curve class $\beta$ to be
  \[
    \mathfrak{Qmap}^{\sim}_{g,n}(X,\beta) :=
    \mathfrak{Qmap}^{\mathrm{ss}}_{g,n}(X,\beta) \times_{{\mathfrak
        M}_{g,n,d}^{\mathrm{wt,ss}}} \tilde{\mathfrak M}_{g,n,d}.
  \]
\end{definition}
For any scheme $S$, 
we denote an $S$-point of $\mathfrak{Qmap}^{\sim}_{g,n}(X,\beta)$ by 
\[
  (\pi:\mathcal C \to S,\mathbf x,e,u),
\]
where $(\pi:\mathcal C \to S,\mathbf x, u) \in
\mathfrak{Qmap}^{\mathrm{ss}}_{g,n}(X,\beta)(S)$ and $(\pi:\mathcal C \to
S,\mathbf x, e) \in \tilde{\mathfrak M}_{g,n,d}(X,\beta)(S)$.
By definition, 
  an arrow
  \[
    (\pi_1:\mathcal C_1\to S_1,\mathbf x_1,e_1,u_1) \longrightarrow (\pi_2:
    \mathcal C_2\to S_2,\mathbf x_2,e_2,u_2)
  \]
  between two families of $\epsilon_0$-semistable quasimaps with entangled tails
  consists of a morphisms $f:S_1\to S_2$, a $2$-morphism $e_1
  \to e_2\circ f$, and a $2$-morphism from $\mathcal C_1\overset{u_1}{\to}
  [W/G]$ to $\mathcal C_1 \to \mathcal C_2 \overset{u_2}{\to} [W/G]$.
Note that 
  since the composition of forgetful morphisms $\tilde{\mathfrak M}_{g,n,d} \to
  {\mathfrak M}^{\mathrm{wt,ss}}_{g,n,d} \to {\mathfrak
    M}_{g,n}$ is representable, the $2$-morphism $e_1 \to
  e_2\circ f$
  is completely determined by the isomorphism $\mathcal C_1 {\to}f^*\mathcal
  C_2$ of underlying curves.
\begin{lemma}
  \label{lem:artin-stack}
  The stack $\mathfrak {Qmap}^{\sim}_{g,n}(X,\beta)$  is an Artin stack
  of finite type, with finite-type separated diagonal.
\end{lemma}
\ifdefined\SHOWPROOFS
\begin{proof}
  By construction, the morphism $\tilde{\mathfrak M}_{g,n,d} \to \mathfrak
  M_{g,n,d}^{\mathrm{wt,ss}}$ is projective. 
  By \cite{ciocan2014stable, cheong2015orbifold}, 
  $\mathfrak{Qmap}_{g,n}(X,\beta)$ is an Artin stack locally of finite
  type with finite-type separated
  diagonal. Hence so is $\mathfrak {Qmap}^{\sim}_{g,n}(X,\beta)$.

  To show that $\mathfrak {Qmap}^{\sim}_{g,n}(X,\beta)$ is of finite type, it suffices
  to show that $\mathfrak{Qmap}^{\mathrm{ss}}_{g,n}(X,\beta)$ is bounded.
  For any object in $\mathfrak{Qmap}^{\mathrm{ss}}_{g,n}(X,\beta)$,
  there are no rational bridge of degree $0$ and no rational tail of degree $<d_0$.
  Hence the number of irreducible components is smaller than $N$, for some $N>0$.
  Let $U\subset Q^{0+}_{g,n+N}(X,\beta)$ be the open substack defined by the
  following conditions:
  \begin{enumerate}
  \item
    the last $N$ markings are nonorbifold markings;
  \item
    after forgetting the last $N$ markings (without stabilization), the
    resulting quasimap lies in $\mathfrak{Qmap}^{\mathrm{ss}}_{g,n}(X,\beta)$.
  \end{enumerate}
  Then forgetting the last $N$ markings defines a surjection $U\to
  \mathfrak{Qmap}^{\mathrm{ss}}_{g,n}(X,\beta)$.
  By \cite[\S2.4.3]{cheong2015orbifold} (c.f.\ \cite[\S3.2]{ciocan2014stable}),
  $Q^{0+}_{g,n+N}(X,\beta)$ is of finite type. Hence $U$ is of finite type and thus
  $\mathfrak{Qmap}^{\mathrm{ss}}_{g,n}(X,\beta)$ is also of finite type. This
  completes the proof.
\end{proof}
\fi

\begin{definition}
  An $S$-family of $\epsilon_0$-semistable quasimaps with entangled tails is
  $\epsilon_+$-stable if the underlying family of quasimaps is
  $\epsilon_+$-stable. In other words, it is $\epsilon_+$-stable if there is no
  length-$d_0$ base point.
\end{definition}

We denote by $\tilde Q^{\epsilon_+}_{g,n}(X,\beta)$ the moduli of
genus-$g$, $n$-marked
$\epsilon_+$-stable quasimaps to $X$ with entangled tails of curve class $\beta$.
As in Definition~\ref{def:semistable-quasimaps-entangled-tails}, 
we have a natural isomorphism
\[
  \tilde Q^{\epsilon_+}_{g,n}(X,\beta) \cong Q^{\epsilon_+}_{g,n}(X,\beta)
  \underset{\mathfrak M_{g,n,d}^{\mathrm{wt,ss}}}{\times} \tilde{\mathfrak
    M}_{g,n,d}.
\]
Since $Q^{\epsilon_+}_{g,n}(X,\beta)$ is a proper Deligne--Mumford stack, we obtain
\begin{lemma}
  The stack $\tilde{Q}^{\epsilon_+}_{g,n}(X,\beta)$ is a proper Deligne--Mumford stack.
\end{lemma}

Recall from \cite{ciocan2014stable, cheong2015orbifold} that
$Q^{\epsilon_+}_{g,n}(X,\beta)$ has a relative perfect obstruction theory
relative to the moduli $\mathfrak M_{g,n}$ of twisted curves. Since the morphism
$\mathfrak M_{g,n,d}^{\mathrm{wt,ss}} \to \mathfrak M_{g,n,d}$ forgetting the
degree assignment is \'etale, the construction also produces a relative
perfect obstruction theory relative to $\mathfrak M_{g,n,d}^{\mathrm{wt,ss}}$.
The induced virtual fundamental classes are the same, by
\cite[Theorem~5.0.1]{costello2006higher}(c.f. \cite{hu2010genus}).
We briefly recall the construction. Let $\pi: \mathcal C \to
Q^{\epsilon_+}_{g,n}(X,\beta)$ be the universal curve and $u: \mathcal C \to
[W/G]$ be the universal quasimap. Then we have a relative perfect obstruction
theory
\begin{equation}
  \label{eq:pot-Q}
  (R\pi_*u^*\mathbb T_{[W/G]})^{\vee}
  \longrightarrow
  \mathbb L_{Q^{\epsilon_+}_{g,n}(X,\beta)/ \mathfrak M_{g,n,d}^{\mathrm{wt,ss}}},
\end{equation}
inducing the virtual fundamental class $[Q^{\epsilon_+}_{g,n}(X,\beta)
]^{\mathrm{vir}} \in A_*(Q^{\epsilon_+}_{g,n}(X,\beta) )$.

Similarly, let $\tilde\pi:\tilde{\mathcal C} \to \tilde
Q^{\epsilon_+}_{g,n}(X,\beta)$ be the universal curve and $\tilde u:
\tilde{\mathcal C} \to [W/G]$ be the universal quasimap.
We have a relative perfect obstruction theory
\[
  (R\tilde\pi_*\tilde u^*\mathbb T_{[W/G]})^{\vee}
  \longrightarrow
  \mathbb L_{\tilde Q^{\epsilon_+}_{g,n}(X,\beta)/\tilde{\mathfrak M}_{g,n,d}}.
\]
Indeed, $R\tilde\pi_*\tilde u^*\mathbb T_{[W/G]}$ is the pullback of
$R\pi_*u^*\mathbb T_{[W/G]}$ via the forgetful morphism
$\tilde Q^{\epsilon_+}_{g,n}(X,\beta) \to Q^{\epsilon_+}_{g,n}(X,\beta)$.
This defines the virtual fundamental class 
\[
  [ \tilde Q^{\epsilon_+}_{g,n}(X,\beta) ]^{\mathrm{vir}} \in A_*(
  \tilde Q^{\epsilon_+}_{g,n}(X,\beta) ).
\]
\begin{lemma}
  \label{lem:vir-cycle-comp-entangled-tails}
  Under the forgetful morphism $\tilde Q^{\epsilon_+}_{g,n}(X,\beta)\to
  Q^{\epsilon_+}_{g,n}(X,\beta)$,
  the pushforward of $[ \tilde Q^{\epsilon_+}_{g,n}(X,\beta)]^{\mathrm{vir}}$ is
  equal to $[ Q^{\epsilon_+}_{g,n}(X,\beta)]^{\mathrm{vir}}$.
\end{lemma}
\ifdefined\SHOWPROOFS
\begin{proof}
  This follows immediately from \cite[Theorem~5.0.1]{costello2006higher}.
\end{proof}
\fi

\subsection{Splitting off entangled tails}
Recall that for $k=1 ,\ldots, m$, the boundary divisors $\mathfrak
D_{k-1}\subset \tilde{\mathfrak M}_{g,n,d}$ is the closure of the locally closed
reduced substack where there are exactly $k$ entangled tails
(Definition~\ref{def:D_k-1}). We now study the Gysin pullback of the virtual
fundamental class to the preimage of $\mathfrak D_{k-1}$. More specifically,
recall that ${\tilde{\mathrm{gl}}_k}$ is the gluing morphism
\eqref{map:quotient-by-S}.
Define
$\tilde Q^{\epsilon_+}_{g,n}(X,\beta)|_{{\tilde{\mathrm{gl}}_k^*}\mathfrak
  D_{k-1}}$
by the following fibered diagram
\begin{equation}
  \label{diag:splitting-tails}
  \begin{tikzcd}
    \tilde Q^{\epsilon_+}_{g,n}(X,\beta)|_{{\tilde{\mathrm{gl}}_k^*}\mathfrak
      D_{k-1}} \arrow[d]\arrow[r]
    & \tilde Q^{\epsilon_+}_{g,n}(X,\beta) \arrow[d]\arrow[r]&
    Q^{\epsilon_+}_{g,n}(X,\beta)\arrow[d]\\
    {\tilde{\mathrm{gl}}_k^*}\mathfrak D_{k-1}\arrow[r,"{\iota_{\mathfrak D}}"]
    & \arrow[r] \tilde{\mathfrak M}_{g,n,d}
    & {\mathfrak M}^{\mathrm{wt,ss}}_{g,n,d}
  \end{tikzcd}.
\end{equation}
We define
\[
  [
  \tilde Q^{\epsilon_+}_{g,n}(X,\beta)|_{{\tilde{\mathrm{gl}}_k^*}\mathfrak
    D_{k-1}}
  ]^{\mathrm{vir}}
  := \iota_{\mathfrak D}^!([\tilde Q^{\epsilon_+}_{g,n}(X,\beta)]^{\mathrm{vir}}).
\]
The map ${\tilde{\mathrm{gl}}_k^*}\mathfrak D_{k-1} \to \mathfrak
M^{\mathrm{wt,ss}}_{g,n,d}$ factors as the composition of morphisms
(c.f. Lemma~\ref{lem:str-D})
\begin{equation}
  \label{eq:level-of-bdry}
  {\tilde{\mathrm{gl}}_k^*}\mathfrak D_{k-1}
  \longrightarrow
  \tilde{\mathfrak M}_{g,n+k,d-kd_0} \times^\prime \big(\mathfrak
  M_{0,1,d_0}^{\mathrm{wt,ss}}\big)^k
  \longrightarrow
  {\mathfrak M}^{\mathrm{wt,ss}}_{g,n+k,d-kd_0} \times^\prime
  \big(\mathfrak M_{0,1,d_0}^{\mathrm{wt,ss}}\big)^k
  \longrightarrow
  \mathfrak M^{\mathrm{wt,ss}}_{g,n,d}.
\end{equation}
Recall that $\times^\prime$ means the fiber product over $\mathbb N^{k}$
matching the orders of automorphism groups at the last $k$-markings.
We have the obvious fibered diagram
\[
  \begin{tikzcd}
    \coprod_{\vec\beta} Q^{\epsilon_+}_{g,n+k}(X,\beta_0)\times_{(I_{\mu}X)^k}
    \prod_{i=1}^k Q^{\epsilon_+}_{0,1}(X,\beta_i)\arrow[d]\arrow[r] & \arrow[d]
    Q^{\epsilon_+}_{g,n}(X,\beta) \\
    {\mathfrak M}^{\mathrm{wt,ss}}_{g,n+k,d-kd_0} \times^\prime
    \big(\mathfrak M_{0,1,d_0}^{\mathrm{wt,ss}}\big)^k
    \arrow[r] & {\mathfrak M}^{\mathrm{wt,ss}}_{g,n,d}
  \end{tikzcd} ,
\]
where $\vec\beta = (\beta_0,\beta_1 ,\ldots, \beta_k)$ runs through all
the decomposition of effective curve classes
\[
  \beta = \beta_0 + \cdots + \beta_k
\]
such that $\deg(\beta_i) = d_0$ for each $i = 1 ,\ldots, k$.

Taking fiber products along the arrows in \eqref{eq:level-of-bdry} gives us the
fibered diagram
\begin{equation}
  \label{eq:splitting-nodes}
  \begin{tikzcd}
    \tilde Q^{\epsilon_+}_{g,n}(X,\beta)|_{{\tilde{\mathrm{gl}}_k^*}\mathfrak D_{k-1}}
    \arrow[d] \arrow[r,"{p}"] & \coprod_{\vec \beta} \tilde
    Q^{\epsilon_+}_{g,n+k}(X,\beta_0)\times_{(I_{\mu}X)^k} \prod_{i=1}^k
    Q^{\epsilon_+}_{0,1}(X,\beta_i)
    \arrow[d] \\
    {\tilde{\mathrm{gl}}_k^*}\mathfrak D_{k-1} \arrow[r]
    & \tilde{\mathfrak M}_{g,n+k,d-kd_0} \times^\prime
    \big(\mathfrak M_{0,1,d_0}^{\mathrm{wt,ss}}\big)^k
  \end{tikzcd}.
\end{equation}
By Lemma~\ref{lem:str-D}, the map $p$ above
is the inflated projective bundle $\tilde{\mathbb P}(\Theta_1\oplus \cdots \oplus
\Theta_k)$, where $\Theta_i$ is line bundle formed by the tensor product of the
relative tangent space at the $(n+i)$-th marking of
$\tilde{Q}^{\epsilon_+}_{g,n+k}(X,\beta_0)$, and the relative tangent space at the
unique marking of $Q^{\epsilon_+}_{0,1}(X,\beta_i)$. In particular $p$ is flat.
\begin{lemma}
  \label{lem:int-with-D-k-1}
  \begin{equation}
  \label{eq:int-with-D-k-1}
  \textstyle
    [\tilde Q^{\epsilon_+}_{g,n}(X,\beta)|_{{\tilde{\mathrm{gl}}_k^*}\mathfrak
      D_{k-1}}]^{\mathrm{vir}}
     = p^*\big( \sum_{\vec\beta} \Delta_{ (I_{\mu} X)^k}^!
     [\tilde{Q}^{\epsilon_+}_{g,n+k}(X,\beta_0)]^{\mathrm{vir}}\boxtimes
     \prod_{i=1}^k[{Q}^{\epsilon_+}_{0,1}(X,\beta_i)]^{\mathrm{vir}}\big),
  \end{equation}
  where $\Delta_{ (I_{\mu} X)^k} : { (I_{\mu} X)^k} \to { (I_{\mu} X)^k} \times
  { (I_{\mu} X)^k}$ is the diagonal morphism.
\end{lemma}
\ifdefined\SHOWPROOFS
\begin{proof}
  We have introduced the relative perfect obstruction theory on
  $Q^{\epsilon_+}_{g,n}(X,\beta)$, relative to ${\mathfrak
    M}^{\mathrm{wt,ss}}_{g,n,d}$. We denote it by $\mathbb E$. Consider the pullback
  of $Q^{\epsilon_+}_{g,n}(X,\beta) \to {\mathfrak M}^{\mathrm{wt,ss}}_{g,n,d}$
  along the sequence of maps \eqref{eq:level-of-bdry}. The pullback of $\mathbb E$
  to each space defines a relative perfect obstruction theory. By a standard
  argument, as in the proof of the splitting node axiom in Gromov--Witten theory,
  \[
    \textstyle
    \sum_{\vec\beta} \Delta_{(I_{\mu}X)^k}^!
    [\tilde{Q}^{\epsilon_+}_{g,n+k}(X,\beta_0)]^{\mathrm{vir}}\boxtimes \prod_{i=1}^k[
    {Q}^{\epsilon_+}_{0,1}(X,\beta_i)]^{\mathrm{vir}}
  \]
  is equal to the virtual cycle defined by the pullback of $\mathbb E$,
  relative to
  $\tilde{\mathfrak M}_{g,n+k,d-kd_0} \times^\prime (\mathfrak
  M_{0,1,d_0}^{\mathrm{wt,ss}})^k$.
  The proof is parallel to the proof of
  \cite[Proposition~5.3.1]{abramovich2008gromov}
  \footnote{We are using $I_{\mu}X$ instead of $\overline{I}_{\mu}X$. Hence we
  need to modify \cite[Lemma~3.6.1]{abramovich2008gromov}. Using the notation
  there, if we have a section $S\to \mathcal G$, then we obtain $\tilde{f}:S \to
  \mathcal I_{\mu}(\mathcal X)$. Since $\varpi^*(T_{\overline{\mathcal I}_\mu
    (\mathcal X)}) \cong T_{\mathcal I_{\mu} (\mathcal X)}$, that lemma
  implies $\pi_*(F^*(T_{\mathcal X})) \cong \tilde{f}^*(T_{\mathcal
    I_{\mu}(\mathcal X)})$.}. See also 
  \cite[Proposition~7.2]{behrend97_intrin_normal_cone}, and
  \cite{behrend1997gromov, li1998virtual}.

  Since the map ${\tilde{\mathrm{gl}}_k^*}\mathfrak D_{k-1} \to  \tilde{\mathfrak
    M}_{g,n+k,d-kd_0} \times^\prime (\mathfrak M_{0,1,d_0}^{\mathrm{wt,ss}})^k$
  is flat, the discussion above shows that the right hand side
  of \eqref{eq:int-with-D-k-1} is the virtual cycle defined by the pullback of
  $\mathbb E$, as a relative perfect obstruction theory relative to
  ${\tilde{\mathrm{gl}}_k^*}\mathfrak D_{k-1}$
  \cite[Proposition~7.2]{behrend1997gromov}. Similarly, the left hand side of
  \eqref{eq:int-with-D-k-1} is also defined by the pullback of $\mathbb
  E$. Hence they must be equal.
\end{proof}
\fi
\section{The master space and its virtual fundamental class}
\label{sec:master-space}
We now define a master space relating
\begin{enumerate}
\item $\epsilon_+$-stable quasimaps with entangled tails;
\item $\epsilon_-$-stable quasimaps.
\end{enumerate}
The objects parameterized by it will be called $\epsilon_0$-stable quasimaps
with \textit{calibrated} tails.

\subsection{The definition of the master space}
Fix $g,n,d,d_0 = 1/\epsilon_0$ as before. Note that we have assumed that
$2g-2+n+\epsilon_0 d\geq 0$, and $\epsilon_0 d>2$ when $g=n=0$.

Let $\mathbb M_{\tilde {\mathfrak M}_{g,n,d}}$ be the calibration bundle of
$\tilde {\mathfrak M}_{g,n,d}$ (Definition~\ref{def:calibration-bundle}). Recall
from Definition~\ref{def:moduli-with-calibrated-tails} the moduli $M\tilde
{\mathfrak M}_{g,n,d}$ of curves with calibrated tails.

\begin{definition}
  \label{def:ss-quasimap-with-calibrated-tails}
  The moduli of genus-$g$, $n$-marked
  $\epsilon_0$-semistable quasimaps with calibrated tails to $X$
  of
  curve class $\beta$ is defined to be
  \[
    M \mathfrak {Qmap}^{\sim}_{g,n}(X,\beta): = \mathfrak
    {Qmap}^{\sim}_{g,n}(X,\beta) \times_{\mathfrak {\tilde M}_{g,n,d}}
    M\mathfrak {\tilde M}_{g,n,d}.
  \]
\end{definition}

By Lemma~\ref{lem:artin-stack}, this is an Artin stack of finite type with
finite-type separated diagonal.
Let $S$ be any scheme. We write
\[
  (\pi:\mathcal C\to S, \mathbf x ,e,u,N,v_1,v_2) \in
  M \mathfrak {Qmap}^{\sim}_{g,n}(X,\beta)(S)
\]
where
$(\pi:\mathcal C\to S, \mathbf x ,e,u)\in \mathfrak
{Qmap}^{\sim}_{g,n}(X,\beta)$
and
$(\pi:\mathcal C\to S, \mathbf x ,e,N,v_1,v_2) \in M\tilde {\mathfrak
    M}_{g,n,d}(S)$.

We now come to the stability condition.
Let $(C,\mathbf x,e,u)$ be a geometric point of $\mathfrak
{Qmap}^{\sim}_{g,n}(X,\beta)$.
\begin{definition}
  \label{def:constant-tail}
  A degree-$d_0$ rational tail $E\subset C$ is called a \textit{constant tail}
  if $E$ contains a base point of length $d_0$.
\end{definition}
In other words, $E$ is a constant tail if and only if
the rational map $E\dasharrow \underline{X}$ induced by $u$ is a constant map.
\begin{definition}
  \label{def:mspace-stability}
  An $S$-family of $\epsilon_0$-semistable quasimaps with calibrated tails
  \[
    (\pi:\mathcal C\to S, \mathbf x ,e,u,N,v_1,v_2)
  \]
  is $\epsilon_0$-stable if over every geometric point $s$ of $S$,
  \begin{enumerate}
  \item
    any constant tail is an entangled tail;
  \item
    if the geometric fiber $\mathcal C_s$ of $\mathcal C$ has rational tails of
    degree $d_0$, then those rational tails contain all the length-$d_0$ base
    points;
  \item
    if $v_1(s)=0$, then  $(\pi:\mathcal C\to S,\mathbf x, u)|_{s}$ is an
    $\epsilon_+$-stable quasimap;
  \item
    if $v_2(s)=0$, then $(\pi:\mathcal C\to S,\mathbf x,u)|_{s}$ is an
    $\epsilon_-$-stable quasimap.
  \end{enumerate}
\end{definition}
Note that constant tails are allowed neither in $Q^{\epsilon_+}_{g,n}(X,\beta)$
nor in ${Q}^{\epsilon_-}_{g,n}(X,\beta)$. Hence they are allowed only when
$v_1\neq 0$ and $v_2\neq 0$.

We denote by $MQ^{\epsilon_0}_{g,n}(X,\beta)$ the category of genus-$g$,
$n$-marked, $\epsilon_0$-stable quasimaps with calibrated tails to $X$ of curve
class $\beta$. Here $M$ means ``{$M$aster} space''.
\begin{lemma}
  \label{lem:open-condition}
  The stability condition above is an open condition.
\end{lemma}
\ifdefined\SHOWPROOFS
\begin{proof}
  It is clear that (3) and (4) are open conditions. It is easy to see
  that (1) and (2) are constructible conditions. Thus we can use the valuative
  criterion for openness.

  We first show that (1) is an open condition. Take any discrete valuation
  $\mathbb C$-algebra
  $R$ with residue field $\mathbb C$ and
  consider any semistable family
  \[
    (\pi:\mathcal C\to \operatorname{Spec} R, \mathbf x ,e,u,N,v_1,v_2).
  \]
  Up to finite base change, let $\mathcal E_1,\ldots,\mathcal E_\ell
  \to\operatorname{Spec}R$ be the degree-$d_0$ rational tails coming from the
  generic fiber, among which $\mathcal E_1,\ldots,\mathcal E_{k}$ are the
  entangled tails of the generic fiber ($k\leq \ell$). Let $E_i$ be the special
  fiber of $\mathcal E_i$ for $i=1,\ldots,\ell$. Suppose the generic fiber does
  not satisfy (1), say $\mathcal E_{k+1}$ is a constant but not entangled tail in
  the generic fiber. Then $E_{k+1}$ is a constant tail. However, by
  Lemma~\ref{lem:limit-of-entanglement-case2}, $E_{k+1}$ is not entangled. Thus
  (1) is violated by the special fiber. Hence the valuative criterion is satisfied
  and (1) is an open condition.

  We now come to (2). Although (2) itself is not an open condition in general,
  we will show that (2) is an open condition on the open subset where (1) holds
  true. In the same setting as above, suppose that the generic fiber does not
  satisfy (2), i.e.\ there is a length-$d_0$ base point in the generic fiber that
  does not lie on $\mathcal E_1,\ldots,\mathcal E_\ell$. The limit of that will be
  a length-$d_0$ base point in the special fiber that does not lie on
  $E_1,\ldots,E_\ell$. Then in the special fiber either (2) is violated or there
  is a new degree-$d_0$ rational tail $E^\prime$ that contains a length-$d_0$ base
  point. In the latter case, $E^\prime$ is not entangled by
  Lemma~\ref{lem:limit-of-entanglement-case2}. Thus (1) is violated. This prove
  that (2) is an open condition where (1) holds true. This completes the proof.
\end{proof}
\fi

Consider an $\epsilon_0$-stable quasimap with calibrated tails
\[
  \xi = (\pi:C \to S, \mathbf x ,e,u,N,v_1,v_2) \in M \mathfrak
  {Qmap}^{\sim}_{g,n}(X,\beta)(\mathbb C).
\]
Let $E\subset C$ be a degree-$d_0$ tail and $y\in E$ be the node (or marking if
$g=0,n=1$).
\begin{definition}
  \label{def:fixed-tails}
  We call $E$ a \textit{fixed tail} if $\mathrm{Aut}(E,y,u|_E)$ is infinite.
\end{definition}
\begin{lemma}
  \label{lem:fixed-tail}
  Suppose $E$ is a fixed tail, then $E$ is a constant tail. Moreover,
  $\mathrm{Aut}(E,y,u|_E)$  is $1$-dimensional, acting nontrivially on the
  tangent space $T_yE$.
\end{lemma}
\ifdefined\SHOWPROOFS
\begin{proof}
  Suppose that $E$ is a fixed tail. Then $(E,y,u|_{E})$ cannot be an
  $\epsilon_+$-stable quasimap.
  Hence it must have a unique base point $z\in E$ of length $d_0$. Hence it is a
  constant tail. Moreover, the forgetful map $\mathrm{Aut}(E,y,u|_{E}) \to
  \mathrm{Aut}(E,y)$ factors through $\mathrm{Aut}(E,y,z)\cong \mathbb C^*$ and
  has a finite
  kernel (c.f.\ the proof of \cite[Proposition~7.1.5]{ciocan2014stable} and
  \cite[\textsection2.4.2]{cheong2015orbifold}). This
  proves the second statement.
\end{proof}
\fi
\begin{remark}
  Indeed, the automorphism group of $(E,y)$ as a pair of Deligne--Mumford stacks only
  acts on $[T_yE/\mathrm{Aut}(y)]$. The automorphism group $(E,y)$ as a curve
  with a trivialized-gerbe-marking does act on $T_yE$. The distinction is not
  important here, since we only care about whether the automorphism group is finite.
\end{remark}

We write
\[
  C = C^\prime\cup E_1\cup\cdots\cup E_k,
\]
where $E_1 ,\ldots, E_k$ are all the entangled tails, $C^\prime$ is the union
of the other irreducible components. Let $y_i = C^\prime\cap E_i$, $i=1 ,\ldots, k$.
\begin{lemma}
  \label{lem:auto-eta}
  The family of quasimaps with entangled tails
  \[
    \eta = (\pi:C\to \operatorname{Spec} \mathbf k, \mathbf x ,e,u)
  \]
  coming from $\xi$ has infinitely many automorphisms if and only if
  \begin{enumerate}
  \item there is at least one degree-$d_0$ tail, and
  \item each entangled tail is a fixed tail.
  \end{enumerate}
  Moreover,  when $\mathrm{Aut}(\eta)$ is infinite, let $\Gamma\subset
  \mathrm{Aut}(\eta)$ be the identity component, then $\Gamma\cong \mathbb C^*$
  acting
 on each $T_{y_i}E_{i}$ by the same nonzero weight for $i=1 ,\ldots, k$.
\end{lemma}
\begin{remark}
  \label{rmk:tails-comparison}
  Note that (2) is equivalent to that for any degree-$d_0$ tail
  $E\subset C$,
  \[
    \text{$E$ is a fixed tail} \iff
    \text{$E$ is a constant tail}  \iff
    \text{$E$ is an entangled tail}.
  \]
  While for a general $\epsilon_0$-stable $\xi$
  we only have
 the
  direction ``$\implies$''.
\end{remark}
\ifdefined\SHOWPROOFS
\begin{proof}[Proof of Lemma~\ref{lem:auto-eta}]
  We first make some general observation. By definition, $\Gamma$ is connected.
  Hence it takes each irreducible component of $C$ to itself. By the stability
  condition, $C^\prime$ contains no constant tail. Hence $(C^\prime, \mathbf x,y_1
  ,\ldots, y_k, u|_{C^\prime})$ is an $\epsilon_+$-stable quasimap, which has
  finitely many automorphisms. Hence $\Gamma$ acts trivially on $C^\prime$. In
  particular it acts trivially on each $T_{y_i}C^\prime$. Since $\tilde{\mathfrak
    M}_{g,n,d}\to \mathfrak M_{g,n,d}^{\mathrm{wt,ss}}$ is representable, the
  automorphisms of $\eta$ are the automorphisms of the quasimap $(C ,\mathbf x,u)$
  that fixes $e(\operatorname{Spec} \mathbb C)$, viewed as a point in the fiber of
  $\tilde{\mathfrak M}_{g,n,d}\to \mathfrak M_{g,n,d}^{\mathrm{wt,ss}}$. By
  Lemma~\ref{lem:structure-E_k-1}, this implies $\Gamma$ acts trivially on
  $\mathbb P\left( T_{y_1}E_1 \oplus \cdots \oplus T_{y_k}E_k\right)$. Hence
  $\Gamma$ acts on $T_{y_i}E_i$ by the same character of $\Gamma$, for $i=1
  ,\ldots, k$. By Lemma~\ref{lem:fixed-tail}, $E_i$ is a fixed tail if and only if
  $\Gamma$ acts nontrivially on $T_{y_i}E_i$.

  We now prove the ``only if'' part. If there is no rational tail of degree
  $d_0$, then $C=C^\prime$ and we have shown that $\mathrm{Aut}(\eta)$ is finite
  in this case. Now suppose that some entangled tail $E_{i_0}$ is not a fixed
  tail, then $\Gamma$ acts trivially on $T_{y_{i_0}}E_{i_0}$. Hence $\Gamma$
  acts trivially on each $T_{y_{i}}E_{i}$. Hence $\Gamma$ acts trivially on each
  $E_i$. Since $\Gamma$ acts on $C^\prime$ trivially, $\Gamma$ is trivial and
  $\mathrm{Aut}(\eta)$ is finite.

  We then prove the ``if'' part. Suppose all the $E_i$'s are fixed tails.
  Let $\mathbb C^*$ act on each $E_i$ such that the weights on $T_{y_i}E_i$ are the
  same, for $i=1 ,\ldots, k$. Then it preserves $e(\operatorname{Spec} \mathbb C)$. By
  Lemma~\ref{lem:fixed-tail}, it preserves $u|_{E_{i}}$ for each $i$. Hence $\mathbb
  C^*$ is contained in the image of $\mathrm{Aut}(\eta) \to
  \mathrm{Aut}(C,\mathbf x)$. Thus $\mathrm{Aut}(\eta)$ must be infinite.
\end{proof}
\fi
\begin{lemma}
  \label{lem:finite-auto}
  Suppose $\xi$ is $\epsilon_0$-stable. Then $\xi$ has finitely many automorphisms.
\end{lemma}
\ifdefined\SHOWPROOFS
\begin{proof}
  Let $\eta, \Gamma$ be as in Lemma~\ref{lem:auto-eta}.
  Since $(v_1,v_2)\neq \vec 0$, the restriction map $\mathrm{Aut}(\xi) \to
  \mathrm{Aut}(\eta)$ is injective. Thus we only need to consider the case when
  $\mathrm{Aut}(\eta)$ is infinite. In this case,  $C$ has at least one constant
  tails. Hence we must have
  $v_1\neq 0$ and $v_2\neq 0$. It suffices to show that the subgroup of
  $\Gamma$ that fixes $v_1/v_2$ is finite. We write
  \[
    C = C_g \cup E_1 \cup\cdots\cup E_{k+\ell},
  \]
  where $E_1, \ldots, E_k$ are the entangled tails as before, $E_{k+1}, \ldots,
  E_{k+\ell}$ are the other degree-$d_0$ tails, and $C_g$ is the union of all
  other irreducible components. Let $y_i = C_g\cap E_i$ for $i=1
  ,\ldots, k+\ell$. Then $v_1/v_2$ is a nonzero section of the calibration
  bundle (over a point), which is naturally isomorphic to $\Theta_1 \otimes
  \cdots \otimes \Theta_{k+\ell}$, where $\Theta_i = T_{E_i,y_i} \otimes
  T_{C_g,y_i}$. By the proof of Lemma~\ref{lem:auto-eta},
  $\Gamma$ acts trivially on $\Theta_{k+1} ,\ldots,
  \Theta_{k+\ell}$ and acts by the same weight on $\Theta_{1} ,\ldots,
  \Theta_{k}$. Hence the subgroup of $\Gamma$ acting trivially on
  $\Theta_1\otimes \cdots \otimes \Theta_{k+\ell}$ is finite and the proof is
  complete.
\end{proof}
\fi
\begin{proposition}
  \label{prop:representability}
  $MQ^{\epsilon_0}_{g,n}(X,\beta)$ is a Deligne--Mumford stack of finite type
  over $\mathbb C$ with finite-type separated diagonal.
\end{proposition}
\ifdefined\SHOWPROOFS
\begin{proof}
  By Lemma~\ref{lem:artin-stack}, $M \mathfrak {Qmap}^{\sim}_{g,n}(X,\beta)$ is
  an Artin stack of finite type with finite-type separated diagonal. By
  Lemma~\ref{lem:open-condition},  $MQ^{\epsilon_0}_{g,n}(X,\beta)$ is an open
  substack of $M \mathfrak {Qmap}^{\sim}_{g,n}(X,\beta)$. By
  Lemma~\ref{lem:finite-auto} and \cite[Remark~8.3.4]{olsson2016algebraic},
  it is of Deligne--Mumford type.
\end{proof}
\fi

\subsection{The virtual fundamental class }
The virtual fundamental class of $MQ^{\epsilon_0}_{g,n}(X,\beta)$ is defined
in the same way as that for $Q^{\epsilon}_{g,n}(X,\beta)$ (c.f.\
\cite[\textsection2.4.5]{cheong2015orbifold}).
Let $\pi:\mathcal C \to MQ^{\epsilon_0}_{g,n}(X,\beta)$ be the universal curve and
$u: \mathcal C \to [W/G]$ be the universal quasimap.
Precisely as in the case of $Q^{\epsilon}_{g,n}(X,\beta)$, the forgetful morphism
\[
  MQ^{\epsilon_0}_{g,n}(X,\beta) \longrightarrow  M\tilde {\mathfrak M}_{g,n,d}
\]
admits a relative perfect obstruction theory
\[
  (R\pi_*(u^*{\mathbb T_{[W/G]}}))^\vee
  \longrightarrow
 \mathbb L_{MQ^{\epsilon_0}_{g,n}(X,\beta)/ M\tilde {\mathfrak M}_{g,n,d}}    .
\]
This defines the virtual fundamental class
\[
  [MQ^{\epsilon_0}_{g,n}(X,\beta)]^{\mathrm{vir}} \in
  A_*(MQ^{\epsilon_0}_{g,n}(X,\beta)).
\]
\section{The properness of the master space}
\label{sec:properness}
The goal of this section is to prove the properness of the master space, via the
valuative criterion.
The proof will be divided into
several cases. In each case the existence part and the uniqueness part of the
criterion will be verified together.
\begin{proposition}
  \label{prop:properness}
  $MQ^{\epsilon_0}_{g,n}(X,\beta)$ is proper over $\mathbb C$.
\end{proposition}

\ifdefined\SHOWPROOFS
\subsection{Warm-up}
Let $(R,\mathfrak m)$ be a complete discrete valuation $\mathbb C$-algebra with
fraction field $K$ and residue field $\mathbb C$ as before. Given any
\[
  \xi^* = (\pi^*:\mathcal C^*\to \operatorname{Spec}K, \mathbf x^*,
  e^*,u^*,N^*,v_1^*,v_2^*)\in MQ^{\epsilon_0}_{g,n}(X,\beta)(K),
\]
we will prove that up to finite base change there is a unique
extension of $\xi^*$ to
\[
  \xi = (\pi:\mathcal C\to \operatorname{Spec}K, \mathbf x, e,u,N,v_1,v_2)\in
  MQ^{\epsilon_0}_{g,n}(X,\beta)(R).
\]
If $v_1^* = 0$ (resp. $v_2^* = 0$), $\xi^*$ lies in the substack
$\tilde{Q}{}^{\epsilon_+}_{g,n}(X,\beta)$ (resp.
$Q^{\epsilon_-}_{g,n}(X,\beta)$), which is proper
\cite[Theorem~2.7]{cheong2015orbifold}. Hence the extension exists uniquely.
Hence, without loss of generality, we assume that $v_1^*$ and $v_2^*$ are both nonzero.

\begin{lemma}
  \label{lem:entanglement-and-calibration}
  Given any extension
  \[
    (\pi:\mathcal C\to \operatorname{Spec}R,\mathbf x) \in
    {\mathfrak M}_{g,n,d}^{\mathrm{wt,ss}}(R)
  \]
  of the underlying weighted curves $(\pi^*:\mathcal C^*\to
  \operatorname{Spec}K,\mathbf x^*)$ of $\xi^*$,
  \begin{enumerate}
  \item
    there is a unique extension $(e,N,v_1,v_2)$ of $(e^*,N^*,v_1^*,v_2^*)$, such
    that
    \[
      \xi := (\pi:\mathcal C\to \operatorname{Spec}R,\mathbf x, e, N,v_1,v_2)\in
    M\tilde{\mathfrak {M}}_{g,n,d}(R);
    \]
  \item
    there is at most one extension $u:\mathcal C \to [W/G]$ of $u^*$ as a
    quasimap.
  \end{enumerate}
\end{lemma}
\ifdefined\SHOWPROOFS
\begin{proof}
  Statement (1) follows from the properness of the representable morphisms
  \[
    M\tilde{\mathfrak {M}}_{g,n,d} \longrightarrow   \tilde{\mathfrak
      {M}}_{g,n,d} \longrightarrow  {\mathfrak M}_{g,n,d}^{\mathrm{wt,ss}}.
  \]
  Statement (2) follows from the separatedness of the moduli of quasimaps on a fixed family of nodal
  curves. See the first paragraph of the proof of
  \cite[Proposition 4.3.1]{ciocan2014stable}. The argument also works in the orbifold case (c.f.
 \cite[\textsection2.4.4]{cheong2015orbifold}).
\end{proof}
\fi
The following lemma will be useful when we contract rational tails in the special fiber.
\begin{lemma}
  \label{lem:extension-of-map}
  Let $\mathcal C$ be a surface and $p_1 ,\ldots, p_n$ be closed
  points of the regular locus of $\mathcal C$.
  Any morphism $u^0: \mathcal C^0= \mathcal C \setminus\{p_1
  ,\ldots, p_n\} \to [W/G]$ uniquely extends to a $u:\mathcal C \to [W/G]$, up
  to unique isomorphisms.
\end{lemma}
\ifdefined\SHOWPROOFS
\begin{proof}
  We first assume that $\mathcal C$ is a scheme near the $p_i$'s. We follow the
  argument in \cite[\textsection 4.3]{ciocan2014stable}.
  A map from a scheme $S$ to $[W/G]$ is a principal $G$-bundle $P\to S$ with a
  section of $P\times_G W \to S$.
  By \cite[Lemma~4.3.2]{ciocan2014stable}, the principal bundle extends. The
  extension is unique up to unique isomorphisms by Hartogs' theorem and the affineness of $G$.
  Since $W$ is affine, the section extends by Hartogs' theorem. The extension is again
  unique, by the separatedness of the morphism $P\times_{G} W\to S$.

  In general, we take an \'etale cover $U\to \mathcal C$. The pullback of
  $u^0$ to $U\times_{\mathcal C} \mathcal C^0$ extends to $\tilde u: U\to [W/G]$
  by the previous argument. Since the
  extension is unique up to unique isomorphisms, it descends to $\mathcal C$.
\end{proof}
\fi

We now consider the situation of contracting rational bridges.
\begin{lemma}
  \label{lem:contraction-bridge}
  Let $\mathcal C \to \operatorname{Spec}  R$ be a family of twisted curves and
  $u: \mathcal C \to [W/G]$ be a family of quasimaps. Let $E\subset \mathcal
  C$ be a chain of rational bridges in the special fiber. Suppose
  that the degree of $u$ on any irreducible component of $E$ is zero.
  Then there
  exists a unique ($2$-)commutative diagram
  \[
    \begin{tikzcd}
      \mathcal C \arrow[d,"{\rho}"']\arrow[rd,"u"] & \\
      \mathcal C^\prime \arrow[r,"{u^\prime}"]& {[W/G]}
    \end{tikzcd}
    ,
  \]
  where $\rho$ is the contraction of $E$ to a (possibly orbifold) point and
  $u^\prime$ is a family of quasimaps. In particular, $\rho$ is also representable.
\end{lemma}
\ifdefined\SHOWPROOFS
\begin{proof}
 By \cite[Lemma~2.3]{cheong2015orbifold}, 
  $E$ is mapped to a single point $w$ of the coarse moduli $\underline X$.
  We work \'etale locally on $X$. Without loss of generality, we assume that
  $\mathcal C$ is mapped into $X$.
  For any \'etale chart $U \to X$ around $w$, form the
  fibered diagram
  \[
    \begin{tikzcd}
      \tilde E \arrow[r] \arrow[d]&  V \arrow[r]\arrow[d] & U \arrow[d]\\
      E\arrow[r] & \mathcal C \arrow[r,"u"] & {X}
    \end{tikzcd}.
  \]
  Since $u$ is representable, $V$ is a scheme. Since the image of
  $V$ contains $E$ and the dual graph of $E$ has no loop, each connected
  component of $\tilde E$ is a disjoint union of rational bridges.
  Let $V \to V^\prime$ be the contraction of each connected component of $\tilde
  E$ to single points. Since each connect component of $\tilde E$ is sent to a
  single point of $U$, the map $V\to U$ factorizes through $V^\prime$.   The
  contraction is unique. Hence varying $U$, we can glue the schemes $V^\prime
  \to U$ to form an algebraic stack $\mathcal C^\prime$ with a representable
  morphism $\mathcal C^\prime \to X$.
  This gives the desired $\mathcal C \to \mathcal C^{\prime}\to X$.
  Any $\mathcal C^{\prime}$ must be locally in the \'etale site of
  $X$ given by this construction, and the gluing is unique.
  This completes the proof.
\end{proof}
\fi

\subsection{Case 1: Creating degree-$d_0$ rational tails}
Let us assume that $(g,n,d)\neq (0,1,d_0)$ and $\xi^*$ does not have
degree-$d_0$ rational tails. Then
\[
  \eta^* = (\pi^*:\mathcal C^*\to \operatorname{Spec}K, \mathbf x^*,u^*)
\]
is a family of $\epsilon_-$-stable quasimaps. By the properness of
$Q^{\epsilon_-}_{g,n}(X,\beta)$, up to finite base change it uniquely extends
to a family of $\epsilon_-$-stable quasimaps
\begin{equation}
  \label{eqn:eta-minus}
  \eta_- = (\pi_{-}:\mathcal C_-\to \operatorname{Spec}R, \mathbf x_-,u_-).
\end{equation}
By Lemma~\ref{lem:entanglement-and-calibration}, we have an
$\epsilon_0$-semistable extension of $\xi^*$
\[
  \xi_- = (\pi_{-}:\mathcal C_-\to \operatorname{Spec}R, \mathbf
  x_-,e_-,u_-,N_-,v_{1-},v_{2-}).
\]
Set
\[
  \delta = \operatorname{ord}(v_{1-})- \operatorname{ord}(v_{2-}),
\]
where ``$\operatorname{ord}$'' means the vanishing order at the unique closed point
of $\operatorname{Spec}  R$.

If $\delta \leq 0$ or there are no length-$d_0$ base points
in the special fiber, $\xi_-$ is $\epsilon_0$-stable.
Otherwise, let
\[
  p_1,\ldots,p_k \in \mathcal C_{-}|_{\operatorname{Spec}  \mathbb C}
\]
be the length-$d_0$ base points in the special fiber. In order to obtain a
stable extension, we need to modify $\mathcal C_-$ at those
points, possibly after finite base change.
Consider a/the totally ramified finite base change $\operatorname{Spec}
R^\prime \to  \operatorname{Spec}  R$ of degree $r$.
Let $K^\prime$ be the fraction field of $R^\prime$.
\begin{definition}
  \label{def:modification}
  A \textit{modification} of $\eta_-$ of degree $r$ is a family of
  $\epsilon_0$-semistable quasimaps
  \[
    \tilde \eta = (\tilde \pi: \tilde{\mathcal C} \to  \operatorname{Spec}
    R^\prime,\tilde {\mathbf x} , \tilde u),
  \]
  together with an isomorphism
  \[
    \tilde\eta|_{\operatorname{Spec}  K^\prime} \cong
    \eta_{-}|_{\operatorname{Spec}  K^\prime}.
  \]
\end{definition}
Up to finite base change, we will classify all those modifications, and we will
show that for a unique modification, the corresponding $\xi$ from
Lemma~\ref{lem:entanglement-and-calibration} is $\epsilon_0$-stable.

We first show that the modification is local at those base points $p_i$ in the
following sense.
Consider any modification $\tilde \eta$ of $\eta_-$ of degree $r$.
We denote the pullback of $\eta_-$ to $\operatorname{Spec} R^\prime$ by
\[
  \eta^\prime=(\pi^\prime:\mathcal C^
  \prime \to  \operatorname{Spec}  R^\prime, \mathbf x^\prime, u^\prime).
\]
\begin{lemma}
  \label{lem:morphism-btw-resolutions}
  The family of curves $\mathcal C^\prime$ is obtained from $\tilde {\mathcal
    C}$ by contracting the
  rational tails of degree $d_0$ in the special fiber.
\end{lemma}
\ifdefined\SHOWPROOFS
\begin{proof}
  Let $\tau:\tilde {\mathcal C} \to \mathcal C^{\prime\prime}$ be the contraction of
  the degree-$d_0$ rational tails in the special fiber. It is an isomorphism
  over an open substack $\mathcal C^0 \subset \mathcal C^{\prime\prime}$, which
  is the complement of finitely many regular points. By
  Lemma~\ref{lem:extension-of-map}, the restriction of $\tilde u$ to
  $\tau^{-1}(\mathcal C^0) \cong \mathcal C^0$ uniquely extends to
  $u^{\prime\prime} : \mathcal C^{\prime\prime} \to [W/G]$. Set $\mathbf
  x^{\prime\prime} = \tau(\tilde{\mathbf x})$. Then $\eta^{\prime\prime}:=
  (\mathcal C^{\prime\prime} \to \operatorname{Spec} R^\prime, \mathbf
  x^{\prime\prime}, u^{\prime\prime})$ is an $R^\prime$-point of 
  $Q^{\epsilon_-}_{g,n}(X,\beta)$, which must be isomorphic to $\eta^\prime$ by
  the separatedness $Q^{\epsilon_-}_{g,n}(X,\beta)$. In particular the
  underlying family of curves are isomorphic.
\end{proof}
\fi
Since the extension $R\subset R^\prime$ is totally ramified, the special fiber
of $\mathcal C^\prime$ is canonically isomorphic to
the special fiber of $\mathcal C_{-}$. 
Each rational tail in  $\tilde{\mathcal C}$ that is contacted must be contracted
to one of the $p_i$'s.
By (2) of the master-space stability condition in
Definition~\ref{def:mspace-stability}, over each $p_i$ there lies a contracted
tail $E_i$. Each $E_i$ intersects the other components of the special fiber at
some an $A_{a_i-1}$-singularity of $\tilde{\mathcal C}$. In particular, $a_i =
1$ if and only if $\tilde{\mathcal C}$ is smooth near $E_i$.

\begin{definition}
  \label{def:singularity-type}
  We say that modification $\tilde \eta$ has singularity type $(a_1/r,\ldots,a_k/r)$.
\end{definition}
It is clear that a further base change does not change the singularity type.
We will only consider sufficiently divisible $r$.
It turns out that the possible singularity types of modifications of a given
$\eta_-$ are bounded from above.
We will define
\[
  b_i\in \mathbb Q_{>0}\cup \{+\infty\},  \quad i = 1 ,\ldots, k,
\]
which will be the maximum of possible singularity types.
First consider the special case when there are no length-$d_0$ base points in
the generic fiber.
In this case, $\eta^*$ is also $\epsilon_+$-stable.
We define $(b_1 ,\ldots, b_k)$ to be the singularity type of the
$\epsilon_+$-stable extension of the quasimaps, which is unique up to further
base change.

In general, let $B^*\subset \mathcal
C_-|_{\operatorname{Spec} K}$ be the length-$d_0$ base locus and $B\subset
\mathcal C_-$ be the closure of $B^*$.  Say $p_{\ell+1},\ldots,p_k$ are the
length-$d_0$ base points in the special fiber that are
contained in $B$. Note that since the special fiber does not have base points of
length strictly greater than $d_0$, $B$ is transversal to the special fiber.
Also, $B$ is disjoint from nodes and markings. Since $R$ is complete we can view
$B$ as a disjoint union of sections of $\pi_-$.

The idea is to first define $b_i = \infty$ for $i=\ell+1
  ,\ldots, k$ and we replace $B$ by additional (possibly orbifold) markings. Then
the generic fiber do not have length-$d_0$ base points any more and we define
$b_1 ,\ldots, b_\ell$ as in the previous case. The argument is standard except for the
following lemma, which enables us to replace base points by markings.

\begin{lemma}
  \label{lem:replace-base-points}
  There is a unique way to replace $B$ by orbifold markings $\mathbf
  y_{\mathrm{reg}-}$, obtaining the family of pointed curves
  \[
    (\mathcal C_{\mathrm{reg}-},(\mathbf x_{\mathrm{reg}-}, \mathbf y_{\mathrm{reg}-})),
  \]
  where $\mathbf x_{\mathrm{reg}-}$ comes from $\mathbf x_{-}$, and $\mathbf
  y_{\mathrm{reg}-}$ comes from $B$, such that
  the restriction of $u_-$ to $\mathcal C_-\backslash B$ uniquely extends to a
  quasimap $u_{\mathrm{reg}-}: \mathcal C_{\mathrm{reg}-} \to [W/G]$ such that
  $\mathbf y_{\mathrm{reg}-}$ is mapped into $X$.
\end{lemma}
\ifdefined\SHOWPROOFS
\begin{proof}
  We write $B = \coprod_{i=1}^{s}B_i$, where each $B_i$ is the image of a
  section of $\mathcal C_{-} \to \operatorname{Spec} R$.
  For a tuple of positive integers $\vec r = (r_1 ,\ldots, r_s)$, let $\mathcal
  C_{\vec r}$ be the stack obtained from $\mathcal C_{-}$ via the $r_i$-th root
  construction along the divisor $B_i$ for all $i=1 ,\ldots, s$. 
  We view ${\mathcal C_{-}\setminus B}$ as an open substack of $\mathcal
  C_{\vec r}$.
  By \cite[Lemma~2.5]{cheong2015orbifold}, there exists a unique $\mathcal
  C_{\vec r}$ such that ${u_{-}}|_{\mathcal C_{-}\setminus B}$ extends to 
  a morphism $U\to [W/G]$, where $U\subset \mathcal C_{\vec r}$ is some open
  substack of $\mathcal C_{\vec r}$ 
  containing both $\mathcal C_{-}\setminus B$ and the generic fiber of $\mathcal
  C_{\vec r}$. Set $\mathcal C_{\mathrm{reg}-}$ to be that $\mathcal C_{\vec
    r}$, $\mathbf x_{\mathrm{reg}-}$ to be $\mathbf x_{-}$ and $\mathbf
  y_{\mathrm{reg}-}$ to be the universal roots of the $B_i$'s. By
  Lemma~\ref{lem:extension-of-map}, the map further extends to $u_{\mathrm{reg-}}:\mathcal
  C_{\mathrm{reg}-}\to [W/G]$.
  
  It remains to show that $u_{\mathrm{reg}-}$ has no base point at
  $\mathbf y_{\mathrm{reg}-}$ in the special fiber. To see this,
  we first take the special fiber of $(\mathcal C_{-}, \mathbf x_{-},u_{-})$ and 
  apply the same construction to replace its base points by markings.
  Thus we obtain a quasimap
  \[
    u_0:(C^\prime, \mathbf x^\prime, \mathbf y^\prime) \longrightarrow  [W/G]
  \]
  that is regular along $\mathbf y^\prime$.
  Note that $(C^\prime
  \setminus {\mathbf y^\prime},u_0|_{C^\prime \setminus {\mathbf y^\prime}})$ is isomorphic to
  the special fiber of
  \begin{equation*}
    (\mathcal C_{\mathrm{reg}-}\setminus
    \mathbf y_{\mathrm{reg}-}, u_{\mathrm{reg}-}|_{(\mathcal
      C_{\mathrm{reg}-}\setminus \mathbf y_{\mathrm{reg}-})}).
  \end{equation*}
  We compare the degrees (c.f.\ page 784 of \cite{cheong2015orbifold})
  \[
    \deg(u_0) = \deg(u_-) - (k-\ell) d_0 = \deg(u_{\mathrm{reg}-}).
  \]
  Since $u_0$ has no base point at $\mathbf
  y^\prime$, $u_{\mathrm{reg}-}$ has no base points at the special fiber of $\mathbf
  y_{\mathrm{reg}-}$. Hence $u_{\mathrm{reg}-}$ has no base point along $\mathbf
  y_{\mathrm{reg}-}$ and the proof is complete.
\end{proof}
\fi

Thus we have finished the definition of $(b_1 ,\ldots, b_k)$.
\begin{lemma}
  \label{lem:modification}
  For $\vec a = (a_1,\ldots,a_k)\in \mathbb Q_{>0}^{k}$ and sufficiently
  divisible $r$, the following two statements are equivalent:
  \begin{enumerate}[(1)]
  \item
    $\eta_-$ has a modification $\tilde{\eta}$ of degree $r$ and singularity
    type $\vec a$;
  \item
    $a_i\leq b_i$ for each $i$.
  \end{enumerate}
  Moreover, assuming that $\tilde\eta$ exists,
  \begin{itemize}
  \item
    $\tilde\eta $ is uniquely determined by $\vec a$ and $r$;
  \item
    for each $i$, $E_i$  contains no length-$d_0$ base point if and only
    if $a_i = b_i$, where
    $E_i$ is the rational tail of $\tilde\eta$  lying over $p_i$.
  \end{itemize}
\end{lemma}
\ifdefined\SHOWPROOFS
\begin{proof}
  We focus on the case when there are no base points of length $d_0$ in the
  generic fiber and briefly describe the modification needed in the
  general case. We will first assume (2) and construct
  $\tilde\eta$. Then we will prove the uniqueness by showing that any
  modification must come from that construction. The remaining assertions
  including the uniqueness then follow from the explicit construction.

  Now we assume that there is no base point of length $d_0$ in the generic
  fiber. Given sufficiently divisible $r$ and $\vec a\in \mathbb Q_{>0}^k$
  satisfying (2), first form the degree-$r$ $\epsilon_+$-stable modification
  \[
    {\eta}_+  = (\pi_+:{\mathcal C}_+ \to
    \operatorname{Spec}  R^\prime, {\mathbf x}_+,u_+).
  \]
  Let $E_1 ,\ldots, E_k$ be the degree-$d_0$ rational tails in the special fiber
  of $\mathcal C_+$. Suppose that $E_i$ intersects the other
  components of the special fiber of $\mathcal C_+$ at the point
  $q_i$, which is an $A_{rb_i-1}$-singularity of $\mathcal C_+$, by the
  definition of $b_i$.
  Let
  \[
   \tilde{\mathcal C}_+ \longrightarrow  \mathcal C_+
  \]
  be the minimal resolution of the singularities at $q_1 ,\ldots, q_k$.
  The fiber over each $q_i$ is a chain of rational curves as is shown the figure
  
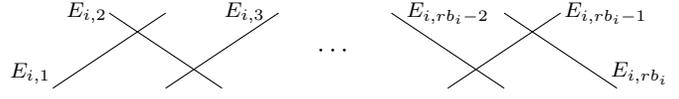
\begin{center}
    \begin{tikzpicture}[xscale = 1.5]
      \draw (-10+.5,-.5) -- (-9+.5,.5); \node at (-10.2+.5,-.3) {\footnotesize
        $E_{i,1}$}; \draw (-9.5+.5,.5) -- (-8.5+.5,-.5); \node at (-9.7+.5,.5)
      {\footnotesize $E_{i,2}$}; \draw (-9+.5,-.5) -- (-8+.5,.5); \node at
      (-8.3+.5,.5) {\footnotesize $E_{i,3}$}; \node at (-7,0) {$\cdots$}; \node
      at (-6.5+.5,.5) {\footnotesize $E_{i,rb_i -2}$}; \node at (-5.1+.5,.5)
      {\footnotesize $E_{i,rb_i-1}$}; \node at (-4.8+.5,-.3) {\footnotesize
        $E_{i,rb_i}$}; \draw (-5-.5,-.5) -- (-6-.5,.5); \draw (-5,.5) --
      (-6,-.5); \draw (-5+1-.5,-.5) -- (-6+1-.5,.5);
    \end{tikzpicture}.
  \end{center}
  Assume that
  $E_{i,rb_i}$ intersects the proper transform of $E_i$.
  By composition we have the quasimap $\tilde{\mathcal C}_+ \to [W/G]$.
  Now apply Lemma~\ref{lem:extension-of-map} and Lemma~\ref{lem:contraction-bridge}.
  Let
  \[
    \tilde{\mathcal C}_+ \longrightarrow \tilde {\mathcal C}
  \]
  be the blowing down of
  \[
    E_{i,1} ,\ldots, E_{i,a_i-1} \quad \text{and} \quad E_{i,a_i+1} ,\ldots, E_{i,rb_i}.
  \]
  The quasimap descends as $\tilde u: \tilde {\mathcal C}\to [W/G]$. Thus
  we obtain the desired $\tilde\eta$.

  Now without assuming (2), we prove that up to further finite base change any
  modification $\tilde\eta$ comes from such a
  construction. Given $\tilde \eta$, we first construct up to finite base change
  a family of $\epsilon_0$-semistable quasimaps
  \[
    \tilde{\eta}_+  = (\tilde \pi_+: \tilde{\mathcal C}_+ \to  \operatorname{Spec}
    R^\prime,\tilde {\mathbf x}_+ , \tilde u_+)
  \]
  with a $\mathcal C_-$-morphisms
  \[
    f:\tilde{\mathcal C}_+ \longrightarrow   \tilde{\mathcal C}.
  \]
  To obtain $\tilde{\eta}_+$, we first resolve the degree-$d_0$ base points on
  each $E_i$ as follows. We add some nonorbifold markings to $\tilde \eta$ so that
  it becomes $\epsilon_-$-stable. Thus $E_i$ is no longer a rational tail. The
  generic fiber is still $\epsilon_+$-stable. We take the $\epsilon_+$-stable
  modification of this family and then forget the additional markings. This
  replaces the length-$d_0$ base points by a degree-$d_0$ rational tails. The
  proper transform of $E_i$ contains two nodes.
  We then take the minimal resolution of the singularity at those two nodes.
  This gives us the desired family $\tilde{\eta}_+$ together with the map
  $\tilde{\mathcal C}_+ \to \tilde{\mathcal C}$.

  Having obtained $\tilde{\eta}_+$, we see that it is indeed the
  $\tilde{\eta}_+$ constructed in the first part of the proof. The argument
  is similar to the one for Lemma~\ref{lem:morphism-btw-resolutions}.
  We contract the chain of unstable rational bridges in the special fiber of
  $\tilde{\mathcal C}_+$
  using Lemma~\ref{lem:contraction-bridge}, and
  obtain an $\epsilon_+$-stable modification of $\eta$. By
  the uniqueness of $\epsilon_+$-stable modifications, it must be equal to the
  $\mathcal C_+$ constructed in the first part of the proof up to base change.
  Thus we have a (representable) morphism
  \[
     \tilde{\mathcal C}_+ \longrightarrow  \mathcal C_+.
  \]
  This morphism contracts a chain of smooth rational curves to each node $q_i$.
  Since the special fiber of $\tilde{\mathcal C}_+$ is reduced, it must be the
  minimal resolution of the $q_i$'s as in the first part of this proof. This
  proves that up to further base change, every modifications is the unique one
  coming from the construction in the first part of this proof. Since up to base
  change the construction
  only depends on $\vec a$, this proves uniqueness. The ``moreover'' part of the Lemma
  follows from the explicit construction and the uniqueness.

  In the general case, we first replace the length-$d_0$ base points of
  $\eta_-$ by additional markings $\mathbf y_{\mathrm{reg}-}$ using
  Lemma~\ref{lem:replace-base-points}.
  The construction above gives us rational tails  with
  desired singularities except at those additional markings. Then we
  modify the surface at the
  additional markings $\mathbf y_{\mathrm{reg}-} = (y_{\ell+1} ,\ldots,
  y_k)$
  in the special fiber to obtain the desired singularity. We modify it at $y_i$
  via a sequence of $a_i$ blowups followed by one blowdown contracting all the
  exceptional divisors except for the last one.
  This gives us the desired rational tail with prescribed singularity. Finally
  we change the additional markings $\mathbf y_{\mathrm{reg}-}$ back into base
  points: in the generic fiber, use the original quasimaps; near $\mathbf
  y_{\mathrm{reg}-}$ in the special fiber, use Lemma~\ref{lem:extension-of-map}.
  It is easy to see that any modification must come from such a construction.
  For each blowup in the construction, the blowup center must be the
  intersection of the special fiber and the proper transform of the $y_i$'s.
  Hence this construction is unique.
\end{proof}
\fi
\ifdefined\SHOWPROOFS
\begin{proof}[Verifying the valuative criterion in Case 1]
  We now prove that up to finite base change $\xi^*$ has a unique stable
  extension $\xi$ in Case 1.

  Recall that up to finite base change we have a unique $\epsilon_-$-stable
  extension $\eta_-$ of $\eta^*$, where $\eta^*$ is the underlying family of
  quasimaps of $\xi^*$. For any stable extension $\tilde \xi$ of $\xi^*$ possibly
  after some degree-$r$ finite base change $\operatorname{Spec} R^\prime \to
  \operatorname{Spec} R$, the underlying family of quasimaps $\tilde \eta$ of
  $\tilde \xi$ is a modification of $\eta_-$, in the sense of
  Definition~\ref{def:modification}. By
  Lemma~\ref{lem:entanglement-and-calibration}, $\tilde\eta$ uniquely determines
  $\tilde\xi$. By Lemma~\ref{lem:modification}, $\tilde\eta$ is uniquely
  determined by $r$ and its singularity type $\vec a =(a_1 ,\ldots, a_k)$. We make
  the convention that $a_i = 0$ if $\tilde{\mathcal C}$ is isomorphic to $\mathcal
  C_-$ near $p_i$, up to base change. Note that by (2) of the stability condition
  in Definition~\ref{def:mspace-stability}, either $a_i>0$ for all $i$ or $a_i=0$
  for all $i$.

  By Lemma~\ref{lem:limit-of-entanglement-case1}, the entangled tails are those
  corresponding to the maximal $a_i$. By Lemma~\ref{lem:limit-calib-1},
  \[
    \operatorname{ord}(\tilde v_1)- \operatorname{ord} (\tilde v_2) =
    r(\operatorname{ord}(v_{1-})- \operatorname{ord} ( v_{2-}))
     - r\sum_{i=1}^k a_i.
  \]
  First suppose that
  $\operatorname{ord}(v_{1-}) > \operatorname{ord} ( v_{2-})$, i.e.\  $v_{1-} =
  0$ at the closed point.
  Then $\xi_-$ is not stable since it has length-$d_0$ base points. Thus a
  nontrivial modification is needed.
  We write $\delta = \operatorname{ord}(v_{1-})- \operatorname{ord} ( v_{2-})$ and
  $|\vec a| = \sum_{j=1}^k a_j$.
  By Lemma~\ref{lem:modification}, we
  see that given $\vec a$ and a sufficiently divisible $r$, there is a stable
  extension $\xi$ whose underlying family of quasimaps is a degree-$r$
  modification of $\eta_-$ of singularity type $\vec a$ if and only if
  \begin{equation}
    \begin{cases}
      |\vec a|\leq \delta;\\
      0 < a_i \leq  b_i, \text{ for all }i = 1 ,\ldots, k;\\
      \text{if }|\vec a| < \delta
       \text{ then }a_i =  b_i \text{ for all } i = 1 ,\ldots, k;\\
      \text{if }a_i < b_i\text{ then }a_i\text{ is maximal among
      }a_1,\ldots,a_k, \text{ for all } i = 1 ,\ldots, k.
    \end{cases}
  \end{equation}
  It is easy to see that the system has a unique solution for $(a_1 ,\ldots,
  a_k)$. This proves the existence and uniqueness of the stable extension $\xi$,
  up to finite base change.

  Now suppose
  $\operatorname{ord}(v_{1-}) \leq \operatorname{ord} ( v_{2-})$, i.e.\
  $v_{1-}\neq 0$ at the close point. Then $\xi :=\xi_-$
  is already stable. We claim that this is the unique stable extension. Indeed,
  if $a_i=0$ for all $i$, then the uniqueness follows from the separatedness of
  $Q^{\epsilon_-}_{g,n}(X,\beta)$. If $a_i>0$ for some $i$, it means there are
  degree-$d_0$ rational tails in the special fiber. By
  Lemma~\ref{lem:limit-calib-1},
  $\operatorname{ord}(\tilde v_1) < \operatorname{ord}(\tilde v_2)$, which
  violates the stability condition.
\end{proof}
\fi

\subsection{Case 2: A single rational tail} Now we study the case
$(g,n,d)=(0,1,d_0)$. In this case the curve will never break into irreducible
components because rational tails of degree $<d_0$ are not allowed.
Hence there is no ``entanglement'', i.e.\ $\tilde{\mathfrak M}_{0,1,d} =
{\mathfrak M}_{0,1,d}^{\mathrm{wt,ss}}$.
Recall that  the calibration bundle is the relative cotangent bundle along the unique
marking (Definition~\ref{def:calibration-bundle}). We denote the calibration
bundle of a family of curves $\mathcal C$ by $\mathbb M_{\mathcal C}$.

As in the previous case, we will start with some extension of $\eta^*$ and then
modify it. If there is no length-$d_0$ base point in the generic fiber, up to
finite base change let $\eta_+ = (\pi_+: \mathcal C_+\to \operatorname{Spec}  R,
\mathbf x_+, u_+)$ be the $\epsilon_+$-stable extension of $\eta^*$. If there
is a length-$d_0$ base point in the generic fiber, take any semistable extension
of $\eta^*$. We still call it $\eta_+$, to unify the notation. In both
cases, fix a nonvanishing section $s_+$ of $\mathbb M_{\mathcal C_+}$.

Recall the definition of a modification in Definition~\ref{def:modification}.
Given a modification $\tilde\eta$ over the degree $r$ base change
$\operatorname{Spec} R^\prime \to \operatorname{Spec} R$, the pullback of
$s_+|_{\operatorname{Spec} K}$ extends to a rational section $\tilde s$ of
$\mathbb M_{\tilde {\mathcal C}}$.
\begin{definition}
  \label{def:order-and-degree}
  We define the degree of the modification $\tilde\eta$ to be $r$, and define
  the order of $\tilde\eta$ to be $\operatorname{ord}(\tilde s)/r$, where
  $\mathrm{ord}(\tilde s)$ is the vanishing order\footnote{The vanishing order
    is negative is $\tilde s$ has a pole.} of $\tilde s$ at the closed
  point of $\operatorname{Spec}  R^\prime$.
\end{definition}
Clearly, a further base change does not affect the order.  When there is no
length-$d_0$ base point in the generic fiber, set $b=0$. Otherwise set
$b=-\infty$.
\begin{lemma}
  \label{lem:modification2}
  Let $r\in \mathbb Z_{>0}$ and $a\in \mathbb Q$. Suppose that $r$ is
  sufficiently divisible. Then there is a modification $\tilde{\eta}$ of degree
  $r$ and order $a$ if and only if $a\geq b$.

  Moreover, assuming that $\tilde\eta$ exists, then
  \begin{itemize}
  \item
    $\tilde\eta $ is uniquely determined by $a$ and $r$;
  \item
    $\tilde\eta$ is $\epsilon_+$-stable if and only if $a = b$.
  \end{itemize}
\end{lemma}
\ifdefined\SHOWPROOFS
\begin{proof}
    The proof is similar to that of Lemma~\ref{lem:modification}.
    Hence we only sketch the proof here, and we only consider the case when there
    is no length-$d_0$ base point in the generic fiber.  The other easier case
    is left to the reader.

    Given $a\geq b = 0$, we prove the existence. First take the degree-$r$ base
    change and we still use $\eta_+$ to denote the family by abuse of notation. If
    $a=0$, there is nothing to do. Otherwise blow up $\mathcal C_+$ at the unique
    marking in the special fiber. Then blow up the new surface at the proper
    transform of the unique marking in the special fiber. Repeating this process, we
    perform $ra$ blowups and obtain the surface $\tilde{\mathcal C}_+$. Finally blow
    down all the irreducible components in the special fiber of $\tilde{\mathcal
      C}_+$ except for the exceptional divisor of the last blowup. Thus we obtain
    $\tilde{\pi}: \tilde {\mathcal C} \to \operatorname{Spec} R$. By
    Lemma~\ref{lem:extension-of-map}, the quasimaps uniquely extends to
    $\tilde{\mathcal C}$. Thus we obtain $\tilde{\eta} = (\tilde{\pi}: \tilde
    {\mathcal C} \to \operatorname{Spec} R, \tilde{\mathbf x},\tilde u)$. A local
    computation near the unique marking shows that $\tilde{\eta}$ is indeed of order
    $a$.

    For the uniqueness, we show that any $\tilde{\eta}$ is obtained this way.
    Given $\tilde\eta$, if it is $\epsilon_+$-stable, then it must be isomorphic
    to $\eta_+$, up to finite base change. In particular it has order $0$.
    From now on, suppose that it is not $\epsilon_+$-stable, i.e.\ the special
    fiber has a base point of length $d_0$. We will recover the $\tilde{\mathcal
      C}_+$ constructed above.
    We add an additional (non-orbifold) marking to
    $\tilde{\eta}$ such that it becomes $\epsilon_-$-stable.
    Then possibly after finite base change take the $\epsilon_+$-stable
    modification of this $\epsilon_-$-stable family. Thus the special fiber has
    two irreducible components $E_1,E_2$, where $E_1$ has degree $d_0$ and $E_2$ has two
    markings. Take the minimal resolution of singularity at $E_1\cap E_2$.
    Thus we obtain some $\tilde{\mathcal C}_+$. It remains to show that
    $\tilde{\mathcal C}_+$ is the one constructed above and that $\tilde {\mathcal C}$ is
    obtained from blowing down $\tilde{\mathcal C}_+$. It follows from the
    properness of $Q^{\epsilon_+}_{0,1}(X,\beta)$ and
    $Q^{\epsilon_-}_{0,2}(X,\beta)$, respectively. The argument is similar to that in
    Lemma~\ref{lem:morphism-btw-resolutions} and we omit the details.

    Thus we have shown that any $\tilde\eta$ comes from the construction above. Note
    that if it is not $\epsilon_+$-stable then we must have $a>0$. Also note
    that  the
    blow up centers of the sequence of blowups $\tilde {\mathcal C}_+ \to \cdots
    \to \mathcal C_+$ must be at the proper transform of the unique marking in
    the special fiber. Hence $\tilde\eta$ is uniquely determined by $a$ and $r$.
    This proves the uniqueness.
    The other statements of the Lemma follow from the explicit construction and
    the uniqueness.
\end{proof}
\fi
\ifdefined\SHOWPROOFS
\begin{proof}[Verifying the valuative criterion in Case 2]
  We now prove that up to finite base change $\xi^*$ has a unique stable
  extension $\xi$ in Case 2. Let $\tilde \eta$ be a modification of $\eta_+$ of
  order $a$ and degree $r$.
  By Lemma~\ref{lem:entanglement-and-calibration}, $\eta_+$ and $\tilde\eta$
  uniquely determine $\xi_+$ and $\tilde\xi$, respectively.
  By the definition of the order of $\tilde\eta$
  (Definition~\ref{def:order-and-degree}), we have
  \[
    \big(\operatorname{ord}(\tilde v_1)
    -
    \operatorname{ord}(\tilde v_2)\big) - r\big(\operatorname{ord}(v_{1+}) -
    \operatorname{ord}(v_{2+})\big)
    = ra,
  \]
  where ``$\mathrm{ord}$'' means the vanishing order at the closed point.
  Write $\delta = \operatorname{ord}(v_{1+}) - \operatorname{ord}(v_{2+})$.
  By Lemma~\ref{lem:modification2},
  the stability of $\tilde \xi$ translates to
  \[
    \begin{cases}
      a = b \\
       a \geq - \delta
    \end{cases}
    \quad\text{or}
    \quad
    \begin{cases}
      a>b \\
      a = -\delta
    \end{cases}.
  \]
  There is a unique solution $a = \max\{b, -\delta\}$.
  Hence, there is a unique $\tilde \eta$ such
  that the associated $\tilde\xi$ is stable.
\end{proof}
\fi
\subsection{Case 3: Reparametrizing degree-$d_0$ rational tails}
In this subsection we assume that there are degree-$d_0$ rational tails in the
generic fiber. Possibly after a finite base change,
let $\mathcal E^*_1 ,\ldots,
\mathcal E^*_{\ell}\subset \mathcal C^*$ be the entangled rational tails in the
generic fiber. Let $\mathcal C^*_g\subset \mathcal C^*$ be the union of the other
components, i.e.\ the closure of the complement of $\mathcal E_1^* ,\ldots,
\mathcal E_{\ell}^*$
in $\mathcal C^*$. Note that $\mathcal
C_g^*$ may contain degree-$d_0$ rational tails that are not entangled.

The markings $\mathbf x^*$ are contained in $\mathcal C^*_g$. We view the nodes
as new markings $\mathbf y^*$ on $\mathcal C_g^*$ to obtain a family of
quasimaps over $\operatorname{Spec} K$
\[
  \eta^*_g = (\mathcal C_g^* \to \operatorname{Spec} K,\mathbf x^*,\mathbf y^*,
  u_g^* = u|_{\mathcal C_g^*}).
\]
Also consider the node on $\mathcal E^*_i$ as new marking $z^*_i$ to obtain
\[
  \eta_i^* = (\mathcal E_i^* \to \operatorname{Spec} K, z^*_i, u^*_i =
  u|_{\mathcal E^*_i}).
\]
Recall from Section~\ref{sec:quasimap-with-entangled-tails-moduli} the
definition of $\epsilon_0$-semistable quasimaps.
\begin{lemma}
  \label{lem:glue-tails}
  For any stable extension $\xi$, its underlying family of quasimaps $\eta$ is
  obtained by gluing $\epsilon_0$-semistable extensions $\eta_i$ of $\eta^*_i$ and
  $\epsilon_+$-stable extension $\eta_g$ of $\eta_g^*$ along the markings $z_i$
  and $y_i$.
\end{lemma}
\ifdefined\SHOWPROOFS
\begin{proof}
  It is clear that any such $\eta$ is obtained by gluing some $\epsilon_0$-semistable
  $\eta_g$ and $\eta_i$.
  We need to show that $\eta_g$ must be $\epsilon_+$-stable. Indeed, by
  Lemma~\ref{lem:limit-of-entanglement-case2}, $\eta_g$
  contains no entangled tails. Since constant tails are required to be
  entangled, $\eta_g$ contains no constant tails. Hence $\eta_g$ must be
  $\epsilon_+$-stable.
\end{proof}
\fi
\ifdefined\SHOWPROOFS
\begin{proof}[Verifying the valuative criterion in Case 3]
  Using the notation introduced above, for each choice of $\eta_1 ,\ldots,
  \eta_\ell$, we obtain $\eta$ from the gluing in Lemma~\ref{lem:glue-tails}. Then $\eta$
  uniquely determines $\xi$ by Lemma~\ref{lem:entanglement-and-calibration}. The
  existence and uniqueness of $\eta_g$ follows from the properness of
  $Q^{\epsilon_+}_{g,n+\ell}(X,\beta^\prime)$, where $\beta^\prime$ is the curve
  class of $u|_{\mathcal C^*_g}$. We
  will show that up to base change there is a unique choice of
  $\eta_1 ,\ldots,
  \eta_\ell$ such that $\xi$ is stable.

  We first pick any $\epsilon_0$-semistable extensions $(\eta_{1+} ,\ldots,
  \eta_{\ell+})$ of $(\eta_1^* ,\ldots, \eta_\ell^*)$ as in Case 2,
  possibly after base change.
  It induces $\xi_+$ by Lemma~\ref{lem:entanglement-and-calibration}.
  We are in the situation of Section~\ref{sec:lim-of-entang2}.
  The tuple of integers $\vec a := (a_1 ,\ldots, a_\ell)$ defined there determines
  which tails in the special fiber are entangled.
   Recall that $a_i$ is defined as the order of
  some sections of $\Theta_i$, where $\Theta_i$ is the space of infinitesimal
  smoothings of the node $y_i=z_i$. We have a canonical isomorphism
  \[
    \Theta_i \cong \mathbb M_{{\mathcal E}_{i+}}^{\vee} \otimes  T_{y_i}\mathcal C_{g}.
  \]
  Hence modifying $\mathcal E_{i+}$ changes of $\Theta_i$ and $\mathbb
  M_{{\mathcal E}_{i+}}^{\vee}$ in the same way.

  Consider any modifications $\tilde\eta_1 ,\ldots, \tilde\eta_\ell$ of
  $(\eta_{1+} ,\ldots, \eta_{\ell+})$ of orders
  $a^\prime_1 ,\ldots, a_\ell^\prime$ and degree $r$, in the sense of
  Definition~\ref{def:modification} and Definition~\ref{def:order-and-degree}.
  We replace $\eta_{i+}$ by $\tilde\eta_i$ for each $i$, to obtain $\tilde\eta$
  and $\tilde\xi$ by the same procedure as before, then the $(a_1 ,\ldots,
  a_{\ell})$ is replaced by
  \[
    (ra_1- ra_1^\prime,\ldots, ra_\ell - ra_\ell^\prime).
  \]
  Write $\delta = \operatorname{ord}(v_{1+}) - \operatorname{ord}(v_{2+})$ and
  $|\vec{a}^\prime| = \sum_{j=1}^\ell a_{j}^\prime$.
  Note that we have
  \begin{equation}
    \label{eq:calibration-bundle-and-tails-in-properness}
    \textstyle
    \mathbb M_{\mathcal C_{+}} \cong \bigotimes_{i=1}^{\ell} \mathbb M_{{\mathcal
    E}_{i+}} \otimes  L_0,
\end{equation}
  where $L_0$ is some line bundle defined in terms of $\eta_g$ and will never be
  changed. There is an relation similar to
  \eqref{eq:calibration-bundle-and-tails-in-properness} for $\tilde\eta$. Hence
  we have
  \[
    \operatorname{ord}(\tilde v_1)
    -
    \operatorname{ord}(\tilde v_2) =
  r(|\vec{a}^\prime| + \delta).
  \]

  Let $b_i=0$ if $\mathcal E^*_i$ has no length-$d_0$ base point and
  $b_i=\infty$ otherwise. 
    By Lemma~\ref{lem:limit-of-entanglement-case2} and Lemma~\ref{lem:modification2},
  the stability condition for $\tilde\xi$
  translates to
  \begin{equation}
    \begin{cases}
      \delta + |\vec{a}^\prime|\geq 0;\\
      a^\prime_i \geq  - b_i, i =1 ,\ldots, \ell;\\
      \text{if } \delta + |\vec{a}^\prime| > 0 \text{ then } a_i^\prime =
      - b_i \text{ for all }i =1 ,\ldots, \ell;\\
      \text{if } a^\prime_i > -b_i\text{ then }a_i-a_i^\prime\text{ is maximal
        among } a_1-a_1^\prime ,\ldots, a_\ell-a_\ell^\prime, i=1 ,\ldots, \ell.
    \end{cases}
  \end{equation}
  This system has a unique solution for $(a_1^\prime ,\ldots,
  a_\ell^\prime)$. This completes the proof.
\end{proof}
\fi
\fi
\section{Localization on the master space}
\label{sec:localization}
Consider the $\mathbb C^*$-action on $MQ^{\epsilon_0}_{g,n}(X,\beta)$
defined by scaling $v_1$:
\begin{equation}
  \label{eq:C*-action}
  \lambda\cdot(\pi:\mathcal C \to S,e,u,N,v_1,v_2) =  (\pi:\mathcal C\to
  S,e,u,N,\lambda v_1,v_2) ,\quad \lambda\in \mathbb C^*.
\end{equation}
We will use the virtual localization \cite{graber1999localization, chang2017torus}
to get relations among the fixed-point components.
By a fixed-point component we mean the union of several connected
components of the fixed-point locus.
There will be three types of fixed-point components.
\subsection{$\epsilon_+$-stable quasimaps with entangled tails}
The Cartier divisor $F_+\subset MQ^{\epsilon_0}_{g,n}(X,\beta)$
defined by $v_1=0$ is a fixed-point component.
It is easy to see that
\[
  F_+ \cong \tilde Q^{\epsilon_+}_{g,n}(X,\beta),
\]
by forgetting $(N,v_1,v_2)$.
Under this isomorphism, we have
\[
  [F_+]^{\mathrm{vir}} = [\tilde Q^{\epsilon_+}_{g,n}(X,\beta)]^{\mathrm{vir}}.
\]
The virtual normal bundle is $\mathbb M_+$,
the calibration bundle of $\tilde Q^{\epsilon_+}_{g,n}(X,\beta)$ in
Definition~\ref{def:calibration-bundle}, with a $\mathbb C^*$-action of weight $1$.
We will see that the first Chern class of $\mathbb M_+$ is irrelevant to the
wall-crossing formula.
\subsection{$\epsilon_-$-stable quasimaps}
Similarly, the Cartier divisor $F_-\subset MQ^{\epsilon_0}_{g,n}(X,\beta)$
defined by $v_2=0$ is a fixed component.
We have 
\[
  F_- \cong Q^{\epsilon_-}_{g,n}(X,\beta) \quad \text{and} \quad
  [F_-]^{\mathrm{vir}} = [ Q^{\epsilon_-}_{g,n}(X,\beta)]^{\mathrm{vir}}.
\]
Note that, in particular,  $v_2$ is nonvanishing when $Q^{\epsilon_-}_{g,n}(X,\beta)$ is empty,
e.g.\ when $g=0,n=1,\deg(\beta)=d_0$. When it is nonempty,
the virtual normal bundle is the line bundle $\mathbb M_-^\vee$, the dual of the
calibration bundle $\mathbb M_-$ of $Q_{g,n}^{\epsilon_-}(X,\beta)$. And the
$\mathbb C^*$-action on $M_-^\vee$ has weight $(-1)$. Again its first Chern
class will be irrelevant.

\subsection{The correction terms: the graph space revisited}
\label{sec:graph-space-revisited}
The other fixed-point components will contribute to the so-called correction
terms (the $k\geq 1$ terms) in the wall-crossing formula.
Their contribution involves the coefficient of $q^{\beta}$ in the
$I$-function with $\deg(\beta) = d_0$,
which is defined via the graph space $QG_{0,1}(X,\beta)$. 
Let $QG^*_{0,1}(X,\beta) \subset QG_{0,1}(X,\beta)$ be the open substack where
the domain curve is irreducible. The definition of the $I$-function only
involves $QG^*_{0,1}(X,\beta)$. This allows us to simplify the perfect
obstruction theory, as follows.

Recall that $QG_{0,1}(X,\beta)$
admits a relative perfect obstruction theory
\begin{equation}
  \label{eq:pot-QG-1}
\big(
  R\pi_*((u,f)^*(\mathbb
  T_{[W/G]}\boxplus T_{\mathbb P^1}[0]))\big)^{\vee}
  \longrightarrow
  \mathbb L_{QG_{0,1}(X,\beta)/\tilde{\mathfrak M}_{0,1,d_0}},
\end{equation}
where $\pi:\mathcal C \to QG_{0,1}(X,\beta)$ is the universal curve and $(u,f):
\mathcal C \to [W/G] \times  \mathbb P^1$ is the universal map.
Now restrict it to $QG^*_{0,1}(X,\beta)$ and consider the forgetful morphism
\begin{equation}
  \label{map:graph-space-forgetful}
  QG^*_{0,1}(X,\beta) \longrightarrow \overline M_{0,1}(\mathbb P^1,1),
\end{equation}
forgetting $u$ and the orbifold structure at the unique marking.
It admits a relative perfect obstruction theory
\begin{equation}
  \label{eq:pot-QG-2}
  \big(R\pi_*(u^*\mathbb T_{[W/G]})\big)^\vee \longrightarrow \mathbb
  L_{QG^*_{0,1}(X,\beta)/ \overline M_{0,1}(\mathbb P^1,1)}.
\end{equation}
Note that the evaluation map
$\overline M_{0,1}(\mathbb P^1,1) \to \mathbb P^1$ is an isomorphism.
The two relative perfect obstruction theories \eqref{eq:pot-QG-1}
and~\eqref{eq:pot-QG-2} induce the same equivariant absolute perfect obstruction
theory $\mathbb E_{QG^*}$ (c.f.\ \cite[\S 3.2]{manolache2012virtual}).
Thus we have a distinguished triangle on $QG^*_{0,1}(X,\beta)$
\begin{equation}
  \label{seq:pot-QG}
  \mathbb L_{\overline M_{0,1}(\mathbb P^1,1)} \longrightarrow
  \mathbb E_{QG^*} \longrightarrow  \big(
  R\pi_*(u^*\mathbb T_{[W/G]})
  \big)^\vee
  \overset{+1}{\longrightarrow }.
\end{equation}
We have suppressed some obvious pullbacks.
Recall that $F_{\star,\beta}:=F^{0,\beta}_{\star,0}\subset QG^*_{0,1}(X,\beta)$ is
the fixed-point component where the marking is at $\infty$ and $u$ has a base
point of length $\deg(\beta)=d_0$ at $0$.
Restricting \eqref{seq:pot-QG} to $F_{\star,\beta}$, $\mathbb L_{\overline M_{0,1}(\mathbb
  P^1,1)}$ becomes isomorphic to the constant line bundle formed by the
cotangent space to $\mathbb P^1$ at $\infty$.
Hence it is in the moving part. Hence the morphism
\[
  (\mathbb E_{QG})^{\mathrm{f}}
  \longrightarrow
  \big(R\pi_*(u^*\mathbb T_{[W/G]})|_{F_{\star,\beta}}\big)^{\vee,\mathrm{f}}
\]
between fixed parts is an isomorphism. In other words, the induced virtual cycle
$[F_{\star,\beta}]^{\mathrm{vir}}$ is defined by the absolute perfect
obstruction theory
\[
  \big(R\pi_*(u^*\mathbb T_{[W/G]})|_{F_{\star,\beta}}\big)^{\vee,\mathrm{f}}
  \longrightarrow  \mathbb L_{F_{\star,\beta}}.
\]
Recall that $\mathbf r$ is the locally constant function on $I_{\mu}X$ that
takes value $r$ on $I_{\mu_r}X$.
We define the operational Chow class
\begin{equation}
  \label{eq:curly-I}
  \mathcal I_{\beta}(z) = \frac{1}{e_{\mathbb
      C^*}\big(\big(R\pi_*(u^*\mathbb
    T_{[W/G]})|_{F_{\star,\beta}}\big)^{\mathrm{mv}}\big)}\in A^*_{\mathbb C^*}(F_{\star,\beta}).
\end{equation}
Thus the $I$-function \eqref{eq:small-I} can be rewritten as
\begin{equation}
  \label{eq:I-from-curly-I}
  I(q,z) =  \sum_{\beta} \mathbf r^2 q^{\beta} \check{\mathrm{ev}}_*
  \Big(
  \mathcal I_{\beta}(z)\cap [F_{\star,\beta}]^{\mathrm{vir}}
  \Big) \in A^*_{\mathbb C^*}(I_{\mu}X).
\end{equation}
\begin{remark}
  In the formation of \eqref{map:graph-space-forgetful} we have taken the
  underlying coarse moduli of the domain curves. This causes no trouble to the
  relative perfect obstruction theory, 
  since the forgetful morphism $\mathfrak M^{\mathrm{orbi}}_{0,1} \to \mathfrak
  M^{\mathrm{coar}}_{0,1}$ is an isomorphism over the locus where the curve is smooth. Here
  $\mathfrak M^{\mathrm{orbi}}_{0,1}$ is the moduli of curve with a
  gerbe marking (without trivialization of the gerbe), and $\mathfrak
  M^{\mathrm{coar}}_{0,1}$ is the moduli of non-orbifold curves.
\end{remark}
\subsection{The correction terms: $g=0, n=1, \deg(\beta)=d_0$ case}
We now study the case $g=0, n=1$ and $\deg(\beta)=d_0$. In this case, the curve
must be irreducible and $v_2$ is never zero. The fixed-point component $F_{\beta}$
parametrizes
\begin{align*}
  F_{\beta} = \{\xi\mid \text{the domain curve is a single fixed tail, $v_1\neq
  0, v_2\neq 0$} \}.
\end{align*}
Let $C$ be the domain curve and $x_{\star}$ be the unique marking. Let $\mathbf
r_{\star}$ be the pullback of $\mathbf r$ via the evaluation map at $x_{\star}$.
The calibration bundle is by definition the relative cotangent space at
$x_{\star}$. Fix once and for all a nonzero tangent vector $v_{\infty}$ at
$\infty$ to $\mathbb P^1$. There is a unique morphism $C \to \mathbb P^1$
sending $x_{\star}$ to $\infty$, the unique base point to $0$, and sending
$(v_2/v_1)^{\otimes \mathbf r_\star}$ to $v_{\infty}$. This determines a point
in $QG_{0,1}(X,\beta)$, which is $\mathbb C^*$-fixed since $C$ is a fixed tail.
Hence it lands in the fixed-point locus $F_{\star,\beta}\subset QG_{0,1}(X,
\beta)$. By working over an arbitrary base scheme instead, we obtain a morphism
\[
  F_{\beta} \longrightarrow  F_{\star,\beta}.
\]
\begin{lemma}
  \label{lem:correction-term-special-degree}
  This morhphism is \'etale of degree $\mathbf r_\star$.
\end{lemma}
\ifdefined\SHOWPROOFS
\begin{proof}
  It follows from the observation that lifting a morphism $S\to F_{\star,\beta}$
  to $S\to F_\beta$ is the same as choosing an $\mathbf r_\star$-th root of
  $v_{\infty}$. Note that this makes sense since in the definition of graph spaces
  we are using twisted curves with trivialized gerbes.
\end{proof}
\fi

We now come to the perfect obstruction theory.
\begin{lemma}
  \label{lem:correction-term-special-contr}
  Via the morphism above, the pullback of $[F_{\star,\beta}]^{\mathrm{vir}}$ is
  equal to $[F_{\beta}]^{\mathrm{vir}}$, and
  \[
    \frac{1}{e_{\mathbb C^*}(N^{\mathrm{vir}}_{F_{\beta}/MQ^{\epsilon_0}_{0,1}(X,\beta)})} =
    (\mathbf r_\star z)\cdot \mathcal I_{\beta}(\mathbf r_\star z).
  \]
\end{lemma}
\ifdefined\SHOWPROOFS
\begin{proof}
  Recall that $MQ^{\epsilon_0}_{0,1}(X,\beta)$ has a relative perfect
  obstruction theory
  \begin{equation}
    \label{seq:pot-MQ-comparison}
    (R\pi_*u^*\mathbb T_{[W/G]})^{\vee}
    \longrightarrow
    \mathbb L_{MQ^{\epsilon_0}_{0,1}(X,\beta)/M \tilde{\mathfrak M}_{0,1,d_0}}.
  \end{equation}
  Restricting to $F_{\beta}$, $\mathbb L_{M
      \tilde{\mathfrak M}_{0,1,d_0}}$ becomes canonically isomorphic to the cotangent space to
    $\mathbb P^1$ at $0$, shifted into degree $1$. Let $\mathbb E_{MQ}$ be
    the absolute perfect obstruction theory induced by
    \eqref{seq:pot-MQ-comparison}. Then the morphism between fixed parts
    \[
      (\mathbb E_{MQ}|_{F_{\beta}})^{\mathrm{f}}
      \longrightarrow
      (R\pi_*u^*\mathbb T_{[W/G]}|_{F_{\beta}})^{\vee,\mathrm{f}}
    \]
    is an isomorphism. The Lemma follows from comparing this to
    Section~\ref{sec:graph-space-revisited}.
\end{proof}
\fi

\subsection{The correction terms, $2g - 2 + n + \epsilon_0 d > 0$ case}
\label{sec:correction-term-main-case}
\begin{lemma}
  \label{lem:fixed-point-with-fixed-tails}
  Let
  \[
    \xi = (C,\mathbf x,e,u,N,v_1,v_2)\in MQ^{\epsilon_0}_{g,n}(X,\beta)(\mathbb C)
  \]
  be an $\epsilon_0$-stable quasimap with calibrated tails. Suppose $v_1\neq 0$ and $v_2\neq 0$.
  Then $\xi$ is $\mathbb C^*$-fixed if and only if
\begin{enumerate}
\item there is at least one degree-$d_0$ tail, and
\item each entangled tail is a fixed tail (c.f.\ Definition~\ref{def:fixed-tails}).
\end{enumerate}
\end{lemma}
\ifdefined\SHOWPROOFS
\begin{proof}
  Since $v_1$ and $v_2$ are both nonzero, $v_1/v_2$ is a nonzero section of the
  calibration bundle. Thus $\xi$ is fixed if and only if the action of
  $\mathrm{Aut}(C,e,u)$ on the calibration bundle induces a surjection
  $\mathrm{Aut}(C,e,u) \twoheadrightarrow \mathbb C^*$. Thus the Lemma follows
  from Lemma~\ref{lem:auto-eta} and the description of the calibration bundle in
  terms of $\Theta_i$ in Section~\ref{sec:cali-bundle}.
\end{proof}
\fi
For each $(k+1)$-tuple $\vec\beta = (\beta^\prime,\beta_1 ,\ldots, \beta_k)$ of
effective curve classes such that $\beta = \beta^\prime + \beta_1 + \cdots +
\beta_k$ and $\deg(\beta_i) = d_0$ for each $i=1 ,\ldots, k$,
we have a fixed-point component (which may be empty)
\begin{align*}
  F_{\vec\beta} = \{\xi\mid &\xi \text{ has exactly $k$ entangled tails,}\\
    &\text{which are all fixed tails,   of degrees }\beta_1,\ldots , \beta_k\}.
\end{align*}
It is easy to see from Lemma~\ref{lem:fixed-point-with-fixed-tails} that
$F_+,F_{-}$ and $F_{\vec\beta}$ contain all the fixed points.
To define $F_{\vec\beta}$ as a substack, we need the following lemma.
\begin{lemma}
  Each $F_{\vec\beta}$ is closed. Hence it is an open and closed substack of
  the fixed-point substack $(MQ^{\epsilon_0}_{g,n}(X,\beta))^{\mathbb C^*}$.
\end{lemma}
\ifdefined\SHOWPROOFS
\begin{proof}
  It is easy to see that $F_{\vec\beta}$ is constructible.
  Consider any $1$-parameter family of
  objects in $F_{\vec\beta}$.
  It is easy to see that the limit cannot be in $F_+$ or in $F_-$. Hence the
  limit must be in $F_{\vec\beta^\prime}$ for some $(1+k^\prime)$-tuple
  $\vec\beta^\prime$.
  In the limit there are at least $k$ fixed tails and at most $k$
  entangled tails (Lemma~\ref{lem:limit-of-entanglement-case2}).
  Recall that the stability condition (Definition~\ref{def:mspace-stability})
  requires that any fixed tail is an entangled tail (c.f.\
  Remark~\ref{rmk:tails-comparison}). Hence we must
  have $k = k^\prime$. Thus the entangled tails in the limit are precisely the
  limit of the entangled tails of the generic member of the family. Hence we must
  have $\vec\beta^\prime = \vec\beta$. Hence the limit is always in
  $F_{\vec\beta}$ and $F_{\vec\beta}$ is closed.
\end{proof}
\fi

We now describe the structure of $F_{\vec\beta}$. The goal is
Lemma~\ref{lem:structure-F-beta}. We first show that for a family of objects in
$F_{\vec\beta}$ over a (possibly nonreduced) base scheme, we can split off the
entangled tails. More precisely, recall that $\mathfrak E^*_{k-1}\subset
\tilde{\mathfrak M}_{g,n,d}$ is the locally closed smooth substack where there
are exactly $k$ entangled tails (Lemma~\ref{lem:entangled-tails-from-gluing}).
\begin{lemma}
  \label{lem:F-beta-factor-through-E-k-1}
  The forgetful morphism $F_{\vec\beta} \to   \tilde{\mathfrak
    M}_{g,n,d}$ factors through $\mathfrak E^*_{k-1}$.
\end{lemma}
\ifdefined\SHOWPROOFS
\begin{proof}
  It is clear that any closed point of
  $F_{\vec\beta}$ is mapped into $\mathfrak E^*_{k-1}$. Hence
  $({F_{\vec\beta}})^{\mathrm{red}}$ is mapped to $\mathfrak E^*_{k-1}$. By
  Lemma~\ref{lem:normal-E_k-1} and Lemma~\ref{lem:auto-eta}, the pullback of the
  normal bundle $N_{\mathfrak E^*_{k-1}/\tilde{\mathfrak M}_{g,n,d}}$
  has a nontrivial $\mathbb C^*$-action. Since $F_{\vec\beta}$ is $\mathbb C^*$-fixed,
  it must be mapped into $\mathfrak E^*_{k-1}$.
\end{proof}
\fi
Recall that $\mathfrak Z_{(k)} \subset \mathfrak U_{k}$ is the proper transform
of the locus $\mathfrak Z_{k} \subset \mathfrak M_{g,n,d}^{\mathrm{wt,ss}}$
where there are at least $k$ rational tails of degree $d_0$.
By the construction of $\tilde {\mathfrak M}_{g,n,d}$,
we have a natural morphism
$\mathfrak E^*_{k-1} \subset \mathfrak E_{k-1} \to \mathfrak Z_{(k)}$.
Also recall \eqref{map:quotient-by-S} that
\[
  \tilde{\mathrm{gl}}_k: \tilde{\mathfrak M}_{g,n+k,d-kd_0} \times^\prime {\big( \mathfrak
      M_{0,1,d_0}^{\mathrm{wt,ss}}\big)}^k \longrightarrow \mathfrak Z_{(k)}
\]
is the morphism that glues the universal curves of
${(\mathfrak M_{0,1,d_0}^{\mathrm{wt,ss}})}^k$
to the last $k$ markings of the universal curve of
$\tilde{\mathfrak M}_{g,n+k,d-kd_0}$ as degree-$d_0$ rational tails. It is an
\'etale morphism whose fibers are the choices of 
the ordering of those tails and the trivialization of gerbes
at the new nodes.
Let $\mathbf r_i$ be the locally constant function whose value is the order the
automorphism group of the $i$-th new node.
Then ${\tilde{\mathrm{gl}}_k^*}$ is of degree
\begin{equation}
  \label{eq:degree-of-q}
  \frac{k!}{\prod_{i=1}^k \mathbf r_i}.
\end{equation}

We form the fibered diagram
\begin{equation}
  \label{diag:splitting-nodes}
  \begin{tikzcd}
    \tilde{\mathrm{gl}}_k^*F_{\vec\beta} \arrow[r]\arrow[d]
    &\arrow[r]\arrow[d] \tilde{\mathrm{gl}}_k^*\mathfrak E^*_{k-1} &
    \tilde{\mathfrak M}_{g,n+k,d-kd_0} \times^\prime {\big( \mathfrak
      M^{\mathrm{wt,ss}}_{0,1,d_0}\big)}^k
    \arrow[d,"\tilde{\mathrm{gl}}_{k}"] \\
    F_{\vec\beta} \arrow[r] & \mathfrak E^*_{k-1}\arrow[r] & \mathfrak Z_{(k)}
  \end{tikzcd} .
\end{equation}
Let $\mathcal C_{{\tilde{\mathrm{gl}}_k^*}F_{\vec\beta}}$ be the universal curve
over ${\tilde{\mathrm{gl}}_k^*}F_{\vec\beta}$
(i.e.\ the pullback of the universal curve of $MQ^{\epsilon_0}_{g,n}(X,\beta)$).
Let $\mathcal C_{\beta^\prime}$ be the pullback to
${\tilde{\mathrm{gl}}_k^*}F_{\vec\beta}$ of the universal curve of
$\tilde{\mathfrak M}_{g,n+k,d-kd_0}$, and $\mathcal E_1 ,\ldots, \mathcal E_k$ be
the pullback to ${\tilde{\mathrm{gl}}_k^*}F_{\vec\beta}$ of the universal curve
of $(\mathfrak M_{0,1})^k$. Then $\mathcal
C_{{\tilde{\mathrm{gl}}_k^*}F_{\vec\beta}}$ is obtained by gluing $\mathcal E_1
,\ldots, \mathcal E_k$ to the last $k$ markings of $\mathcal C_{\beta^\prime}$.
We denote the node of $\mathcal C_{{\tilde{\mathrm{gl}}_k^*}F_{\vec\beta}}$ on
$\mathcal E_i$ by $p_i$.
Restricting the universal quasimap to $\mathcal C_{\beta^\prime}$ 
  gives rise to
\begin{equation}
  \label{map:F-to-Q-n+k}
  {\tilde{\mathrm{gl}}_k^*}F_{\vec\beta} \longrightarrow
  \tilde Q^{\epsilon_+}_{g,n+k}(X,\beta^\prime).
\end{equation}

Let $\mathbb M_{\beta^\prime}$ be the calibration bundle on $\tilde{\mathfrak
  M}_{g,n+k,d-kd_0}$, and $\mathbb M_{\vec\beta}$ be the calibration
bundle on $F_{\vec\beta}$, i.e.\  the restriction of the calibration bundle on
$\tilde{\mathfrak M}_{g,n,d}$. 
We will suppress obvious pullback notation below.
Then we have canonical isomorphisms of line bundles
on ${\tilde{\mathrm{gl}}_k^*}F_{\vec\beta}$
\[
  \mathbb M^{\vee}_{\beta^\prime}\otimes T_{p_1}\mathcal E_1 \otimes T_{p_1}\mathcal
  C_{\beta^\prime}
  \otimes  \cdots \otimes
  T_{p_k}\mathcal E_k \otimes T_{p_k}\mathcal C_{\beta^\prime}
  \cong \mathbb M^{\vee}_{\vec\beta} \overset{v_2/v_1}{\cong} \mathcal
  O_{{\tilde{\mathrm{gl}}_k^*}F_{\vec\beta}},
\]
where $T_{p_i}\mathcal E_i$ is the line bundle formed by the relative (orbifold)
tangent spaces along $p_i$ (c.f.\ Section~\ref{sec:orbifold-psi-classes}), etc.
By Lemma~\ref{lem:normal-E_k-1},
we have canonical isomorphisms on ${\tilde{\mathrm{gl}}_k^*}F_{\vec\beta}$
\[
  T_{p_1}\mathcal E_1 \otimes T_{p_1}\mathcal C_{\beta^\prime} \cong \cdots \cong
  T_{p_k}\mathcal E_k \otimes T_{p_k}\mathcal C_{\beta^\prime}.
\]
We write $\Theta$ for $T_{p_1}\mathcal E_1 \otimes  T_{p_1} \mathcal
C_{\beta^\prime}$ as a line bundle on ${\tilde{\mathrm{gl}}_k^*}F_{\vec\beta}$.
Then we have a canonical isomorphism
\begin{equation}
  \label{map:root}
  \Theta^{\otimes k} \cong \mathbb M^{\vee}_{\beta^\prime}
\end{equation}
on $\tilde{\mathrm{gl}}_k^*F_{\vec\beta}$.
Let
\begin{equation}
  \label{map:Y-to-Q-n+k}
  Y \longrightarrow  \tilde{Q}^{\epsilon_+}_{g,n+k}(X,\beta^\prime)
\end{equation}
be the stack of $k$-th roots of the pullback to
$\tilde{Q}^{\epsilon_+}_{g,n+k}(X,\beta^\prime)$ of $\mathbb
M^{\vee}_{\beta^\prime}$, 
and $L$ be the universal $k$-th root.
Then \eqref{map:root} gives rise to
\begin{equation}
  \label{map:to-Y}
  {\tilde{\mathrm{gl}}_k^*}F_{\vec\beta} \longrightarrow  Y.
\end{equation}

We now come to the entangled tails.
We want to compare the restriction of the quasimaps to $\mathcal E_i$ with the
fixed-domain quasimaps parameterized by the graph space.
However, $\mathcal E_i$ is in general not a trivial family of curves.
Indeed, it is twisted by a line bundle on $Y$, in the following sense.
Let $V_i^* \to Y$ be the total space of
$L^{\otimes \mathbf r_i} \otimes (T_{p_i}\mathcal C_{\beta^\prime})^{\otimes
-\mathbf r_i}$ minus its zero section.
Define
\[
  Y^\prime = V_1^* {\times_Y} \cdots \times_Y V_k^*,
\]
and set
\[
  \mathcal E_i^{\prime} = \mathcal E_i \times_Y Y^{\prime} = \mathcal
  E_i\times_{{\tilde{\mathrm{gl}}_k^*}F_{\vec\beta}}
  ({\tilde{\mathrm{gl}}_k^*}F_{\vec\beta} \times_YY^\prime).
\]
Thus an $S$-point of ${\tilde{\mathrm{gl}}_k^*}F_{\vec \beta}\times_Y Y^\prime$
is a morphism
$f:S\to {\tilde{\mathrm{gl}}_k^*}F_{\vec\beta}$
together with nonvanishing sections
$s_i \in H^0(S,f^*(T_{p_i}\mathcal E_i)^{\otimes \mathbf r_i})$,
$i=1 ,\ldots, k$. Moreover,
\[
  \mathcal E_i^\prime \times_S ({\tilde{\mathrm{gl}}_k^*}F_{\vec\beta}\times_Y
  Y^\prime) \longrightarrow  S
\]
is
the $i$-th entangled tail of the induced family of curve over $S$. We will see
that the coarse moduli of the fibers are canonically isomorphic to $\mathbb P^1$.

Note that $(T_{p_i}\mathcal E_i)^{\otimes \mathbf r_i}$ is canonically
isomorphic to the tangent space to the coarse moduli of $\mathcal E_i$ at the
node.
Fix a nonzero tangent vector $v_\infty$ to $\mathbb P^1$ at $\infty$.
There are unique morphisms
\[
  f_i : \mathcal E^\prime_i \longrightarrow  \mathbb P^1, \quad i=1 ,\ldots, k,
\]
sending the marking $p_i$ to $\infty$, the unique base point on $\mathcal E_i$
to $0$ and the vector $s_i$ to $v_\infty$.
Consider the standard $(\mathbb C^*)^k$-action on $Y^\prime$ by scaling the $(s_1
,\ldots, s_k)$,
and consider its trivial action on $\mathcal E_i$.
This induces an action on $\mathcal E_i^\prime$.
Also consider the $\mathbb C^*$-action on $\mathcal E_i$ induced by the $\mathbb
C^*$-action \eqref{eq:C*-action} on the master space.
Lift it to a $\mathbb C^*$-action
on $\mathcal E^\prime$ by acting trivially on $Y^\prime$.
Finally recall the $\mathbb C^*$-action
on $\mathbb P^1$ defined by \eqref{eq:C*-action-on-P1}.
\begin{lemma}
\label{lem:equivariant-action-P1-bundle}
For any $(\lambda,t) = (\lambda,t_1 ,\ldots, t_k)\in \mathbb C^*\times (\mathbb C^*)^k$,
the following diagram
\begin{equation}
  \label{cd:action-on-E-prime}
  \begin{tikzcd}
    \mathcal E_i^\prime \arrow[r,"{f_i}"] \arrow[d,"{(\lambda^k,t)}"']& \mathbb P^1
    \arrow[d,"{\lambda^{\mathbf r_i} t_i^{-1}}"] \\
    \mathcal E_i^\prime \arrow[r,"{f_i}"] & \mathbb P^1
  \end{tikzcd}.
\end{equation}
is commutative.
\end{lemma}
\ifdefined\SHOWPROOFS
\begin{proof}
  This is straightforward. It suffices to verify that the two composite arrows $\mathcal
  E_i^\prime \rightrightarrows \mathbb P^1$ both send $s_i$ to
  $\lambda^{\mathbf r_i} t_i^{-1}v_\infty$.
  Note that the first $\mathbb C^*$ acts
  diagonally on $T_{p_1}\mathcal E_1\oplus \cdots \oplus T_{p_k}\mathcal E_k$
  (c.f.~Lemma~\ref{lem:auto-eta} and the proof of Lemma~\ref{lem:finite-auto}).
  Hence its weight on the calibration bundle is $k$ times its weight on
  $T_{p_i}\mathcal E_i$.
\end{proof}
\fi

Consider the family of curves $\mathcal E_i^\prime \to
{\tilde{\mathrm{gl}}_k^*}F_{\vec\beta} \times_YY^\prime$ introduced above. We
have quasimaps $u^\prime_i:\mathcal E_i^\prime \to [W/G]$
coming from the quasimaps $\mathcal E_i \to [W/G]$. Combining this with the
$f_i$ above, we
obtain a morphism
\[
  {\tilde{\mathrm{gl}}_k^*}F_{\vec\beta} \times_YY^\prime 
  \longrightarrow QG_{0,1}(X,\beta_i).
\]
Since $\mathcal E_i$ is a fixed tail, so is $\mathcal E_i^\prime$. Hence this
map lands in $F_{\star,\beta_i}$.
Since $F_{\star,\beta_i}$ is fixed by the $\mathbb C^*$-action induced by the
action on $\mathbb P^1$, this map is invariant
\footnote{Indeed one should also consider the cocycle condition
  of the $2$-morphisms since $QG_{0,1}(X,\beta)$ is a stack. It is satisfied
  since the action only involves maps to $\mathbb P^1$, which is a scheme.}
under the $(\mathbb C^*)^k$-action on $Y^\prime$ by
Lemma~\ref{lem:equivariant-action-P1-bundle}. Hence it descends to
\begin{equation}
  \label{map:to-F}
  {\tilde{\mathrm{gl}}_k^*}F_{\vec\beta} \longrightarrow  F_{\star,\beta_i}.
\end{equation}

Consider the evaluation map at the last $k$ markings
\[
  \mathrm{ev}_Y: Y \longrightarrow (I_\mu X)^k,
\]
and the evaluation map
\[
  \operatorname{ev}_{\star,\beta_i}: F_{\star,\beta_i} \longrightarrow I_{\mu}X
\]
at the unique marking. Let $\check{\operatorname{ev}}_{\star,\beta_i}$ be the
composition of $\operatorname{ev}_{\star,\beta_i}$ with the involution
$I_{\mu}X\to I_{\mu}X$ inverting the banding.
Using $\mathrm{ev}_Y$ and
$\check{\operatorname{ev}}_{\star,\beta_i}$ we form the fiber product
\[
  \textstyle
  Y \times_{(I_{\mu}X)^k}\prod_{i=1}^k F_{\star,\beta_i}.
\]
\begin{lemma}
  \label{lem:structure-F-beta}
  The morphism
  \[
    \textstyle
    \varphi: {\tilde{\mathrm{gl}}_k^*}F_{\vec\beta} \overset{}{\longrightarrow}
    Y \times_{(I_\mu X)^k}\prod_{i=1}^k F_{\star,\beta_i},
  \]
  induced by \eqref{map:to-Y} and \eqref{map:to-F} is representable, finite,
  \'etale, of degree $\prod_{i=1}^k \mathbf r_i$.
\end{lemma}
\ifdefined\SHOWPROOFS
\begin{proof}
  It is a morphism over $Y$. To
  prove the lemma,
  it suffices to work over a faithfully flat (representable) cover of $Y$.
  We construct a cover $Y^{\prime\prime}\to Y$ over which  $L$
  and $T_{p_i}\mathcal C_{\beta^\prime}$ are canonically trivialized.
  The construction is similar to that of $Y^\prime\to Y$ and we
  omit the details.
  After base change to $Y^{\prime\prime}$ we will see that $\varphi$ is isomorphic to
  a finite \'etale cover.

  Given an $S$-point $\xi$ of $Y^{\prime\prime} \times_{(I_\mu
    X)^k}\prod_{i=1}^k F_{\star,\beta_i}$, we want to show that a lifting to an
  $S$-point of $Y^{\prime\prime}\times_{Y}
  {\tilde{\mathrm{gl}}_k^*}F_{\vec\beta}$ is equivalent to the
  choice of an $\mathbf r_i$-th root of a nonzero section of a certain line
  bundle.
  
  Given $\xi$,
  we first glue the rational curves $\mathcal E^{\prime\prime}_i$ from
  $F_{\star,\beta_i}$ to the $(n+i)$-th marking of the curves $\mathcal
  C^{\prime\prime}_{\beta^\prime}$ from $Y$, and glue the quasimaps.
  Let $\mathcal C^{\prime\prime}_{\beta}\to S$ be the family of curves thus obtained.
  Note that the tails $\mathcal E_{i}^{\prime\prime}$ are ordered and the new
  nodes are trivialized gerbes.
  To obtain an $S$-point of
  $Y^{\prime\prime}\times_{Y}{\tilde{\mathrm{gl}}_k^*}F_{\vec\beta}$, we need to define the
  entanglement (i.e.\ a morphism $S\to \tilde{\mathfrak M}_{g,n,d}$ that lifts
  the classifying morphism $S\to \mathfrak M_{g,n,d}^{\mathrm{wt,ss}}$) and the
  calibration (i.e.\ $N,v_1$ and $v_2$).

  For the entanglement, we need to make a choice.
  Let $p_i$ be the new node on $\mathcal E^{\prime\prime}_i$ that comes from gluing.
  Recall that $v_{\infty}$ is a fixed nonzero tangent vector at $\infty$ to
  $\mathbb P^1$. It gives a canonical trivialization $s_i$ of $(T_{p_i}\mathcal
  E^{\prime\prime}_i)^{\otimes \mathbf r_i}$
  via the (relative) tangent map of $\mathcal E^{\prime\prime}_i \to \mathbb
  P^1$.
  We choose an $\mathbf r_i$-th root of $s_i$ for $i=1 ,\ldots, k$, and obtain
  \begin{equation}
    \label{map:trivialization-tails}
    \mathcal O_S \cong T_{p_1}\mathcal E^{\prime\prime}_1 \cong \cdots \cong
    T_{p_k}\mathcal E^{\prime\prime}_k .
  \end{equation}

  Using $S\to Y^{\prime\prime}$, we have
  \[
    \mathcal O_S \cong T_{p_1}\mathcal C^{\prime\prime}_{\beta^\prime} \cong
    \cdots \cong T_{p_k}\mathcal C^{\prime\prime}_{\beta^\prime},
  \]
  since we have fixed such isomorphisms on $Y^{\prime\prime}$.
  Thus we obtain nonvanishing sections of $T_{p_i}\mathcal
  C^{\prime\prime}_{\beta^\prime}\otimes T_{p_i}\mathcal E^{\prime\prime}_i$
  for $i=1 ,\ldots, k$.
  By Lemma~\ref{lem:structure-E_k-1} and
  Lemma~\ref{lem:entangled-tails-from-gluing}, this is equivalent to a lifting of
  the classifying morphism $S \to \mathfrak Z_{(k)}$ to $S \to \mathfrak E_{k-1}^*
  \subset \mathfrak E_{k-1}$.
  Recall that we have $\mathfrak E_{k-1}^* \subset \tilde{\mathfrak M}_{g,n,d}$.
  Thus we obtain $S\to
  \tilde{\mathfrak M}_{g,n,d}$.  This defines the entanglement and $\mathcal
  E_1 ,\ldots, \mathcal E_k$ are precisely the entangled tails.

  For the calibration, we first set $N$ to be the trivial bundle and set $v_2 = 1$.
  Let $\mathbb M_{\beta}$ be the calibration bundle of
  $\mathcal C^{\prime\prime}_{\beta}$.
  Using the relations
  $L^{\otimes k} \cong \mathbb M^{\vee}_{\beta^\prime}$ on $Y$,
  the isomorphisms
  \[
    \mathbb M_{\beta^\prime}^{\vee}\otimes T_{p_1}\mathcal E^{\prime\prime}_1
    \otimes T_{p_1}\mathcal C^{\prime\prime}_{\beta^{\prime}}
    \otimes  \cdots \otimes
    T_{p_k}\mathcal E^{\prime\prime}_k \otimes T_{p_k}\mathcal
    C^{\prime\prime}_{\beta_0} \cong \mathbb M_{\beta}^{\vee}
  \]
  on $S$, the trivializations of (the pullbacks of) $L,T_{p_i} \mathcal
  C_{\beta^\prime}$ coming from $Y^{\prime\prime}$, and the
  isomorphisms~\eqref{map:trivialization-tails}, we obtain a trivialization of
  $\mathbb M_{\beta}$. Set $v_1 = v_1/v_2$ to be that trivialization. This defines
  the calibration, independent of any choice.

  In summary, we have obtained a natural equivalence of categories between the
  groupoids of $S$-points of $Y^{\prime\prime}\times_Y
  {\tilde{\mathrm{gl}}_k^*}F_{\vec\beta}$ and the groupoids of $S$-points of
  $Y^{\prime\prime} \times_{(I_\mu X)^k}\prod_{i=1}^k F_{\star,\beta_i}$ together
  with a choice of an $\mathbf r_i$-th root of each $s_i$, which is nonvanishing.
  This proves the Lemma.
\end{proof}
\fi

Let $[Y]^{\mathrm{vir}}\in A_*(Y)$ be the flat pullback of
$[\tilde Q^{\epsilon_+}_{g,n+k}(X,\beta^\prime)]^{\mathrm{vir}}$, and let
$[{\tilde{\mathrm{gl}}_k^*}F_{\vec\beta}]^{\mathrm{vir}}$ be the flat pullback
of $[F_{\vec\beta}]^{\mathrm{vir}}$,
the virtual cycle defined by the fixed part of the absolute perfect obstruction
theory.
Let $\tilde{\psi}(\mathcal E_i)$ be the orbifold $\psi$-class of the rational tail
$\mathcal E_i$ at the node and $\tilde\psi_{n+i}$ be the orbifold $\psi$-class
at the $(n+i)$-th marking of ${\tilde Q}^{\epsilon_+}_{g,n+k}(X,\beta^\prime)$.
Let $\psi(\mathcal E_i)$ and $\psi_{n+i}$ be the coarse $\psi$-classes (c.f.\
Section~\ref{sec:orbifold-psi-classes}). Recall from
Section~\ref{sec:boundary-divisor}  that 
the divisor $\mathfrak D_{i}\subset \tilde{\mathfrak M}_{g,n,d}$ is the closure
of the locus where there are exactly ${(i+1)}$ entangled tails. Recall $\mathcal
I_\beta(z)$ from \eqref{eq:curly-I}.
\begin{lemma}
  \label{lem:obstruction-theory-F-beta}
  Via the morphism $\varphi$ in Lemma~\ref{lem:structure-F-beta}, we have
  \[
    [{\tilde{\mathrm{gl}}_k^*}F_{\vec\beta}]^{\mathrm{vir}} = \varphi^*([Y]^{\mathrm{vir}}
    \underset{(I_{\mu}X)^{k}}{\times} \textstyle\prod_{i=1}^k
    [F_{\star,\beta_i}]^{\mathrm{vir}}),
  \]
  and
  \begin{equation}
    \label{eq:normal-bundle-F-beta}
    \begin{aligned}
      \frac{1}
      {e_{\mathbb C^*}(
        N^{\mathrm{vir}}_{F_{\vec\beta}/MQ^{\epsilon_0}_{g,n}(X,\beta)}|_{{\tilde{\mathrm{gl}}_k^*}F_{\vec\beta}}
        )}
      = & \frac {\prod_{i=1}^k(\frac{\mathbf r_i}{k}z+\psi(\mathcal E_i))}
      {- \frac{z}{k} -\tilde\psi(\mathcal E_1) - \tilde\psi_{n+1} - \sum_{i=k}^\infty[\mathfrak
        D_i]}
      \cdot\\
      &  \mathcal I_{\beta_1}(\frac{\mathbf
        r_1}{k}z+\psi(\mathcal E_1))\boxtimes \cdots \boxtimes \mathcal
      I_{\beta_k}(\frac{\mathbf r_k}{k}z+\psi(\mathcal E_k)).
    \end{aligned}
  \end{equation}
\end{lemma}
\begin{remark}
  \label{rmk:entanglement-implies-equation-of-psi}
  Note that since $\mathcal E_1 ,\ldots, \mathcal E_k$ are the entangled
  tails, we have
  \[
    \tilde\psi(\mathcal E_1) + \tilde\psi_{n+1} = \cdots =
    \tilde\psi(\mathcal E_k) + \tilde\psi_{n+k}
  \]
  on ${\tilde{\mathrm{gl}}_k^*} F_{\vec \beta}$ (c.f.\ Lemma~\ref{lem:normal-E_k-1}).
\end{remark}
\ifdefined\SHOWPROOFS
\begin{proof}[Proof of Lemma~\ref{lem:obstruction-theory-F-beta}]
  Let $M^* \mathfrak E^*_{k-1}\subset M\tilde{\mathfrak M}_{g,n,d}$
  be the intersection of the open substack $\{v_1\neq 0, v_2\neq 0\}$ and the preimage of
  the locally closed substack $\mathfrak E^*_{k-1}\subset \tilde{\mathfrak M}_{g,n,d}$.
  Using $M^* \mathfrak E^*_{k-1} \to \mathfrak  E^*_{k-1}\to \mathfrak
  Z_{(k)}$, we form the pullback ${\tilde{\mathrm{gl}}_k^*}(M^*\mathfrak
  E_{k-1}^*)$, where
  \[
    \tilde{\mathrm{gl}}_k:
    \tilde{\mathfrak M}_{g,n+k,d-kd_0} \times^\prime
    \big( \mathfrak
        M_{0,1,d_0}^{\mathrm{wt,ss}}\big)^k \longrightarrow \mathfrak Z_{(k)}
  \]
  is the gluing morphism in Lemma~\ref{lem:structure-Z_k}, which is \'etale.
  Note that ${\tilde{\mathrm{gl}}_k^*}F_{\vec\beta} \to
  M\tilde{\mathfrak M}_{g,n,d}$ factors through
  ${\tilde{\mathrm{gl}}_k^*}(M^*\mathfrak E_{k-1}^*)$, by
  Lemma~\ref{lem:F-beta-factor-through-E-k-1}.

  Let $\pi_1: \mathcal C_1 \to
  {\tilde{\mathrm{gl}}_k^*}F_{\vec\beta}$ be the universal curve and $u_1:\mathcal C_1\to [W/G]$ be
  the universal map. Let $\mathbb E_{MQ}$ be the absolute perfect
  obstruction theory on $MQ^{\epsilon_0}_{g,n}(X,\beta)$. We write $\mathbb
  E_1 = \big(R\pi_{1*}u_1^*\mathbb T_{[W/G]})^{\vee}$. Then we have the distinguished
  triangle on ${\tilde{\mathrm{gl}}_k^*}F_{\vec\beta}$
  \[
    \mathbb L_{M\tilde{\mathfrak M}_{g,n,d}}|_{{\tilde{\mathrm{gl}}_k^*}F_{\vec\beta}}
    \longrightarrow
    \mathbb E_{MQ}|_{{\tilde{\mathrm{gl}}_k^*}F_{\vec\beta}}
    \longrightarrow
    \mathbb E_1
    \overset{+1}{\longrightarrow}.
  \]
    We claim that we have a natural isomorphism
    \[
      (\mathbb L_{M\tilde{\mathfrak M}_{g,n,d}}|_{{\tilde{\mathrm{gl}}_k^*}F_{\vec\beta}})^{\mathrm{f}}
      \cong
      \mathbb L_{\tilde{\mathfrak M}_{g,n+k,d-d_0k}}|_{{\tilde{\mathrm{gl}}_k^*}F_{\vec\beta}}.
    \]
    Indeed, the normal bundle of $M^* \mathfrak E_{k-1}^*$ in $M\tilde{\mathfrak M}_{g,n,d}$
  is isomorphic to the normal bundle of $\mathfrak E^*_{k-1}$ in
  $\tilde{\mathfrak M}_{g,n,d}$, which is in the moving
  part by Lemma~\ref{lem:normal-E_k-1}. Hence
  \begin{equation*}
    (\mathbb L_{M\tilde{\mathfrak
        M}_{g,n,d}}|_{{\tilde{\mathrm{gl}}_k^*}F_{\vec\beta}})^{\mathrm{f}}
    \cong
    (\mathbb L_{M^*\mathfrak E_{k-1}^*}|_{{\tilde{\mathrm{gl}}_k^*}F_{\vec\beta}})^{\mathrm{f}}
    \cong
    (\mathbb L_{{\tilde{\mathrm{gl}}_k^*}M^*\mathfrak
      E_{k-1}^*}|_{{\tilde{\mathrm{gl}}_k^*}F_{\vec\beta}})^{\mathrm{f}}.
  \end{equation*}
  Consider the composition of
  \[
    {\tilde{\mathrm{gl}}_k^*} (M^*\mathfrak E_{k-1}^*)
    \longrightarrow
    \tilde{\mathfrak M}_{g,n+k,d-d_0k}\times^\prime \big(\mathfrak M^{\mathrm{wt,ss}}_{0,1,d_0}\big)^k
    \overset{\mathrm{pr}_1}{\longrightarrow}
    \tilde{\mathfrak M}_{g,n+k,d-d_0k}.
  \]
  We see that
  $\mathbb L_{{\tilde{\mathrm{gl}}_k^*}(M^* \mathfrak E_{k-1}^*)/\tilde{\mathfrak
    M}_{g,n+k,d-d_0k}}|_{{\tilde{\mathrm{gl}}_k}^*F_{\vec\beta}}$ is isomorphic to
  $F[1]$, where $F$ is the dual of the sheaf
  of infinitesimal automorphisms of all the tails $\mathcal E_i$ that fixes
  the node on $\mathcal E_i$ and acts trivially on the relative tangent space to
  $\mathcal E_i$ at the node.
  Hence $F$ is dual to the direct sum of the relative tangent spaces
  along the unique base point on the $\mathcal E_i$'s.
  In particular, it is in the moving part.
  Thus the claim follows from the distinguished triangle of cotangent complexes.

  Thus $[{\tilde{\mathrm{gl}}_k^*}F_{\vec\beta}]^{\mathrm{vir}}$
  is defined by the relative perfect obstruction theory $(\mathbb
  E_1)^{\mathrm{f}}$ over $\tilde{\mathfrak M}_{g,n+k,d-d_0k}$. Moreover,
  \begin{equation}
    \label{eq:euler-class-normal-bundle-curve-part}
    \frac{1}{e_{\mathbb C^*}((\mathbb L^{\vee}_{M \tilde{\mathfrak
          M}_{g,n,d}}|_{{\tilde{\mathrm{gl}}_k^*}F_{\vec\beta}})^{\mathrm{mv}})} =
    \frac
    {\prod_{i=1}^k(\frac{\mathbf r_i }{k}z+\psi(\mathcal E_i))}
    {- \frac{z}{k} -\tilde{\psi}(\mathcal
      E_1) - {\tilde{\psi}}_{n+1} - \sum_{i=k}^\infty[\mathfrak D_i]},
  \end{equation}
  by the above discussion, Lemma~\ref{lem:normal-E_k-1} and Lemma~\ref{lem:str-D}.

  Let $\pi_2: \mathcal C_2 \to Y \underset{(I_{\mu}X)^{k}}{\times}
  \prod_{i=1}^k F_{\star,\beta_i}$ be the universal curve obtained from gluing
  the unique marking of $F_{\star,\beta_i}$ to the $(n+i)$-th marking of $Y$
  (c.f.\ Lemma~\ref{lem:structure-F-beta}).
  Let $u_2: \mathcal C_2\to [W/G]$ be the universal map. Write
  $\mathbb E_2 = (R\pi_{2*}u_2^*\mathbb T_{[W/G]})^{\vee}$.
  Using the discussion in Section~\ref{sec:graph-space-revisited} and a standard
  splitting-node
  argument (c.f.\ Lemma~\ref{lem:int-with-D-k-1}), the virtual cycle
  $[Y]^{\mathrm{vir}} \underset{(I_{\mu}X)^{k}}{\times}
  \prod_{i=1}^k [F_{\star,\beta_i}]^{\mathrm{vir}}$ is defined by the relative
  perfect obstruction theory $(\mathbb E_2)^{\mathrm{f}}$ relative to $\tilde{\mathfrak
    M}_{g,n+k,d-d_0k}$ (c.f.\ Section~\ref{sec:graph-space-revisited}).

  We  want to  compare $\mathbb E_1$ and $\mathbb E_2$.
  We already have the fibered diagram
  \[
    \begin{tikzcd}
      Y^\prime\times_Y {\tilde{\mathrm{gl}}_k^*}F_{\vec\beta}
      \arrow[r,"{\varphi^\prime}"] \arrow[d,"{p_1}"']&
      Y^\prime \underset{(I_{\mu}X)^{k}}{\times}
      \prod_{i=1}^k F_{\star,\beta_i}\arrow[d,"{p_2}"] \\
      {\tilde{\mathrm{gl}}_k^*}F_{\vec\beta} \arrow[r,"{\varphi}"] &
      Y \underset{(I_{\mu}X)^{k}}{\times}
      \prod_{i=1}^k F_{\star,\beta_i}
    \end{tikzcd}.
  \]
  Since $\varphi^\prime$ is defined via the universal property of
  $F_{\star,\beta_i}$, it is equipped with an isomorphism of the universal curves
  \begin{equation}
    \label{map:isom-universal-curve-F-beta-graph-space}
    \tilde\varphi : p_1^*\mathcal C_1 \longrightarrow p_2^* \mathcal
    C_2
  \end{equation}
  that commutes with the maps to $[W/G]$. Hence it induces an isomorphism
  \begin{equation}
    \label{map:isom-rel-pot}
    \alpha: p_1^* \mathbb E_1 \cong p_1^* \varphi^*\mathbb E_2
  \end{equation}

  We now consider group actions. Let $\mathbb C^*\times (\mathbb C^*)^k$ act on
  $p_1^*\mathcal C_1$: the first factor acts on $\mathcal C_1$ by \eqref{eq:C*-action}
  and acts trivially on $Y^\prime$; the second factor acts trivially on
  $\mathcal C_1$ and acts on $Y^\prime$ by the standard scaling of the fibers of
  $Y^\prime \to Y$. Let $(\mathbb C^*)^k \times (\mathbb C^*)^k$ act on $p_2^*{\mathcal C_2}$:
  the first factor acts on $\mathcal C_2$ by the product of the $\mathbb
  C^*$-actions \eqref{eq:C*-action-on-P1} on the graph spaces
  $QG_{0,1}(X,\beta_i)$ and acts trivially on $Y^\prime$; the action of the
  second factor is defined in the same way as the second factor of $\mathbb
  C^*\times(\mathbb C^*)^k$. These also define actions on $Y^\prime\times_Y
  {\tilde{\mathrm{gl}}_k^*}F_{\vec\beta}$ and $Y^\prime \underset{(I_{\mu}X)^{k}}{\times}
      \prod_{i=1}^k F_{\star,\beta_i}$, where the first factors act trivially. 
  Thus $p_1^*\mathbb E_1$ is $\mathbb C^*\times (\mathbb C^*)^k$-equivariant and
  $p_2^*\mathbb E_2$ is  $(\mathbb C^*)^k\times (\mathbb C^*)^k$-equivariant.

  Applying Lemma~\ref{lem:equivariant-action-P1-bundle} to all the entangled
  tails, for any $(\lambda,t) = (\lambda, t_1 ,\ldots, t_k) \in
  \mathbb C^*\times (\mathbb C^*)^k$,
  we have a natural commutative diagram
  \[
    \begin{tikzcd}
      \label{cd:action-on-C1-C2}
      p_1^*\mathcal C_1 \arrow[r,"{\tilde\varphi}"] \arrow[d,"{(\lambda^k,t)}"']&
      p_2^*\mathcal C_2 \arrow[d,"{(\lambda^{\mathbf r_1} t_1^{-1} ,\ldots,
        \lambda^{\mathbf r_k} t_k^{-1},t)}"] \\
      p_2^*\mathcal C_1 \arrow[r,"{\tilde\varphi}"] & p_2^*\mathcal C_2
    \end{tikzcd}.
  \]
    Now we view all
    $(\mathbb C^*)^k\times (\mathbb C^*)^k$-actions also as
    $\mathbb C^*\times (\mathbb C^*)^k$-actions via
    the group homomorphism\footnote{The rational exponents make sense after we
      raise the group action to a certain power.}
    \begin{gather*}
      \mathbb C^*\times (\mathbb C^*)^k \longrightarrow  (\mathbb C^*)^k \times
      (\mathbb C^*)^k,\\
      (\lambda,t) \mapsto (\lambda^{\mathbf r_1/k}t_1^{-1} ,\ldots,
      \lambda^{\mathbf r_k/k}t_k^{-1}, t).
    \end{gather*}
  Then $\tilde\varphi$ is $\mathbb C^*\times (\mathbb C^*)^k$-equivariant. Hence
  so is $\alpha$. Note that $\mathbb C^*$ acts trivially on $Y^\prime\times_{Y}
  \tilde{\mathrm{gl}}_k^* F_{\vec\beta}$. We take the $\mathbb C^*$-invariant
  part of \eqref{map:isom-rel-pot} and obtain an isomorphism
  \[
    \alpha^{\mathbb C^*}: (p_1^* \mathbb E_1)^{\mathbb C^*} \cong (p_1^*
    \varphi^*\mathbb E_2)^{\mathbb C^*}
  \]
  of $(\mathbb C^*)^k$-equivariant objects. Note that $(p_1^*\mathbb
  E_1)^{\mathbb C^*} = p_1^*\mathbb E_1^{\mathrm{f}}$ and $(p_1^*
  \varphi^*\mathbb E_2)^{\mathbb C^*} = p_1^*
  \varphi^*\mathbb E^{\mathrm{f}}_2$. Using the equivalence between the $(\mathbb
  C^*)^k$-equivariant derived category of $Y^\prime\times_Y
  {\tilde{\mathrm{gl}}_k^*}F_{\vec\beta}$ and the derived category of
  ${\tilde{\mathrm{gl}}_k^*}F_{\vec\beta}$, we obtain an isomorphism between
  the fixed parts
  \[
    \mathbb E_1^{\mathrm{f}} \cong
    \varphi^*\mathbb E_2^{\mathrm{f}}.
  \]
  The construction is natural and the isomorphism commutes with the maps to
  the relative cotangent complexes. This proves that they induce the same
  virtual fundamental class.

  We now come to the moving part of $\mathbb E_1$ and $\mathbb E_2$
  respectively.
  We first work in the equivariant derived category.
  Let
  \[
    (z_1 ,\ldots, z_k, w_1 ,\ldots, w_k)
  \]
  be the equivariant parameters of $(\mathbb C^*)^k\times (\mathbb
  C^*)^k$, and let $z$ be the equivariant parameter of $\mathbb C^*$ as before.
  We obtain
  \[
    \frac{1}{e_{(\mathbb C^*)^k\times (\mathbb C^*)^k}(p_2^*(\mathbb
      E_2^{\vee,\mathrm{mv}}))}  = \mathcal I_{\beta_1}(z_1)\boxtimes
    \cdots \boxtimes \mathcal I_{\beta_k}(z_k).
  \]
  Using \eqref{map:isom-rel-pot}, we have
  \[
   \frac{1}{e_{\mathbb C^*\times (\mathbb C^*)^k}(p_1^*\mathbb
     E_1^{\vee,{\mathrm{mv}}})}  = \mathcal I_{\beta_1}(\frac{\mathbf r_1}{k}z-w_1)\boxtimes
   \cdots \boxtimes \mathcal I_{\beta_k}(\frac{\mathbf r_k}{k}z-w_k).
  \]
  We have suppressed some obvious pullback notation.

  Using the canonical isomorphism between the $\mathbb C^*$-equivariant intersection
  theory of ${\tilde{\mathrm{gl}}_k^*}F_{\vec\beta}$ and the $\mathbb C^*\times (\mathbb
  C^*)^k$-equivariant intersection theory of $Y^\prime\times_Y
  {\tilde{\mathrm{gl}}_k^*}F_{\vec\beta}$, we obtain
  \[
   \frac{1}{e_{\mathbb C^*}(\mathbb
     E_1^{\vee,{\mathrm{mv}}})}  = \mathcal I_{\beta_1}(\frac{\mathbf
     r_1}{k}z+\psi(\mathcal E_1))\boxtimes
   \cdots \boxtimes \mathcal I_{\beta_k}(\frac{\mathbf r_k}{k}z+\psi(\mathcal E_k)).
  \]
  Combining this with \eqref{eq:euler-class-normal-bundle-curve-part}, we obtain
  the desired formula \eqref{eq:normal-bundle-F-beta}.
\ifdefined\SHOWDETAIL
  {\color{blue}
    \newline
    Here is the details about the ``canonical isomorphism'' above.
    Let $X$ be any scheme and $L$ be a line bundle. (This $X$ will be
    $[{\tilde{\mathrm{gl}}_k^*}F_{\vec\beta}/\mathbb C^*]$. It can be reduced to
    the scheme case by the definition of equivariant intersection theory.)
    Let
    \[
      \pi: L^* \longrightarrow  X
    \]
    be the total space
    of $L$ minus the zero section. Let $\mathbb C^*$ act on $L^*$ by the
    standard action on fibers. Let $C^*$ act on $X$ trivially. Then we have
    \[
      A^*_{\mathbb C^*}(X) \cong A^*(X) \otimes_{\mathbb Q} \mathbb Q[z].
    \]
    The pullback morphism
    \[
      p^*:  A^*_{\mathbb C^*}(X) \longrightarrow  A^*_{\mathbb C^*}(L^*)
    \]
    is surjective with kernel generated by
    \[
      z - c_1(L).
    \]
    To see the sign is correct, let $\mathbb C_{\mathrm{std}}$ be the trival bundle with the
    standard action.
    Then $p^*(L) \otimes \mathbb C_{\mathrm{std}}$ has a invariant section.
    Hence
    \[
      c_1(p^*(L) \otimes \mathbb C_{\mathrm{std}}) = p^*(c_1(L)) - z =0
    \]
    in $A_{\mathbb C^*}^*(L^*)$.

    The composition
    \[
      A^*_{\mathbb C^*}(X) \overset{\alpha\mapsto \alpha \otimes
        1}{\longrightarrow} A^*_{\mathbb C^*}(X)
      \overset{\pi^*}{\longrightarrow}
      A^*_{\mathbb C^*}(L^*)
    \]
    is an isomorphism.

    Let $E$ be a vector bundle on $X$. Give it the trivial $\mathbb C^*$-action.
    Then $\pi^*(E)$ is a equivariant bundle with ``trivial'' action.
    Let
    \[
      \alpha_i\in A_*(X).
    \]
    If
    \[
      e_{\mathbb C^*}(\pi^* E) = \sum_{i}\pi^*\alpha_i\cdot z^*,
    \]
    then
    \[
      e(E) = \sum_{i} \alpha_i\cdot (c_1(L))^{i}.
    \]
    \newline
  }
  \fi
\end{proof}
\fi

\section{The wall-crossing formula}
\label{sec:the-formula}
\subsection{General discussion}
We will use the so-called master space technique to obtain the desired
wall-crossing formulas.
First suppose that $2g-2+n+\epsilon_0d >0$.
By the virtual localization formula \cite{graber1999localization, chang2017torus}, we have
\begin{equation}
  \label{eq:equation-on-master}
  [MQ^{\epsilon_0}_{g,n}(X,\beta)]^{\mathrm{vir}} = \sum_{\star}
  (\iota_{F_{\star}})_*\Big(\frac{[F_{\star}]^{\mathrm{vir}}}{e_{\mathbb C^*}(N^{\mathrm{vir}}_{F_{\star}/MQ^{\epsilon_0}_{g,n}(X,\beta)})}
  \Big)\in A^{\mathbb C^*}_*(MQ^{\epsilon_0}_{g,n}(X,\beta)) \otimes_{\mathbb Q[z]} \mathbb Q(z),
\end{equation}
where the sum is over all the fixed-point components $F_+,F_-$ and $F_{\vec\beta}$
in the previous section, and $\iota_{F_\star}$ is the inclusion
of ${F_\star}$ into $MQ^{\epsilon_0}_{g,n}(X,\beta)$.

Define the morphism
\[
  \tau: MQ^{\epsilon_0}_{g,n}(X,\beta) \longrightarrow Q_{g,n}^{\epsilon_-}(\mathbb P^N,d)
\]
by
\begin{itemize}
\item
  composing the quasimaps with \eqref{eq:embedding-of-quotient},
\item
  taking the coarse moduli of the domain curves,
\item
  taking the $\epsilon_-$-stabilization of the obtained quasimaps to $\mathbb P^{N}$.
\end{itemize}
Consider the trivial $\mathbb C^*$-action on $Q_{g,n}^{\epsilon_-}(\mathbb
P^N,d)$. Then $\tau$ is $\mathbb C^*$-equivariant. We pushforward
relation~\eqref{eq:equation-on-master} along $\tau$ and obtain
\begin{equation}
  \label{eq:equation-on-Q-Pn}
  \sum_{\star}
  \tau_{*}(\iota_{F_{\star}})_*
  \Big(\frac{[F_{\star}]^{\mathrm{vir}}}
  {e_{\mathbb C^*}(N^{\mathrm{vir}}_{F_{\star}/MQ^{\epsilon_0}_{g,n}(X,\beta)})}
  \Big)
  =  \tau_{*}[MQ^{\epsilon_0}_{g,n}(X,\beta)]^{\mathrm{vir}}.
\end{equation}
Since the right hand side lies in
\[
A_*^{\mathbb C^*}(Q_{g,n}^{\epsilon_-}(\mathbb P^N, d))
\cong
A_*(Q_{g,n}^{\epsilon_-}(\mathbb P^N,d))\otimes_{\mathbb Q} \mathbb Q[z],
\]
so does the left hand side. In particular, the residue at $z=0$ of the left hand side is
zero.
This is the master space technique that we will use.

To simplify the notation, we will suppress the pushforward notation whenever
possible.
When $M$ is a moduli space with a morphism $\tau: M \to
Q_{g,n}^{\epsilon_-}(\mathbb P^N,d)$ that is obvious from the context,
we write
\[
  \int_{\beta} \alpha : = \tau_*(\alpha \cap \beta), \quad \text{ for }\alpha\in A^*(M),
  \beta\in A_*(M).
\]
Note that we have ${Q}^{\epsilon_{-}}_{g,n+k}(\mathbb P^N,\deg(\beta^\prime))\to
Q^{\epsilon_{-}}_{g,n}(\mathbb P^N,d)$  that replaces the last $k$ markings by
length-$d_0$ base points (c.f.\ \eqref{map:marking-to-base-points}).

The case $g=0,n=1, d = d_0$ is similar, the only difference being that we replace
$Q^{\epsilon_-}_{g,n}(\mathbb P^N, d)$ by $I_{\mu}X$ and pushforward
the relations along $\check{\mathrm{ev}}_{\star}$ (c.f.\ \eqref{map:star-check}).
\subsection{The $g=0,n=1$ and $d=d_0$ case}
Recall the locally constant function $\mathbf r$ from
Section~\ref{sec:state-space-and-invariants} and set $\mathbf r_1 =
{\mathrm{ev}}_1^*(\mathbf r)$.
\begin{lemma}[Theorem \ref{thm:genus-0-special-case}]
  \label{lem:genus-zero-special-case}
  \[
     \int_{[Q^{\epsilon_+}_{0,1}(X,\beta)]^{\mathrm{vir}}}
     \mathbf r_1^2 {\psi_1}^{\ell} =
    \mathrm{Res}_{z=0} \big(z^{\ell+1} I_{\beta}(z)\big),
    \quad \ell = 0,1,2, \cdots.
  \]
\end{lemma}
\ifdefined\SHOWPROOFS
\begin{proof}
  Let $\tilde\psi_1$ be the equivalent orbifold $\psi$-class on
  $MQ_{0,1}^{\epsilon_0}(X,\beta)$.
  By Lemma~\ref{lem:correction-term-special-degree},
  Lemma~\ref{lem:correction-term-special-contr}, \eqref{eq:I-from-curly-I},
  and the localization formula, we have
  \begin{align*}
    \int_{
    [MQ^{\epsilon_0}_{0,1}(X,\beta)]^{\mathrm{vir}}
    }\tilde{\psi}_1^{\ell} =&
      \int_{[Q^{\epsilon_+}_{0,1}(X,\beta)]^{\mathrm{vir}}}
        \frac{\tilde{\psi}_1^\ell|_{Q^{\epsilon_+}_{0,1}(X,\beta)}}{-z + \alpha} +
             \int_{[F_{\star,\beta}]^{\mathrm{vir}}} \mathbf r_1^2 z^{\ell+1}\cdot
        (\mathcal I_{\beta}(\mathbf r_1z))\\
    = & \int_{[Q^{\epsilon_+}_{0,1}(X,\beta)]^\mathrm{vir}}
        \frac{(\psi_1/\mathbf r_1)^\ell}{-z + \alpha} +   z^{\ell+1}
         I_{\beta}(\mathbf r_1z).
  \end{align*}
  \ifdefined\SHOWDETAIL
  {\color{blue}
    \newline
\begin{align*}
    \int_{
    [MQ^{\epsilon_0}_{0,1}(X,\beta)]^{\mathrm{vir}}
    }\tilde{\psi}^{\ell} =&
                    \int_{
                    [F_+]^{\mathrm{vir}}
                    } \frac{\tilde{\psi}^\ell|_{F_{+}}}{-z + \alpha} +
                    \int_{[F_{\beta}]^{\mathrm{vir}}} z^{\ell}\cdot (\mathbf r z)\cdot
                     (\mathcal I_{\beta}(\mathbf rz))
    \\
    = & \int_{[Q^{\epsilon_+}_{0,1}(X,\beta)]^{\mathrm{vir}}}
        \frac{\tilde\psi^\ell}{-z + \alpha} + \int_{[F_{\star,\beta}]^{\mathrm{vir}}}  \mathbf r^2 z^{\ell+1}\cdot
        (\mathcal I_{\beta}(\mathbf rz))\\
    = & \int_{[Q^{\epsilon_+}_{0,1}(X,\beta)]^{\mathrm{vir}}}
        \frac{(\psi/\mathbf r)^\ell}{-z + \alpha} +   z^{\ell+1} I_{\beta}(\mathbf rz).
  \end{align*}
    \newline
  }
  \fi
  Here $\alpha$ is the (non-equivariant) first Chern class of the calibration bundle on
  $Q^{\epsilon_+}_{0,1}(X,\beta)$.
Taking the residues of both sides, we obtain
\[
  \int_{[Q^{\epsilon_+}_{0,1}(X,\beta)]^{\mathrm{vir}}}(\psi_1/\mathbf r_1)^\ell =
  \mathrm{Res}_{z=0}\big(z^{\ell+1} I_{\beta}(\mathbf r_1z)\big).
\]
Applying the change of variable $z\mapsto z/\mathbf r_1$, we obtain the desired formula.
\end{proof}
\fi
\subsection{The main case}
We now study the case $2g-2+n+\epsilon_0d>0$.
Let $\mathbf r_i$ be the locally constant function indicating the order the
automorphism group at the $(n+i)$-th marking as in Section~\ref{sec:correction-term-main-case}.
By Lemma~\ref{lem:obstruction-theory-F-beta},
the contribution from $F_{\vec\beta}$ to the residue of the left hand side  of
\eqref{eq:equation-on-Q-Pn} is
\begin{equation}
  \label{eq:residue-gl-F-beta}
  \begin{aligned}
    \int_{[\tilde{\mathrm{gl}}^*_k F_{\vec\beta}]^{\mathrm{vir}}}
    &
    \frac{\prod_{i=1}^k \mathbf r_i}{k!}
     \mathrm{Res}_{z=0} \Big(   \frac {\prod_{i=1}^k(\frac{\mathbf
        r_i}{k}z+\psi(\mathcal E_i))} {- \frac{z}{k} -\tilde\psi(\mathcal
      E_1) - \tilde\psi_{n+1} - \sum_{i=k}^\infty[\mathfrak D_i]} \cdot \\
     & \mathcal
    I_{\beta_1}(\frac{\mathbf r_1}{k}z+ \psi(\mathcal E_1))\boxtimes
    \cdots \boxtimes \mathcal I_{\beta_k}(\frac{\mathbf r_k}{k}z+
    \psi(\mathcal E_k)) \Big).
  \end{aligned}
\end{equation}
See the paragraph before Lemma~\ref{lem:obstruction-theory-F-beta} for the notation.
Note that the factor $\frac{\prod_{i=1}^k \mathbf r_i}{k!}$ comes from the
degree of ${\tilde{\mathrm{gl}}_k}$ in \eqref{eq:degree-of-q}.

Recall that we have morphisms (see \eqref{map:Y-to-Q-n+k} and
Lemma~\ref{lem:structure-F-beta})
\[
  \textstyle
  \tilde{\mathrm{gl}}^*_k F_{\vec\beta}
  \overset{\varphi}{\longrightarrow}
  Y  \times_{(I_{\mu}X)^{k}} \prod_{i=1}^k F_{\star,\beta_i}
  \overset{\mathrm{pr}_Y}{\longrightarrow}
  Y
  \longrightarrow
  \tilde{Q}^{\epsilon_+}_{g,n+k}(X,\beta^\prime)
  \longrightarrow
  \tilde{\mathfrak M}_{g,n+k,d-kd_0}.
\]
We want to apply the projection formula to pushforward
\eqref{eq:residue-gl-F-beta}
to $\tilde{Q}^{\epsilon_+}_{g,n+k}(X,\beta^\prime)$.
Apply the change of variable, which does not change the residue (c.f.\
Remark~\ref{rmk:entanglement-implies-equation-of-psi})
\[
  z \mapsto k(z-\tilde\psi(\mathcal E_1)-\tilde\psi_{n+1}) = \cdots =
  k(z-\tilde\psi(\mathcal E_k)-\tilde\psi_{n+k}),
\]
and use the relation $\mathbf r_{i} \tilde{\psi}_{n+i} = \psi_{n+i}$ (c.f.\
Section~\ref{sec:orbifold-psi-classes}).
Thus \eqref{eq:residue-gl-F-beta} becomes
\begin{equation*}
  \begin{aligned}
    \int_{[\tilde{\mathrm{gl}}^*_k F_{\vec\beta}]^{\mathrm{vir}}}
    \frac{\prod_{i=1}^k \mathbf r_i}{(k-1)!}
     \mathrm{Res}_{z=0} \Big(   &\frac {\prod_{i=1}^k({\mathbf
        r_i}z-\psi_{n+i})} {- {z}  - \sum_{i=k}^\infty[\mathfrak D_i]} \cdot \\
     & \mathcal
    I_{\beta_1}({\mathbf r_1}z - \psi_{n+1})\boxtimes
    \cdots \boxtimes \mathcal I_{\beta_k}({\mathbf r_k}z-
    \psi_{n+k}) \Big).
  \end{aligned}
\end{equation*}
By (2) of Lemma~\ref{lem:divisors1}, the pullback (as
a divisor class) of $\mathfrak D_{i}$ to
$\tilde{\mathrm{gl}}^*_k F_{\vec\beta}$ is equal to the pullback
of $\mathfrak D^\prime_{i-k}$, where the $\mathfrak D^\prime_{i-k}$'s are the
boundary divisors of $\tilde{\mathfrak M}_{g,n+k,d-kd_0}$.

Now apply the projection formula.
Note that $Y\to \tilde{Q}^{\epsilon_+}_{g,n+k}(X,\beta^\prime)$ has degree $1/k$.
Using Lemma~\ref{lem:structure-F-beta},
Lemma~\ref{lem:obstruction-theory-F-beta}, and \eqref{eq:I-from-curly-I},
we see that \eqref{eq:residue-gl-F-beta} becomes
\begin{equation}
  \label{eq:contr-F-beta-primitive}
  \begin{aligned}
  \int_{[\tilde{Q}^{\epsilon_+}_{g,n+k}(X,\beta^\prime)]^{\mathrm{vir}}
  }
  \frac{1}{k!} \mathrm{Res}_{z=0} \big(
    \frac
    {\prod_{i=1}^k \mathrm{ev}_{n+i}^*\big(
      ({\mathbf r}z - \psi)\
      I_{\beta_i}({\mathbf r}z - \psi)
      \big)
    }
    { -{z}  - \sum_{i=0}^\infty[\mathfrak D^\prime_i]}
  \big),
  \end{aligned}
\end{equation}
where we have extended the definition of $\mathrm{ev}^*_{n+i}$ by linearity in
powers of $z$ and $\psi$.
We have also suppressed the subscripts of $\mathbf r_i$ and $\psi_{n+i}$ inside
the $\mathrm{ev}^*_{n+i}(\cdots)$. This convention will be adopted from now on.

We expand
\begin{align}
  \label{eq:expansion}
  \textstyle
  &\frac{1}{-z-\sum_{i=0}^\infty [\mathfrak D^\prime_{i}] }
  =
  \frac{-1}{z} + \sum_{s\geq 1}
  \sum_{r=1}^{\infty} (-z)^{-s-1} [\mathfrak
  D^\prime_{r-1}]( \sum_{i=0}^\infty [\mathfrak D^{\prime}_{i}]
  )^{s-1}.
\end{align}
Recall that for $f(z) = \sum_ia_iz^i$, $[f]_{i} :=a_i$.
\begin{lemma}
  \label{lem:intergral-D-r-1-term}
  For $s\geq 1$, $r = 1,2,\ldots$, we have
  \begin{equation}
    \label{eq:integral-D-r-1-term}
    \begin{aligned}
      &\int_{[\tilde Q^{\epsilon_+}_{g,n+k}(X,\beta^\prime)]^{\mathrm{vir}}}
       [\mathfrak D^\prime_{r-1}] (\sum_{i=0}^\infty [\mathfrak D^\prime_i])^{s-1} \\
      =
      &
      \sum_{\vec\beta^\prime}\sum_{J}
      \int_{[Q^{\epsilon_+}_{g,n+k+r}(X,\beta^{\prime\prime})]^{\mathrm{vir}}}
      \frac{(-1)^{s-r}}{r!}
              \prod_{a=1}^{r} \left[
        \mathrm{ev}^*_{n+k+a}\big((\mathbf r z-\psi)I_{\beta_a^\prime}(\mathbf r
        z- \psi)\big) \right]_{-j_a-1},
    \end{aligned}
  \end{equation}
  where $\vec\beta^\prime = (\beta^{\prime\prime},\beta_1^\prime ,\ldots,
  \beta_r^\prime)$ runs through all
  $(r+1)$-tuples of effective curves classes such that
     $\beta^\prime = \beta^{\prime\prime} + \sum_{a=1}^{r}\beta_a^\prime$, $\deg(\beta_a^\prime)
     = d_0$ for $a=1 ,\ldots, r$;
     $J = (j_1 ,\ldots, j_r)$ runs through all $r$-tuples of non-negative integers such
    that $j_1 + \cdots +j_r = s-r$.
\end{lemma}
\ifdefined\SHOWPROOFS
\begin{proof}
  By Lemma~\ref{lem:int-with-D-k-1}, we have
  \begin{equation*}
  \begin{aligned}
    & \int_{[\tilde Q^{\epsilon_+}_{g,n+k}(X,\beta^\prime)]^{\mathrm{vir}}}
    [\mathfrak D^\prime_{r-1}] (\sum_{i=0}^\infty [\mathfrak D^\prime_i])^{s-1}
      \\
    = & \sum_{\vec\beta^\prime}\frac{\prod_{a=1}^r \mathbf r_{n+k+a}}{r!} \int_{
        [\tilde{Q}^{\epsilon_+}_{g,n+k+r}(X,\beta^{\prime\prime})\underset{(I_\mu X)^r}{\times}
        \prod_{a=1}^r {Q}^{\epsilon_+}_{0,1}(X,\beta^\prime_a)]^{\mathrm{vir}}
        }
        p_*\big((\sum_{i=0}^\infty [\mathfrak D^\prime_i])^{s-1}
        \big),
  \end{aligned}
\end{equation*}
  where
  \[
    \textstyle
    p: {\tilde{\mathrm{gl}}_k^*}\mathfrak D_{r-1}
    \!\!\!\!\!\!\!
    \underset{\tilde{\mathfrak M}_{g,n+k,d-kd_0}}{\times}
    \!\!\!\!\!\!\!
    \tilde Q^{\epsilon_+}_{g,n+k}(X,\beta^\prime) \longrightarrow
    \coprod_{\vec\beta^\prime}\Big(
    \tilde{Q}^{\epsilon_+}_{g,n+k+r}(X,\beta^{\prime\prime})\underset{(I_\mu X)^r}{\times} \prod_{a=1}^r
    {Q}^{\epsilon_+}_{0,1}(X,\beta^\prime_a)\Big)
  \]
  is the inflated projective bundle as in Lemma~\ref{lem:int-with-D-k-1}
  (c.f.\ Definition~\ref{def:inflated-projective-bundle},
  Lemma~\ref{lem:str-D}), and $\vec\beta^\prime$ is as above.
  By Lemma~\ref{lem:divisors2}, the restriction of $\sum_{i=0}^\infty [\mathfrak D^\prime_i]$
   to \[{\tilde{\mathrm{gl}}_k^*}\mathfrak
   D^\prime_{r-1}
    \underset{\tilde{\mathfrak M}_{g,n+k,d-kd_0}}{\times}
  \tilde Q^{\epsilon_+}_{g,n+k}(X,\beta^{\prime}) \]
  is equal to
  \[
   \frac{1}{r} \big(  (r-1)[D_0] + (r-2)[D_1]+ \cdots + [D_{r-2}]
    + c_1(\Theta_{n+k+1}) + \cdots + c_1(\Theta_{n+k+r})\big),
  \]
  where $D_0 ,\ldots, D_{r-2}$ are the tautological divisors of the inflated
  projective bundle (c.f.\ Section~\ref{sec:inflated-proj-bundle}), and
  $\Theta_{n+k+a}$ is the tensor product of two orbifold relative tangent
  bundles, one at the $(n+k+a)$-th marking of
  $\tilde{Q}^{\epsilon_+}_{g,n+k+r}(X,\beta^{\prime\prime})$ and the other at the unique
  marking of ${Q}^{\epsilon_+}_{0,1}(X,\beta^\prime_a)$.
  Then by Lemma~\ref{lem_inflated_projective_bundle} and the projection formula, the
  integral \eqref{eq:integral-D-r-1-term} is equal to
  \begin{align*}
       \sum_{\vec\beta^\prime}\sum_J\frac{\prod_{a=1}^{r}\mathbf r_{k+a}}{r!} & \int_{\gamma}
    (-1)^{s-r}\cdot \prod_{a=1}^r (\tilde\psi_{n+k+a}+ \tilde\psi^\prime_{n+k+a})^{j_a},
  \end{align*}
  where
  $
  \gamma =
  [\tilde{Q}^{\epsilon_+}_{g,n+k+r}(X,\beta^{\prime\prime})]^{\mathrm{vir}}\underset{(I_\mu X)^k}{\times}
  \prod_{a=1}^r [{Q}^{\epsilon_+}_{0,1}(X,\beta^\prime_a)]^{\mathrm{vir}}
  $, $\tilde\psi_{n+k+a}$ is the orbifold $\psi$-class of
  $\tilde{Q}^{\epsilon_+}_{g,n+k+r}(X,\beta^{\prime\prime})$ at the $(n+k+a)$-th
  marking, and $\tilde\psi^\prime_{n+k+a}$ is the orbifold $\psi$-class of
  $Q_{0,1}^{\epsilon_+}(X,\beta_a^\prime)$.

  We integrate the powers of $\tilde\psi^\prime_{n+k+a}$ against
  $[{Q}^{\epsilon_+}_{0,1}(X,\beta^\prime_i)]^{\mathrm{vir}}$ using
  Lemma~\ref{lem:genus-zero-special-case}, and \eqref{eq:integral-D-r-1-term} becomes
  \begin{align*}
     &  \sum_{\vec\beta^\prime}\sum_J \int_{
        [\tilde{Q}^{\epsilon_+}_{g,n+k+r}(X,\beta^{\prime\prime})]^{\mathrm{vir}}}\frac{(-1)^{s-r}}{r!}
       \prod_{a=1}^r \sum_{b=0}^{j_a}{ \binom{j_a}{b}} {\tilde\psi}_{n+k+a}^b
       \mathrm{ev}^*_{n+k+a}[\mathbf r z I_{\beta_a^\prime}(\mathbf rz)
       ]_{b-j_a-1}.
  \end{align*}
  This is equal to
  \[
      \sum_{\vec\beta^\prime}\sum_J \int_{
        [\tilde{Q}^{\epsilon_+}_{g,n+k+r}(X,\beta^{\prime\prime})]^{\mathrm{vir}}}\frac{(-1)^{s-r}}{r!}
      \prod_{a=1}^r
      \mathrm{ev}^*_{n+k+a}\mathrm{Res}_{z=0}
      \big(
      \mathbf r z
      (z+\tilde\psi)^{j_a} I_{\beta_a^\prime}(\mathbf rz)
      \big).
  \]
  Apply the change of variable $z\mapsto z-\tilde\psi$, which does not change
  the residue. Observe that $\mathbf
  r\tilde\psi = \psi$ and recall Lemma~\ref{lem:vir-cycle-comp-entangled-tails}.
  Then we obtain the desired formula.
\end{proof}
\fi

Using this Lemma and the expansion \eqref{eq:expansion}, we
rewrite \eqref{eq:contr-F-beta-primitive} as follows.
\begin{corollary}
  \label{lem:contr-F-beta-form2}
  The contribution to the left hand side of \eqref{eq:equation-on-Q-Pn} from
  $F_{\vec\beta}$ is
  \begin{align*}
    - \frac{1}{k!} \sum_{r=0}^\infty\sum_{\vec\beta^\prime}\sum_{\vec b} \frac{(-1)^r}{r!}
      \int_{[ Q^{\epsilon_+}_{g,n+k+r}(X,\beta^{\prime\prime})]^{\mathrm{vir}}} &\big[
        \prod_{i=1}^k \mathrm{ev}_{n+i}^*((\mathbf rz -\psi)I_{\beta_i}(\mathbf r z-\psi))
        \big]_{b_0} \cdot \\
      & \prod_{a=1}^{r} \mathrm{ev}^*_{n+k+a}
        \big[(\mathbf rz-\psi)
    I_{\beta_a^\prime}(\mathbf r z-\psi) \big]_{b_a},
  \end{align*}
  where
    $\vec\beta^\prime = (\beta^{\prime\prime},\beta^\prime_1,\ldots, \beta_r^\prime)$
  runs through all decompositions of $\beta^\prime$ as in
  Lemma~\ref{lem:intergral-D-r-1-term}, and
    $\vec b= (b_0 ,\ldots, b_r)$
  runs through all $(r+1)$-tuples of integers such that $b_0+ \cdots + b_r = 0$ and
  $b_1 ,\ldots, b_r <0$.
\end{corollary}

\ifdefined\SHOWPROOFS
\begin{proof}
  This is a straightforward computation and we leave it to the reader.
  \ifdefined\SHOWDETAIL
  {\color{blue}
    \begin{align*}
    & \int_{[\tilde Q^{\epsilon_+}_{g,n+k}(X,\beta^\prime)]^{\mathrm{vir}}}
    \frac{1}{k!} \mathrm{Res}_{z=0} \big( \frac {\prod_{i=1}^k
      \mathrm{ev}_{n+i}^*((\mathbf r_iz-\psi) I_{\beta_{i}}
      (\mathbf r_iz- \psi))} {- z - \sum_{i=0}^\infty[\mathfrak
      D^\prime_i]} \big)\\
    =
    & \int_{[\tilde Q^{\epsilon_+}_{g,n+k}(X,\beta^\prime)]^{\mathrm{vir}}}
    \frac{1}{k!} \mathrm{Res}_{z=0} \Big(  \prod_{i=1}^k
      \mathrm{ev}_{n+i}^*((\mathbf r_iz- \psi)\ I_{\beta_{i}}
      (\mathbf r_iz- \psi))  \\
      &\cdot  \big(
  \frac{-1}{z} + \sum_{s\geq 1}
  \sum_{r=1}^{\infty} (-z)^{-s-1} [\mathfrak
  D^\prime_{r-1}]( \sum_{i=0}^\infty [\mathfrak D^{\prime}_{i}]
  )^{s-1}
  \big)\Big) \\
  =& - \frac{1}{k!}
     \int_{[Q^{\epsilon_+}_{g,n+k}(X,\beta^\prime)]^{\mathrm{vir}}}
     \Big[
     \prod_{i=1}^k
     \mathrm{ev}_{n+i}^*(\mathbf r_iz-\psi)\ I_{\beta_{i}}
     (\mathbf r_iz- \psi)
     \Big]_{0} \\
      & +
      \frac{1}{k!}
      \sum_{s\geq 1}
      \sum_{\vec\beta^\prime}\sum_{J}
      \int_{[Q^{\epsilon_+}_{g,n+k+r}(X,\beta^{\prime\prime})]^{\mathrm{vir}}}
      \frac{(-1)^{r-1}}{r!}
      \Big[
        \prod_{i=1}^k \mathrm{ev}_{n+i}^*(
        (\mathbf rz- \psi)
        I_{\beta_i}(\mathbf rz- \psi)
        )
        \Big]_{s} \cdot
        \\ &\phantom{aaaaaaaaaaaaaaaaaaaaaaa}\prod_{a=1}^{r} \left[
          \mathrm{ev}^*_{n+k+a}(\mathbf r z- \psi)I_{\beta_a^\prime}(\mathbf r z-\psi) \right]_{-j_a-1}\\
        = &
  - \frac{1}{k!}
  \int_{[Q^{\epsilon_+}_{g,n+k}(X,\beta^\prime)]^{\mathrm{vir}}}
  \Big[
  \prod_{i=1}^k
            \mathrm{ev}_{n+i}^*
            ((\mathbf r_iz-\psi)\ I_{\beta_{i}} (\mathbf r_iz-\psi) )
      \Big]_{0} \\
      & -
      \frac{1}{k!}
      \sum_{s\geq 1}
      \sum_{\vec\beta^\prime}\sum_{J}
      \int_{[Q^{\epsilon_+}_{g,n+k+r}(X,\beta^{\prime\prime})]^{\mathrm{vir}}}
      \frac{(-1)^{r}}{r!}
      \Big[
        \prod_{i=1}^k \mathrm{ev}_{n+i}^*(
        (\mathbf rz- \psi)I_{\beta_i}(\mathbf rz- \psi)
        )
        \Big]_{s} \cdot
        \\ &\phantom{aaaaaaaaaaaaaaaaaaaaaaa}\prod_{a=1}^{r} \left[
          \mathrm{ev}^*_{n+k+a}(\mathbf r z- \psi)I_{\beta_a^\prime}(\mathbf r
             z- \psi) \right]_{-j_a-1}
  \end{align*}
   This is the desired formula. }
  \fi
\end{proof}
\fi

We are finally ready to prove the main theorem.
Recall that for any $\beta$,
  $\mu_{\beta}(z) \in A^*(I_{\mu}X)$
is the coefficient of $q^\beta$ in $[zI(q,z)-z]_{z^{\geq 0}}$, where
$[\cdot]_{z^{\geq 0}}$ means the truncation obtained by taking only nonnegative powers of
$z$.
Note that for $\beta\neq 0$,
\begin{equation}
  \label{eq:mu-zI-0-term}
  \mu_{\beta}(z)|_{z=-\psi} = [(z-\psi)I_{\beta}(z-\psi)]_0.
\end{equation}
\begin{theorem}[Theorem~\ref{thm:Chow-version}]
  \label{thm:main-case}
  \begin{equation*}
    \begin{aligned}
       \int_{[Q^{\epsilon_-}_{g,n}(X,\beta)]^{\mathrm{vir}}}&1  -
      \int_{[Q^{\epsilon_+}_{g,n}(X,\beta)]^{\mathrm{vir}}}1 \\
      & = \sum_{k\geq 1} \sum_{\vec\beta} \frac{1}{k!}
      \int_{[Q^{\epsilon_+}_{g,n+k}(X,\beta^\prime)]^{\mathrm{vir}}}
      \prod_{i=1}^k \mathrm{ev}_{n+i}^* \big[(z-\psi)I_{\beta_i}(z-\psi)\big]_0
    \end{aligned}
  \end{equation*}
where
$\vec\beta$ runs through all the $(k+1)$-tuples of effective curve classes
  \[
    \vec\beta = (\beta^\prime,\beta_1,\ldots, \beta_k)
  \]
  such that $\beta = \beta^\prime + \beta_1 + \cdots + \beta_k$ and
  $\deg(\beta_i) = d_0$ for all $i=1 ,\ldots, k$.
\end{theorem}
\ifdefined\SHOWPROOFS
\begin{proof}
  By the localization formula,
  the sum of residues at $z=0$ of the left hand side of
  \eqref{eq:equation-on-Q-Pn} is zero.
  Apply Corollary~\ref{lem:contr-F-beta-form2} and
  replace $k+r$ by $k$, i.e.\
  the markings $n+1 ,\ldots, n+k$ in Corollary~\ref{lem:contr-F-beta-form2} are
  re-labeled as $n+1 ,\ldots, n+k-r$, and
  the markings $n+k+1,\ldots, n+k+r$ in Corollary~\ref{lem:contr-F-beta-form2} are
  re-labeled as $n+k-r+1, \ldots, n+k$. Also revoke
  Lemma~\ref{lem:vir-cycle-comp-entangled-tails}.
  Thus we obtain
  \begin{align*}
    \int_{[Q_{g,n}^{\epsilon_-}(X,\beta)]^{\mathrm{vir}}}1 -
    & \int_{[Q_{g,n}^{\epsilon_+}(X,\beta)]^{\mathrm{vir}}}1 -
      \sum_{k\geq 1}
      \sum_{\vec\beta}
      \sum_{r=0}^{k-1}
      \sum_{\vec b}
      \frac{(-1)^r}{r!(k-r)!}
      \int_{[Q^{\epsilon_+}_{g,n+k}(X,\beta^\prime)]^{\mathrm{vir}}}\\
    & \prod_{i=1}^k \big[\mathrm{ev}_{n+i}^*((\mathbf rz-\psi)I_{\beta_i}(\mathbf r z-\psi))
      \big]_{b_i} = 0,
  \end{align*}
  where $\vec\beta = (\beta^\prime,\beta_1 ,\ldots, \beta_k)$ is as above
  and
  \[
    \vec b = (b_1 ,\ldots, b_k)
  \]
  runs through all $k$-tuples of integers such that $b_1 + \cdots  + b_k = 0$ and
  $b_{k-r+1} ,\ldots, b_k <0$.
  Using the symmetry of the last $k$ markings, we rewrite it as
  \begin{align*}
    & \int_{Q_{g,n}^{\epsilon_-}(X,\beta)}1  -  \int_{ Q_{g,n}^{\epsilon_+}(X,\beta)}1 =\\
    & \sum_{k\geq 1}
      \sum_{\vec\beta} \sum_{N\subsetneq \{1 ,\ldots, k\}}
      \sum_{\vec b}
      \frac{(-1)^{\# N}}{k!}
      \int_{[Q^{\epsilon_+}_{g,n+k}(X,\beta_0)]^{\mathrm{vir}}}
      \prod_{i=1}^k \big[
      \mathrm{ev}_{n+i}^*((\mathbf rz-\psi)I_{\beta_i}(\mathbf r z - \psi))
      \big]_{b_i},
  \end{align*}
  where $\vec\beta$ is as before and
  \[
    \vec b = (b_1 ,\ldots, b_k)
  \]
  runs through all $k$-tuples of integers such that $b_1 + \cdots  + b_k = 0$ and
  $b_{i} <0$ for each $i\in N$.
  For each fixed $k, \vec b$, it is easy to see that
  \[
    \sum_{N} (-1)^{\# N} = \begin{cases}
      1, &{\text{if } b_i\geq 0 \text{ for all } i=1 ,\ldots, k};\\
      0, &\text{otherwise},
    \end{cases}
  \]
  where the sum runs through all $N\subsetneq \{1 ,\ldots, k\}$ such that $b_i<0$ for
  each $i\in N$.
  Finally observe that $b_i\geq 0$ for all $i=1 ,\ldots, k$ implies that $b_i=0$
  for all $i$. Hence we obtain the desired formula. Note  the change of variable
  $\mathbf rz\mapsto z$ does not affect the
  degree-$0$ term.
\end{proof}
\fi
\subsection{The genus-$0$ case}
\label{sec:proof-genus-0}
\ifdefined\SHOWPROOFS
\begin{proof}[Proof of Theorem~\ref{thm:big-J-general}]
  By Corollary~\ref{cor:higher-genus-numerical}, we see that
  Theorem~\ref{thm:big-J-general} holds true modulo the constant-in-$t$ terms.
  Let $\epsilon_- < \epsilon_0 = \frac{1}{d_0} <\epsilon_+$ as before and let
  \[
    \mu^{\epsilon_0}(q,z) = \sum_{\deg(\beta)=1/\epsilon_0} \mu_\beta(z) q^\beta.
  \]
  It suffices to prove that
  \begin{equation}
    \label{eq:J-cross-1-wall}
    J^{\epsilon_+}(\mu^{\epsilon_0}(q,-z),q,z) = J^{\epsilon_-}(0,q,z).
  \end{equation}
  By definition, its right hand side equals to
  \[
    1 +  \sum_{0<\deg(\beta)\leq
      1/\epsilon_0} I_{\beta}(z) q^\beta
    +  \sum_{\deg(\beta)>1/\epsilon_0} q^\beta
    \sum_{p} T_p \langle {
      \frac{T^p}{z(z-\psi)}} \rangle_{0,1,\beta}^{X,\epsilon_-},
  \]
  where $\{T^p\}$ is basis for $H^*_{\mathrm{CR}}(X,\mathbb Q)$ and $\{T_p\}$ is
  its dual basis under the pairing \eqref{eq:pairing}.
  By Theorem~\ref{thm:Chow-version}, this is equal to
  \begin{align*}
    1 + & \sum_{0<\deg(\beta)\leq
    1/\epsilon_0} I_{\beta}(z) q^\beta   \\
 + &\sum_{\substack{(\beta\geq 0, k\geq 1) \text{~or~}
    \\ (\deg(\beta)>1/\epsilon_0, k=0)}}
    \frac{q^{\beta}}{k!}
    \sum_{p}T_p \langle {
      \frac{T^p}{z(z-\psi)}, \mu^{\epsilon_0}(q,-\psi) ,\ldots,
      \mu^{\epsilon_0}(q,-\psi)} \rangle_{0,1+k,\beta}^{X,\epsilon_+}.
  \end{align*}
  Comparing this to the left hand side of
  ~\eqref{eq:J-cross-1-wall}, we see that it suffices to prove
  \[
    \frac{\mu^{\epsilon_0}(q,z)}{z} + \sum_{\deg(\beta) = 1/\epsilon_0}q^\beta
    \sum_p T_p \langle {\frac{T^p}{z(z-\psi)}}
    \rangle_{0,1,\beta}^{X,\epsilon_+} = \sum_{\deg(\beta) = 1/\epsilon_0}
    I_{\beta}(z) q^\beta.
  \]
  This follows immediately from the definitions and Theorem~\ref{thm:genus-0-special-case}.
\end{proof}
\fi

\appendix
\addcontentsline{toc}{section}{Appendices}

\section{Intersection theory on inflated projective bundles}
Recall from Section~\ref{sec:inflated-proj-bundle} the definition of inflated
projective bundles.
\begin{lemma}
  \label{lem_inflated_projective_bundle}
  Let $X$ be any Deligne--Mumford stack, and $\Theta_1,\ldots,\Theta_r$ be line
  bundles on $X$. Let
  $p:\tilde{\mathbb P} \to X$ be the inflated projective bundle associated to
  $\Theta_1 ,\ldots, \Theta_r$. For $i=1,\cdots,r-1$, let
  $D_i$ be the $i$-th tautological divisor.
  Then for any $a\in A_*(X)$ and any integer $s\geq r$, we have
  \[
    p_*\Big( \big(
        \sum_{i=0}^{r-2} (r-i)[D_i] + \sum_{i=1}^r c_1(\Theta_i)
      \big)^{s-1}
      \cap p^*a \Big) =
    \sum_{j_1+\cdots+j_r = s-r} r^{s-1}\cdot c_1(\Theta_1^{j_1})\cdots c_1(\Theta_r^{j_r}) \cap a.
  \]
\end{lemma}
\ifdefined\SHOWPROOFS
\begin{proof}
  First assume that $X$ is a  smooth  variety. We will apply the localization formula.
  Consider the standard component-wise $\mathbb T = (\mathbb C^*)^r$-action on $\bigoplus_i\Theta_i$.
  This makes each $\Theta_i$ a $\mathbb T$-equivariant line bundle and induces
  an action of $\mathbb T$ on $\tilde{\mathbb P}$.

  The fixed-point loci $F_\sigma$ are indexed by bijections
  \[
    \sigma: \{1 ,\ldots, r\} \longrightarrow  \{\Theta_1 ,\ldots, \Theta_r\}
  \]
  in the following way.
  Let $P_{r-2}$ be the blowup of  $\mathbb
  P_X(\Theta_1\oplus \cdots\oplus \Theta_r)$ along all the
  \[
    V_{\Theta_i} = \mathbb P_X(0 \oplus \cdots 0 \oplus \Theta_i \oplus 0 \cdots
    \oplus 0),\quad i = 1 ,\ldots, r.
  \]
  Let $E_{\Theta_i}$ be the exceptional divisor over $V_{\Theta_i}$. Then
  $E_{\Theta_i}$ is canonically isomorphic to
  \[
    \mathbb P_X(\Theta_1\otimes \Theta_i^{\vee}\oplus \cdots
    \widehat{\Theta_i\otimes \Theta_i^\vee}\cdots\oplus
    \Theta_r\otimes \Theta_i^{\vee}).
  \]
  Let $\tilde E_{\Theta_i}\subset \tilde {\mathbb P}$ be the preimage of $E_{\Theta_i}$, then
  $\tilde E_{\Theta_i} \to X$ is equal to the inflated projective bundle
  associated to $\Theta_1\otimes \Theta_i^{\vee}, \ldots ,
  \widehat{\Theta_i\otimes \Theta_i^\vee}, \ldots,
  \Theta_r\otimes \Theta_i^{\vee}$ (c.f.\ Lemma~\ref{lem:str-D}).

 Note that every fixed-point component is contained in some
 $\tilde{E}_{\Theta_i}$. Thus we define $F_{\sigma}$ by induction on $r$.
 When $r=1$, define $F_{\sigma}$ to be $\tilde{\mathbb P}$, for the unique
$\sigma$; for $r>1$, define $F_{\sigma}$ to be the fixed-point component
$F_{\sigma^\prime}$
 in the inflated projective bundle $\tilde E_{\Theta_{\sigma(r)}}$, where
 \begin{equation*}
   \sigma^\prime: \{1 ,\ldots, r-1\} \longrightarrow
   \{\Theta_1\otimes \Theta_{\sigma(r)}^{\vee}, \ldots,
   \widehat{\Theta_{\sigma(r)}\otimes \Theta_{\sigma(r)}^\vee}, \ldots,
   \Theta_r\otimes \Theta_{\sigma(r)}^{\vee}\}
 \end{equation*}
 is defined by 
 \[
   \sigma^\prime(j) =
   \begin{cases}
   \sigma(j) \otimes  \Theta_{\sigma(r)}^\vee,  \cond{j<\sigma(r)};\\
   \sigma(j+1) \otimes  \Theta_{\sigma(r)}^\vee,  \cond{j\geq \sigma(r)}.
 \end{cases}
 \]

  Given $\sigma$, set $i = \sigma(r)$. We observe that the restriction of
  $\mathcal O_{\mathbb P(\Theta_1 \oplus
    \cdots \oplus \Theta_r)}(-1)$ to $E_{\Theta_i}$ is canonically isomorphic to
  (the pullback of)
  $\Theta_i$. The normal bundle of $\tilde E_{\Theta_i}$ in $\tilde {\mathbb
    P}$ is isomorphic to the normal bundle of $E_{\Theta_i}$ in ${\mathbb P}$,
  which is canonically isomorphic to $\mathcal O_{E_{\Theta_i}}(-1)$.
  For $0\leq i\leq r-3$, the restriction of $D_i$ to $\tilde{E}_{\Theta_i}$ is
  isomorphic to the $i$-th tautological divisor on $\tilde{E}_{\Theta_i}$.
  The restriction of $\mathcal O_{\tilde {\mathbb P}}(D_{r-2})$ to
  $\tilde{E}_{\Theta_i}$ is isomorphic to the normal bundle of
  $\tilde{E}_{\Theta_i}$ in $\tilde {\mathbb P}$, which is isomorphic to
  $\mathcal O_{E_{\Theta_i}}(-1)$.
  Hence the restriction of $\mathcal O_{E_{\Theta_i}}(-1)$ to $F_\sigma$ is
  canonically isomorphic to $\sigma(r-1)\otimes \sigma(r)^\vee$

  Apply the same reasoning to each of the inflated projective bundles $\tilde
  E_{\Theta_i}$. Inductively, we get
  \begin{itemize}
  \item
    the restriction of $\mathcal O_{\tilde{\mathbb P}}(D_i)$ to $F_\sigma$ is
    isomorphic to $\sigma(i+1)\otimes
    \sigma{(i+2)}^{\vee}$, for $i=0 ,\ldots, r-2$;
  \item
    the $K$-theory class of the normal bundle of $F_\sigma$ in $\tilde {\mathbb
      P}$ equals to
    \[
      [\sigma(1) \otimes \sigma(2)^\vee] + \cdots  + [\sigma(r-1)\otimes \sigma(r)^\vee].
    \]

  \end{itemize}
  Each $F_\sigma$ is isomorphic to $X$ via the projection.
  Let $a_i$ be the non-equivariant first Chern class of $\Theta_i$, let
  $\lambda_1 ,\ldots, \lambda_r$ be the equivalent parameters of $\mathbb T$.
  By the localization formula \cite{edidin1998localization, atiyah1984moment}, we have
  \begin{equation}
    \begin{aligned}
      &p_*\left( \left(
          \textstyle
          \sum_{i=0}^{r-2} (r-i)[D_i] + \sum_{i=1}^r c^{\mathbb T}_1(\Theta_i)
        \right)^{s-1}
        \cap p^*a \right) \\=&
      \sum_{ \sigma \in S_r}\frac{r^{s-1} (a_{\sigma(1)}+\lambda_1)^s \cap a}
      {(a_{\sigma(1)}+\lambda_{\sigma(1)}-a_{\sigma(2)}-\lambda_{\sigma(2)})
        \cdots
        (a_{\sigma(r-1)}+\lambda_{\sigma(r-1)} - a_{\sigma(r)} - \lambda_{\sigma(r)})}.
    \end{aligned}
  \end{equation}
  The desired formula then follows from the following elementary Lemma~\ref{lem:comb}.

  We now reduced the general case to the special case.
  We first slightly generalize the desired formula.
  For any $f: Y\to X$ and any $a\in A_*(Y)$, replacing $\Theta_i$ by
  $f^*\Theta_i$ and $\tilde{\mathbb P}$ by $\tilde{\mathbb P}\times_X Y$, both
  sides produce classes in $A_{*-(s-r)}(Y)$. It is easy to see that both sides
  of the desired formula, as operations on $a$, are operational Chow classes in
  $A^{s-r}(X)$ (c.f.\ \cite{fulton2013intersection}). Fix $N>\dim X$, let
  $V_i^*$ be the total space of $\Theta^{\oplus N}_i$ minus its zero section.
  Consider
  \[
    \pi: X^\prime := V_1^*\times_{X}\cdots \times_{X} V_r^* \longrightarrow   X.
  \]
  Then the flat pullback $\pi^*: A_*(X)\to A_{*+rN}(X^\prime)$ is an isomorphism. Hence
  it suffices to verify the desired formula on $X^\prime$ for the line bundles $\pi^*\Theta_i$.
  The $N$ tautological sections of $\pi^*\Theta_i$ give rise to a map
  \[
    \tau_i: X^\prime \longrightarrow  \mathbb P^{N-1},
  \]
  such that $\tau_i^* \mathcal O_{\mathbb P^{N-1}}(1)\cong \pi^*\Theta_i$. Hence
  it suffices to verify the desired formula as an identify of operational Chow
  classes on $(\mathbb P^{N-1})^{r}$. By the Poincar\'e duality between Chow
  groups and operational Chow groups for smooth varieties, it suffice to take
  $a$ to be the fundamental class of $(\mathbb P^{N-1})^{r}$ and varify the
  identity in $A_*((\mathbb P^{N-1})^{r})$.
  This reduces the general case to the special case $X= (\mathbb P^{N-1})^{r}$.
\end{proof}
\fi

\begin{lemma}
  \label{lem:comb}
  Let $S_r$ be the group of permutations of $\{1 ,\ldots, r\}$.
  For $s \geq r-1$, we have an equation in $\mathbb Q(x_1 ,\ldots, x_r)$
  \begin{equation}
    \label{eqn-rational-function}
    \sum_{\sigma\in S_r}
    \frac{x_{\sigma(1)}^{s-1}}{(x_{\sigma(1)}-x_{\sigma(2)})(x_{\sigma(2)}-x_{\sigma(3)})
      \cdots (x_{\sigma(r-1)}-x_{\sigma(r)})} = \sum_{\substack{j_1+ \cdots + j_r =
      s-r, \\ j_1 ,\ldots, j_r \geq 0}}x_1^{j_1}\cdots x_r^{j_r}.
  \end{equation}
\end{lemma}
\ifdefined\SHOWPROOFS
\begin{proof}
  We only need to prove the identity for generic $\mathbf x = (x_1 ,\ldots,
  x_r)\in \mathbb C^r$. We assume that the $x_i$'s are distinct.
  Fixing $\mathbf x$, we define a rational function in $t$
  \[
    f_{\mathbf x} (t) = t^{r-s-1} \frac{1}{(1-tx_1)\cdots (1-tx_r)}.
  \]
  Then the right hand side of (\ref{eqn-rational-function}) is equal to
  \[
    \operatorname{Res}_{t=0}f_{\mathbf x}(t).
  \]
  The function $f_{\mathbf x}(t)$  has  no residue at $t=\infty$, hence the sum of
   residues at all finite poles is equal to zero. We compute
  \[
    \operatorname{Res}_{t=1/x_i} f_{\mathbf x}(t)=  \frac{-x_i^{s-1}}{(x_i-x_1)\cdots
      \widehat{(x_i-x_i)} \cdots (x_i-x_r)}.
  \]
  It remains to show that for each $i = 1 ,\ldots, r$, 
  \begin{equation}
    \label{eqn-rational-function2}
    \begin{aligned}
      &\frac{1}{(x_i-x_1)\cdots
        \widehat{(x_i-x_i)} \cdots (x_i-x_r)}  \\
      =& \sum_{\sigma\in S_r,\sigma(1)=i}
      \frac{1}{(x_{\sigma(1)}-x_{\sigma(2)})(x_{\sigma(2)}-x_{\sigma(3)})
        \cdots (x_{\sigma({r-1})}-x_{\sigma(r)})}
    \end{aligned}
  \end{equation}
  To prove (\ref{eqn-rational-function2}), we view both sides as functions in
  the single variable $x_i$ and view other variables as distinct constants. We
  verify that all poles are simple poles of equal residues on both sides
  (\ref{eqn-rational-function2}). Indeed, for $j\neq i$ the residue at $x_j$ of
  the left hand side is
  \[
    \frac{1}{(x_j-x_1)\cdots \widehat{(x_j-x_i)}\cdots \widehat{(x_j-x_j)}\cdots (x_j-x_r)};
  \]
  while on the right hand side the residue at $x_j$  is
  \[
    \sum_{\sigma\in S_r,\sigma(1)=i,\sigma(2)=j}
    \frac{1}{(x_{\sigma(2)}-x_{\sigma(3)})(x_{\sigma(3)}-x_{\sigma(4)})
      \cdots (x_{\sigma({r-1})}-x_{\sigma(r)})}.
  \]
  They are equal by induction on $r$. Now (\ref{eqn-rational-function2}) follows
  from the fact that both sides converge to $0$ as $x_i\to \infty$.
\end{proof}
\fi

\bibliographystyle{/home/yang/Dropbox/Research/bibtex/amsalpha-abbrev.bst}
\bibliography{/home/yang/Dropbox/Research/bibtex/references.bib}
\end{document}